\documentclass[reqno,11pt]{amsart}

\usepackage{amscd}

\usepackage{amsmath,amsthm,bbm,amssymb,latexsym,amsfonts,graphicx,tikz-cd,enumerate,comment,mathabx}

\usepackage[hidelinks]{hyperref}

\usepackage{dsfont}
\usepackage{euscript}
\usepackage[myheadings]{fullpage}

\RequirePackage[left=1in,right=1in,top=1in,bottom=1in]{geometry}

\newcommand{\bB}{\mathbb{B}}

\newcommand{\bD}{\mathbb{D}}
\newcommand{\bE}{\mathbb{E}}
\newcommand{\bG}{\mathbb{G}}
\newcommand{\bL}{\mathbb{L}}
\newcommand{\bO}{\mathbb{O}}

\newcommand{\bS}{\mathbb{S}}
\newcommand{\bZ}{\mathbb{Z}}

\newcommand{\sA}{\EuScript{A}}
\newcommand{\sC}{\EuScript{C}}
\newcommand{\sD}{\EuScript{D}}
\newcommand{\sE}{\EuScript{E}}
\newcommand{\sF}{\EuScript{F}}
\newcommand{\sG}{\EuScript{G}}
\newcommand{\sH}{\EuScript{H}}
\newcommand{\sI}{\EuScript{I}}
\newcommand{\sK}{\EuScript{K}}
\newcommand{\sL}{\EuScript{L}}
\newcommand{\sM}{\EuScript{M}}

\newcommand{\sO}{\EuScript{O}}
\newcommand{\sP}{\EuScript{P}}

\newcommand{\fL}{\mathfrak{L}}

\newcommand{\on}[1]{\operatorname{#1}}
\newcommand{\swedge}{\hspace{-0.2ex}\raisebox{0.15ex}{\scalebox{0.75}{$\wedge$}}}

\DeclareMathOperator{\BB}{B}
\DeclareMathOperator{\CC}{C}
\DeclareMathOperator{\LL}{L}
\DeclareMathOperator{\NN}{N}
\DeclareMathOperator{\RR}{R}
\DeclareMathOperator{\TT}{T}
\DeclareMathOperator{\UU}{U}

\newcommand{\modfact}{\mod^{\operatorname{fact}}}
\newcommand{\comfact}{\operatorname{com} \to \operatorname{fact}}
\newcommand{\uHom}{\underline{\operatorname{Hom}}}
\newcommand{\uEnd}{\underline{\operatorname{End}}}

\DeclareMathOperator{\dR}{dR}
\DeclareMathOperator{\id}{id}
\DeclareMathOperator{\ind}{ind}
\DeclareMathOperator{\coind}{coind}

\DeclareMathOperator{\acc}{acc}
\DeclareMathOperator{\enh}{enh}
\DeclareMathOperator{\aULA}{aULA}
\DeclareMathOperator{\add}{add}
\DeclareMathOperator{\cc}{c}
\DeclareMathOperator{\co}{co}
\DeclareMathOperator{\Fun}{Fun}
\DeclareMathOperator{\Map}{Map}
\DeclareMathOperator{\Sect}{Sect}
\DeclareMathOperator{\Loc}{Loc}
\DeclareMathOperator{\cts}{cts}

\DeclareMathOperator{\ram}{-ram}
\DeclareMathOperator{\cofib}{cofib}
\DeclareMathOperator{\fib}{fib}
\DeclareMathOperator{\disj}{disj}
\DeclareMathOperator{\fil}{fil}

\DeclareMathOperator{\inv}{inv}
\DeclareMathOperator{\Sym}{Sym}
\DeclareMathOperator{\Hom}{Hom}

\DeclareMathOperator{\colim}{colim}

\DeclareMathOperator{\Vect}{Vect}
\DeclareMathOperator{\Vac}{Vac}
\DeclareMathOperator{\Spec}{Spec}

\newcommand{\fSet}{\mathsf{fSet}}
\newcommand{\Part}{\mathsf{Part}}
\DeclareMathOperator{\surj}{surj}
\DeclareMathOperator{\cl}{cl}
\DeclareMathOperator{\red}{red}

\DeclareMathOperator{\spec}{spec}
\DeclareMathOperator{\Av}{Av}
\DeclareMathOperator{\oblv}{oblv}
\newcommand{\DD}{\mathsf{D}}
\newcommand{\IndCoh}{\mathsf{IndCoh}}
\newcommand{\QCoh}{\mathsf{QCoh}}

\DeclareMathOperator{\coh}{coh}

\DeclareMathOperator{\End}{End}
\newcommand{\Ind}{\mathsf{Ind}}
\DeclareMathOperator{\op}{op}
\DeclareMathOperator{\res}{res}
\DeclareMathOperator{\cores}{cores}
\newcommand{\Sph}{\mathsf{Sph}}
\DeclareMathOperator{\Sat}{Sat}
\DeclareMathOperator{\Gr}{Gr}
\DeclareMathOperator{\LS}{LS}
\DeclareMathOperator{\gr}{gr}
\DeclareMathOperator{\ch}{ch}
\newcommand{\Rep}{\mathsf{Rep}}
\DeclareMathOperator{\Ran}{Ran}
\DeclareMathOperator{\fact}{fact}
\DeclareMathOperator{\Fact}{Fact}
\DeclareMathOperator{\act}{act}
\DeclareMathOperator{\coact}{coact}
\DeclareMathOperator{\ren}{ren}
\DeclareMathOperator{\fha}{fha}

\newcommand{\Whit}{\mathsf{Whit}}
\DeclareMathOperator{\triv}{triv}
\DeclareMathOperator{\coChev}{coChev}
\DeclareMathOperator{\Chev}{Chev}
\DeclareMathOperator{\coPrim}{coPrim}
\DeclareMathOperator{\Prim}{Prim}

\DeclareMathOperator{\com}{com}

\DeclareMathOperator{\unit}{unit}

\DeclareMathOperator{\mult}{mult}
\newcommand{\ShvCat}{\mathsf{ShvCat}}

\newcommand{\Cat}{\mathsf{Cat}}
\newcommand{\Grp}{\mathsf{Grp}}

\newcommand{\FactCat}{\mathsf{FactCat}}
\newcommand{\FactAlg}{\mathsf{FactAlg}}

\newcommand{\DGCat}{\mathsf{DGCat}}
\newcommand{\AssocAlg}{\mathsf{AssocAlg}}
\newcommand{\ComAlg}{\mathsf{ComAlg}}
\newcommand{\ComHopfAlg}{\mathsf{ComHopfAlg}}
\newcommand{\LieCoalg}{\mathsf{LieCoalg}}
\newcommand{\LieAlg}{\mathsf{LieAlg}}
\newcommand{\AssocCoalg}{\mathsf{AssocCoalg}}
\newcommand{\CocomCoalg}{\mathsf{CocomCoalg}}
\DeclareMathOperator{\laxfact}{lax-fact}
\DeclareMathOperator{\laxuntl}{lax-untl}
\DeclareMathOperator{\nonuntl}{non-untl}
\DeclareMathOperator{\untl}{untl}
\DeclareMathOperator{\OblvUnit}{OblvUnit}
\DeclareMathOperator{\AddUnit}{AddUnit}
\DeclareMathOperator{\tstr}{t-str}

\DeclareMathOperator{\temp}{temp}
\DeclareMathOperator{\antitemp}{anti-temp}

\DeclareMathOperator{\td}{td}
\DeclareMathOperator{\nv}{nv}
\DeclareMathOperator{\strict}{strict}

\newcommand{\e}{\mathds{1}}

\newcommand{\mathendash}{\text{\textendash}}

\renewcommand{\mod}{\mathendash\mathsf{mod}}
\newcommand{\comod}{\mathendash\mathsf{comod}}

\numberwithin{equation}{subsection}

\newtheorem{theorem}{Theorem}[subsection]
\newtheorem{lemma}[theorem]{Lemma}

\newtheorem{corollary}{Corollary}[theorem]
\newtheorem{proposition}[theorem]{Proposition}

\theoremstyle{remark}
\newtheorem{rem}[theorem]{Remark}
\newtheorem{example}[theorem]{Example}
\newtheorem{definition}[theorem]{Definition}

\date{\today}

\begin{document}

\title{Langlands duality on the Beilinson-Drinfeld Grassmannian}

\author{Justin Campbell and Sam Raskin}

\maketitle

\begin{abstract}

We calculate various categories of
equivariant sheaves on the Beilinson-Drinfeld Grassmannian
in Langlands dual terms. For one, we obtain the 
factorizable derived geometric Satake theorem. More generally, 
we calculate the categorical analogue of 
unramified vectors in the Jacquet module of sheaves on the Grassmannian.

In all cases, our spectral categories involve factorization modules
for factorization algebras related to the Langlands dual group.

\end{abstract}

\setcounter{tocdepth}{1}
\tableofcontents

\section{Introduction}\label{s:intro}

\subsection{What is this paper about?}

Our main objectives in this paper are as follows:

\begin{itemize}

\item Realize a strategy of Gaitsgory-Lurie giving 
a natural construction of the derived geometric Satake equivalence
\cite{bezrukavnikov-finkelberg} (cf. the first footnote in 
\emph{loc. cit}.).

\item Extend derived geometric Satake from the usual affine Grassmannian
to the Beilinson-Drinfeld (or factorizable) affine Grassmannian. 

\item Extend the Arkhipov-Bezrukavnikov-Ginzburg equivalence
\cite{ABG} to the Beilinson-Drinfeld affine Grassmannian.

\end{itemize}

In fact, we unify the latter two subjects with a mutual generalization to parabolic
subgroups $P$ of $G$; for $P = G$ we obtain factorizable derived Satake,
while for $P = B$ we obtain factorizable Arkhipov-Bezrukavnikov-Ginzburg.

These results have been widely anticipated. In particular, Gaitsgory 
has spoken and written about our main results for some time, 
cf. \cite{dennis-laumonconf} \S 4.7 and \S 6.6.3, 
\cite{whatacts} \S 10.3.3,
\cite{jerusalem2014} Talk V.2 Conjecture 1, 
\cite{paris-notes} Program \S 4.1.1, and \cite{arinkin-gaitsgory} 
Remark 12.6.7 to name a few. 

However, the questions we consider have previously resisted
precise formulations. In the above sources, one finds 
no definitions of the spectral (or \emph{Langlands dual}) sides
of the equivalences under consideration. So although our results
have been anticipated, we can find no precisely formulated 
conjecture in the literature that 
we can point to and say ``We proved \emph{this} conjecture!" 
In large part, our task in this paper is to precisely formulate
the problem and to do so in a way that is amenable to its solution. 

Accordingly, below we will give an extended discussion 
of our point of view on the spectral side.

\subsection{Applications} 

Derived Satake at a point is closely related to subtle issues in the
global geometric Langlands program; see the above references as 
well as \cite{independence} and \cite{coeffs}.
Beraldo \cite{dario-top-ch} has given a lovely interpretation of the
meaning of full factorizable Satake, for which we refer to \emph{loc. cit}.

For a long time (cf. \cite{whatacts} and \cite{dennis-laumonconf}),
it has been anticipated that our results in the parabolic case are required
to understand global geometric Langlands near reducible local systems.  
Forthcoming work \cite{eis} uses our results
in completely settling exactly this problem. 

%%% More??

\subsection{The Gaitsgory-Lurie heuristic (spherical case)}

We now recall the idea of Gaitsgory and Lurie, in part because it was 
never written down (as far as we are aware).
For purposes of an introduction, we give ourselves the liberty
to be lax with our homological algebra and notation around 
factorization structures.

We let $\Gr_G$ denote the affine Grassmannian, which is the quotient
$\mathfrak{L}G/\mathfrak{L}^+G$ of the algebraic loop group of $G$ by the 
group of jets into $G$. Beilinson-Drinfeld viewed $\Gr_G$ as a factorization space
fibered over the Ran space of finite subsets of the curve $X$. The geometric side of derived Satake is the
category $\Sph_G := \DD(\Gr_G)^{\mathfrak{L}^+ G}$ of spherical D-modules on the 
affine Grassmannian\footnote{In the main body of this text, the symbol $\Sph_G$ is used for a slightly modified version of $\DD(\Gr_G)^{\mathfrak{L}^+ G}$, which we will not need in this introduction.}. 

Gaitsgory-Lurie observed that $\Sph_G$ is a \emph{factorization monoidal category} 
because of the Beilinson-Drinfeld factorization structure on the affine Grassmannian, and as such,
it acts on the factorization category $\DD(\Gr_G)^{\mathfrak{L}N^-,\psi}$
of Whittaker D-modules on the affine Grassmannian.

By a version of the geometric Casselman-Shalika formula of 
\cite{fgv}, it is known that 
\begin{equation}
\label{csequiv}
\DD(\Gr_G)^{\mathfrak{L}N^-,\psi} \simeq \Rep(\check{G})
\end{equation}
as factorization categories 
(see Theorem \ref{csformula} below).

We use the following standard heuristics: 
factorization categories are analogous to $\bE_2$-categories 
(also known as braided monoidal categories), while factorization monoidal
categories are analogous to $\bE_3$-categories. These heuristics
are made somewhat precise in the literature, see 
\cite{higheralgebra} \S 5.5 and notes from the 
talk by Nick Rozenblyum in \cite{paris-notes}. See also 
\cite{nocera}, where the $\bE_3$-structure on $\Sph_G$ is made 
completely rigorous.

Now recall that for an $\bE_n$-category $\sC$, there is
a certain category $Z_{\bE_n}(\sC)$ called its
\emph{$\bE_n$-center}, which is the universal 
$\bE_{n+1}$-category acting on $\sC$ as an $\bE_n$-category
in a suitable sense -- we refer to \cite{higheralgebra} \S 5.3 for definitions
and details. Explicitly, as a monoidal category, 
one can calculate $Z_{\bE_n}(\sC)$ as 
$\mathsf{End}_{\sC\mod_{\bE_n}}(\sC)$, the category of endofunctors
of $\sC$ as an $\bE_n$-module category over itself. 

Following this heuristic, we imagine that $\Sph_G$ is an $\bE_3$-category (which we know it literally is, cf. \cite{nocera})
acting on the $\bE_2$-category $\Rep(\check{G})$ via the equivalence (\ref{csequiv}).
Therefore, heuristically, we obtain a canonical functor of $\bE_3$-categories
\begin{equation}\label{eq:gaitsgory-lurie}
\Sph_G \longrightarrow Z_{\bE_2}(\Rep(\check{G})).
\end{equation}

On the other hand, it is well-known that:
\[
Z_{\bE_2}(\Rep(\check{G})) = Z_{\bE_2}(\QCoh(\bB \check{G})) = 
\QCoh((\bB \check{G})^{\bS^2})
\]

\noindent (cf. \cite{bzfn} Corollary 5.12). Here the last term
$(\bB \check{G})^{\bS^2}$ is the derived stack parametrizing maps
from the homotopy type of the $2$-sphere to the stack $\bB \check{G}$.
Writing $\bS^2$ as the suspension of $\bS^1$, this stack may be calculated
more explicitly as
\[
(\bB \check{G})^{\bS^2} = 
\bB \check{G} \underset{(\bB \check{G})^{\bS^1}}{\times} \bB \check{G} = 
\bB \check{G} 
\underset{(\check{G}\overset{\on{ad}}{/} \check{G})}{\times}
\bB \check{G}
= 
(\Omega_1 \check{G})/\check{G}
\]

\noindent where \[ \Omega_1 \check{G} := \Spec(k) \underset{\check{G}}{\times} \Spec(k) \]
is the DG group of automorphisms of the identity in $\check{G}$.
Using logarithms, one finds that \[ \Omega_1 \check{G} = \Omega_0 \check{\mathfrak{g}}
= \Spec(\Sym(\check{\mathfrak{g}}[1])), \] so we can interpret the functor
\eqref{eq:gaitsgory-lurie} as a functor
\[
\Sph_G \longrightarrow 
\Sym(\check{\mathfrak{g}}[1]))\mod(\Rep(\check{G})).
\]

The derived Satake theorem then says that this functor is
\emph{almost} an equivalence. More precisely, the derived Satake theorem of
\cite{bezrukavnikov-finkelberg} says that this is true \emph{up to 
renormalization} -- we should use the \emph{renormalized Satake category}
and the category $\IndCoh((\bB \check{G})^{\bS^2})$.

To summarize, their heuristic says to follow the recipe:
\begin{enumerate}

\item Use the factorization action of $\Sph_G$ on $\Rep(\check{G})$
to obtain a functor to a suitable analogue of the $\bE_2$-center
of $\Rep(\check{G})$.

\item Settle the homological algebra issues to obtain a functor that
is plausibly an equivalence.

\item Perform an explicit calculation to prove that the functor so
obtained is an equivalence.

\end{enumerate}

\begin{rem}

In \cite{bezrukavnikov-finkelberg}, one finds that many of the difficulties
have to do with constructing the functor. The Gaitsgory-Lurie idea is
of a markedly different nature than what Bezrukavnikov-Finkelberg considered.

\end{rem}

\subsection{Centers and $\bE_n$-modules}

We will reinterpret centers (and more generally: centralizers)
in terms of modules over rings (or near enough).
This material is not necessary, but we find it plays
an important motivating role for our main results.

Suppose that $\sC$ and $\sD$ are symmetric monoidal categories.

Given an $\bE_n$-monoidal functor $F:\sC \to \sD$, its \emph{$\bE_n$-centralizer} $Z_{\bE_n}(F)$
is the category 
\[ \mathsf{Hom}_{\sC\mod_{\bE_n}}(\sC,\sD), \] i.e., the 
$\bE_n$-Hochschild homology of $\sD$ as an $\bE_n$-module 
category for $\sC$. Its key property is that an $\bE_n$-map
$\sE \to Z_{\bE_n}(F)$ is equivalent to an $\bE_n$-map
$\sC \otimes \sE \to \sD$ whose composition
\[ \sC \xrightarrow{\id_{\sC} \otimes \e_{\sE}} \sC \otimes \sE 
\longrightarrow \sD \] is $F$ as an $\bE_n$-functor.
(We recover the $\bE_n$-center of $\sC$
by taking the centralizer of $\id_{\sC}$.)

We now give an alternative to
$\bE_n$-centralizers (that coincides in some cases). 
First, one 
may form $\mathsf{Hom}(\sC,\sD)$, the category of continuous
DG functors from $\sC$ to $\sD$. 
Using Day convolution, this category is naturally symmetric monoidal as well.

We remind the reader that $\bE_n$-algebras in 
$\mathsf{Hom}(\sC,\sD)$ are the same as (left) lax $\bE_n$-monoidal
(continuous DG) functors from $\sC$ to $\sD$.
A variant: given $F:\sC \to \sD$ a lax $\bE_n$-monoidal
functor, thought of as an $\bE_n$-algebra in 
$\mathsf{Hom}(\sC,\sD)$, an $\bE_n$-module for
$F$ in $\mathsf{Hom}(\sC,\sD)$ is a functor
$G:\sC \to \sD$ of lax $\bE_n$-module categories
for $\sC$. (In fact, this variant can be put on equal footing
with its predecessor using general colored operads.)

Now suppose that $\sC$ is \emph{rigid} (symmetric) monoidal. 
Then recall that a \emph{lax} linear morphism of
$\sC$-module categories is automatically linear --
this is standard for left module categories, but holds
just the same for $\bE_n$-module categories. Therefore,
in the above discussion, we have
\[
Z_{\bE_n}(F) := 
\mathsf{Hom}_{\sC\mod_{\bE_n}}(\sC,\sD) =
\mathsf{Hom}_{\sC\mod_{\bE_n}^{\on{lax}}}(\sC,\sD) 
\]

\noindent where $\sC\mod_{\bE_n}^{\on{lax}}$ is 
the category of \emph{lax} $\bE_n$-module categories
for $\sC$. In particular, we see that when $\sC$ is
rigid, we can extend the definition of 
$\bE_n$-centralizer to allow $F$ to be 
merely lax $\bE_n$-monoidal. Moreover, using the preceding
paragraph, we can think of $Z_{\bE_n}(F)$ as 
$\bE_n$-modules for $F$ in the symmetric monoidal category
$\mathsf{Hom}(\sC,\sD)$.

We now make one additional manipulation, continuing to 
assume $\sC$ is rigid. Then recall that $\sC$ is canonically
self-dual as a DG category. In particular, we obtain
\[ \mathsf{Hom}(\sC,\sD) \simeq \sC^{\vee} \otimes \sD \simeq 
\sC \otimes \sD \]
\noindent where one easily finds\footnote{For completeness,
we supply the argument. Giving a lax symmetric monoidal functor
$\sE \to \mathsf{Hom}(\sC,\sD)$ is the same as giving a lax
symmetric monoidal functor $\sC \otimes \sE \to \sD$. For
$\sE = \sC \otimes \sD$, we have the lax symmetric monoidal
functor
$\sC \otimes \sE = \sC \otimes \sC \otimes \sD \xrightarrow{m \otimes \on{id}} \sC \otimes \sD \xrightarrow{\underline{\on{Hom}}_{\sC}(\e_{\sC},-) \otimes \on{id}} \sD$
(the key point being that $\e$ is compact and 
$\underline{\on{Hom}}_{\sC}(\e,-):\sC \to \Vect$ 
is lax symmetric monoidal). This gives rise to the lax symmetric
monoidal functor $\sC \otimes \sD \to \mathsf{Hom}(\sC,\sD)$.
By definition, this is the composite equivalence from before. 
One can check that this lax symmetric monoidal equivalence is
actually symmetric monoidal.
}
that the composite
equivalence is symmetric monoidal. 

To summarize, we obtain the following under the
rigidity assumption on $\sC$. Below, for $F:\sC \to \sD$
(possibly lax) $\bE_n$-monoidal, we let $\sK_F \in \sC \otimes \sD$
be the corresponding object.

\begin{enumerate}

\item The $\bE_n$-center of $\sC$ is the category of $\bE_n$-modules
for $\sK_{\id_{\sC}} \in \bE_n\mathendash{\mathsf{alg}}(\sC \otimes \sD)$.

\item More generally, for $F:\sC \to \sD$ lax $\bE_n$-linear, 
the $\bE_n$-centralizer of $F$ is the category of $\bE_n$-modules
for $\sK_F \in \bE_n\mathendash{\mathsf{alg}}(\sC \otimes \sD)$.

\end{enumerate}

In particular, we have expressed $\bE_n$-centers and centralizers
in terms of $\bE_n$-modules.

\begin{example}\label{ex:sat-mot}

For $\sC = \Rep(\check{G})$, we find that $Z_{\bE_2}(\Rep(\check{G}))$
is the category of $\bE_2$-modules for the regular representation
$\sO_{\check{G}} \in \ComAlg(\Rep(\check{G} \times \check{G}))$.

\end{example}

In \S \ref{ss:intro-parabolic} below, we will extend this discussion
to the setting where there is a parabolic subgroup; we remark
that there, we actually do consider $\bE_2$-modules 
for lax monoidal functors and use centralizers rather than centers.

\begin{rem}

A technical remark: in \cite{R2} \S 6, it was observed that
factorization categories attached to \emph{rigid} symmetric monoidal
categories have much more favorable properties than general
factorization categories. We consider the use of rigidity
in the above discussion to be related to this observation, but 
here working purely in the $\bE_2$-setting.

\end{rem}

\subsection{Chiral setting}\label{ss:intro-chiral}

Following standard analogies, we consider chiral (or factorization)
algebras as de Rham analogues of $\bE_2$-algebras (which are Betti
objects), and similarly for $\bE_2$-modules. 

Recall from \cite{R2} that any symmetric monoidal category naturally gives rise to a factorization
category (analogous to thinking of an $\bE_{\infty}$-category as an $\bE_2$-category), and
any commutative algebra in a symmetric monoidal category gives rise to a factorization algebra
in the associated factorization category (analogous to thinking of an $\bE_{\infty}$-algebra as an $\bE_2$-algebra). 
 
We can then form the category of factorization modules
\[
	\Sph^{\spec,\nv} := 
\sO_{\check{G}}\modfact(\Rep(\check{G} \times \check{G}))
\]

\noindent (the abbreviation $\nv$ stands for ``naive"). In \S \ref{s:renorm} and \S \ref{s:spec-hecke}, we
explain how to perform
homological algebra corrections to $\Sph^{\spec,\nv}$
(analogous to replacing $\QCoh$ by $\IndCoh$) to 
obtain a factorization category $\Sph_{\check{G}}^{\spec}$.

Following Gaitsgory, we refer in this paper to this 
process of correcting homological defects as \emph{renormalization}.
The specific technical issues we face here are novel 
and significant. Their resolution
occupies a large portion of this work. 

In \S \ref{s:functor}, we explain how the Gaitsgory-Lurie paradigm
plays out in the chiral setting. Roughly speaking, their strategy
for constructing the derived Satake functor 
goes through without hiccups in our 
setting. The factorizable derived Satake equivalence itself is stated as 
Theorem \ref{mainthm}.

Our proof that the functor is an equivalence takes up \S \ref{s:mainthmproof}.
We do this by reducing to the case that $G = T$ is a torus; in \S \ref{s:torus} we verify this case directly.
Our process of reduction to the torus uses the study of Jacquet functors from \cite{R2}-\cite{R3}.

\subsection{The Gaitsgory-Lurie heuristic (parabolic case)}\label{ss:intro-parabolic}

As far as we are aware, Gaitsgory and Lurie did not observe
that a similar strategy also applies in the presence of a parabolic
$P$ of $G$.

Let $N_P$ denote the unipotent radical of $P$ and let $M = P/N_P$. On
the geometric side, we consider the
factorization category \[ \DD(\Gr_G)_{\fL N_P \fL^+M} \]
where on the right hand side, we have formed the \emph{coinvariant 
category}, cf. \cite{R3}. By \cite{lin-nearby}, we note
that this category can also be thought of as 
\[ \DD(\Gr_G)^{\fL N_{P^-} \fL^+M} \] i.e., \emph{invariants} for
the \emph{opposite} parabolic.

On the spectral side, there are \emph{two} natural candidates
to consider, attached to two lax symmetric monoidal
functors
\[
\Chev_{\Omega},\Chev_{\Upsilon}:\Rep(\check{G}) \longrightarrow \Rep(\check{M}).
\] 

\noindent Both functors come by first restricting a
representation from $\check{G}$ to $\check{P}$ and then applying
a functor $\Rep(\check{P}) \to \Rep(\check{M})$.

For $\Chev_{\Omega}$, the relevant 
functor $\Rep(\check{P}) \to \Rep(\check{M})$
is (derived) invariants for $\check{N}_{\check{P}}$.

The functor $\Chev_{\Upsilon}$ is more complicated to define.
Consider $\check{\mathfrak{n}}_{\check{P}}$ as a Lie algebra
in $\Rep(\check{M})$. We consider 
\[ Z_{\Rep(\check{M}),\bE_1}(\check{\mathfrak{n}}_{\check{P}}\mod(\Rep(\check{M})), \] i.e., the $\bE_1$-center (alias: Drinfeld 
center)
in the sense of $\Rep(\check{M})$-linear monoidal categories.
Because $\check{\mathfrak{n}}_{\check{P}}\mod(\Rep(\check{M}))$
is itself symmetric monoidal, there is a natural functor
\[ \check{\mathfrak{n}}_{\check{P}}\mod(\Rep(\check{M}))
\longrightarrow Z_{\Rep(\check{M}),\bE_1}(\check{\mathfrak{n}}_{\check{P}}\mod(\Rep(\check{M})). \]
We remark that one may identify 
$Z_{\Rep(\check{M}),\bE_1}(\check{\mathfrak{n}}_{\check{P}}\mod(\Rep(\check{M}))$ more explicitly with 
\[ (\check{\mathfrak{n}}_{\check{P}} \otimes \on{C}^{\cdot}(\bS^1))\mod(\Rep(\check{M})) = 
U_{\bE_2}(\check{\mathfrak{n}}_{\check{P}})\mod_{\bE_2}(\Rep(\check{M})). \]
(We highlight these later expressions are closer to those used in 
\cite{R2}-\cite{R3} and in the body of this paper, although they are 
perhaps less conceptual.)
Then $\Chev_{\Upsilon}$ comes by composing the forgetful functor
$\Rep(\check{G}) \to \Rep(\check{P})$ with the composite functor 
\[
\Rep(\check{P}) \subset \check{\mathfrak{n}}_{\check{P}}\mod(\Rep(\check{M}))
\longrightarrow
Z_{\Rep(\check{M}),\bE_1}(\check{\mathfrak{n}}_{\check{P}}\mod(\Rep(\check{M}))
\xrightarrow{\oblv} \check{\mathfrak{n}}_{\check{P}}\mod(\Rep(\check{M}))
\xrightarrow{\oblv}
\Rep(\check{M}).
\]

We expect the $\bE_2$-centralizers
$Z_{\bE_2}(\Chev_{\Omega})$ and $Z_{\bE_2}(\Chev_{\Upsilon})$
of these functors to be ``approximately" (i.e., up to issues 
of completion) equivalent 
by a Koszul duality procedure. 
%More precisely, 
%$Z_{\bE_2}(\Chev_{\Omega})$ is the
%$\Rep(\check{M})$-linear $\bE_1$-center of $\Rep(\check{P})$, while
%$Z_{\bE_2}(\Chev_{\Upsilon})$ is the $\Rep(\check{M})$-linear
%$\bE_1$-center of $\check{\mathfrak{n}}_{\check{P}}\mod(\Rep(\check{M}))$ (as appeared above).

Although $\Chev_{\Omega}$ is easier to define and arguably 
is the better object, we use (the de Rham analogue of) 
$\Chev_{\Upsilon}$. There are several reasons for this, which
we take a moment to record:

\begin{itemize}

\item One should think that $\Chev_{\Upsilon}$ is adapted to
$\fL N_P \fL^+M$-coinvariants and $\Chev_{\Omega}$ is adapted to invariants.
One may use either invariants or coinvariants due to \cite{lin-nearby},
so we are similarly free to use either $\Chev_{\Upsilon}$ or
$\Chev_{\Omega}$.

\item The paper \cite{R3} uses $\Chev_{\Upsilon}$, which 
makes our references somewhat easier. 

However, we remark that in the \emph{quantum} 
setting, more recent papers in the subject, such as 
\cite{chen-fu}, \cite{ruotao-flags}, and \cite{gl-small}
prefer $\Chev_{\Omega}$. 

\item Related to the previous point, the paper \cite{R3} uses certain technical features of
$\Chev_{\Upsilon}$ related to renormalization and t-structures, whose analogues have not been documented for $\Chev_{\Omega}$.
The upcoming work \cite{eis} will treat the case of $\Chev_{\Omega}$.

\end{itemize}

With those remarks out of the way, we note that there is 
a pairing
\[
\DD(\Gr_G)_{\mathfrak{L}N_P\mathfrak{L}^+M} 
\otimes \DD(\Gr_G)^{\fL N^-,\psi} \longrightarrow 
(\DD(\mathfrak{L}G)_{\mathfrak{L}N_P\mathfrak{L}^+M})^{\mathfrak{L}N^-,\psi}
\]

\noindent by convolution. One can further $!$-restrict to:
\[
(\DD(\mathfrak{L}(P \times N_M^-)_{\mathfrak{L}N_P\mathfrak{L}^+M})^{\mathfrak{L}N^-,\psi} = 
\DD(\Gr_M)^{\fL N_M^-,\psi}
\]

\noindent for $N_M^- := N^- \cap M$. Applying the equivalence (\ref{csequiv}) for
$G$ and $M$, we can rewrite the resulting functor as a pairing
\[
\DD(\Gr_G)_{\mathfrak{L}N_P\mathfrak{L}^+M} 
\otimes \Rep(\check{G}) \longrightarrow 
\Rep(\check{M}).
\]

\noindent By \cite{R3} Theorem 4.15.1, the induced functor
\[
\Rep(\check{G}) \longrightarrow \Rep(\check{M})
\]

\noindent obtained by pairing with the unit object 
in $\DD(\Gr_G)_{\mathfrak{L}N_P\mathfrak{L}^+M}$ is the functor
$\Chev_{\Upsilon}$ (or rather, its de Rham counterpart). 
If we imagined we were in the $\bE_2$-setting, this would
yield a functor
\[
\DD(\Gr_G)_{\mathfrak{L}N_P\mathfrak{L}^+M} \longrightarrow 
Z(\Chev_{\Upsilon})
\]

\noindent that ought to be an equivalence up to issues of 
renormalization.

Using the essentially the same discussion as in 
\S \ref{ss:intro-chiral}, we formulate and prove a purely
de Rham version of
exactly this assertion, Theorem \ref{mainthmps}. As we discuss in the paper, 
on underlying categories (i.e., forgetting factorization structures), 
we recover (by different means) the main theorem of \cite{ABG}. 

\begin{rem}

In truth, the references to \cite{R3} above are for
$P = B$. This is ultimately due to a choice in \cite{bg-deformations}
to work only with the Borel. Work in progress by
F\ae{}rgeman-Hayash \cite{FH} extends \cite{bg-deformations} to the
case of general parabolics. As such, our results for general parabolics
(besides $B$ and $G$) are conditional on their work.

\end{rem}

\subsection{Notation and conventions}

We work over a field $k$ of characteristic zero. We also assume that $k$ is algebraically closed, which is not strictly necessary but simplifies some notations.

Denote by $X$ a fixed smooth and connected (but not necessarily proper) curve over $k$.

We fix a reductive group $G$ over $k$ and a Borel subgroup $B \subset G$. Let $N \subset B$ denote the unipotent radical of $B$ and $T = B/N$ the Cartan. Choose a splitting $T \to B$, so that we can consider the opposite Borel $B^-$, with unipotent radical $N^-$. We will denote by $P$ a standard parabolic, with unipotent radical $N_P$, Levi subgroup $M$, opposite parabolic $P^-$, etc. The coweight lattice will be denoted by $\Lambda = \Hom(\bG_m,T)$.

The Langlands dual group of $G$ (over $k$) will be denoted by $\check{G}$, with corresponding standard parabolics $\check{P}$, etc.

We freely use the language of higher category theory as developed in \cite{htt} and \cite{higheralgebra}, and usefully summarized in \cite{gr-i} Chapter 1. We will use the term \emph{category} to mean $(\infty,1)$\emph{-category}, and \emph{DG category} to mean \emph{presentable stable} $k$\emph{-linear category}. DG categories assemble into a category $\DGCat$, which is symmetric monoidal with respect to the Lurie tensor product. The unit for this tensor product is the DG category $\Vect$ of (complexes of) $k$-vector spaces.

We work in the setting of derived algebraic geometry, i.e., our test objects are understood to be affine DG schemes over $k$. For any prestack $Y$ over $k$, we denote by $\QCoh(Y)$ the DG category of quasicoherent sheaves on $Y$. If $Y$ is locally almost of finite type, we can also consider the DG category $\IndCoh(Y)$ of ind-coherent sheaves, whose theory is developed in \cite{indcoh} and \cite{gr-i}. We will denote the DG category of D-modules on such a prestack by $\DD(Y)$. By definition, we have \[ \DD(Y) = \IndCoh(Y_{\dR}), \] where $Y_{\dR}$ is the \emph{de Rham prestack} of $Y$ defined by \[ \Map(S,Y) = \Map(S_{\cl,\red},Y) \] (here $S_{\cl,\red}$ denotes the maximal reduced classical subscheme of the test scheme $S$).

We will make use of the theory of D-modules on certain infinite-dimensional spaces such as loop groups, which has been developed in \cite{infdim} and \cite{B}. We use the $\DD^*$ version of this theory by default.

Finally, for a (possibly lax) prestack $Y$, we will denote by $\ShvCat(Y)$ the category of sheaves of categories on $Y$. We refer the reader to \cite{G3} for more details on this notion, or \cite{R1} for the lax setting.

\subsection{Acknowledgements} 
 
As the above introduction should make clear, this paper develops 
an old idea of Dennis Gaitsgory and Jacob Lurie. 

We warmly thank Dennis Gaitsgory for the boundless generosity with which he
shares his ideas.
We learned of his work with Lurie through conversations with him, but also
learned the major technical tools used in this paper from him. 

In addition, we thank Dima Arinkin, David Ben-Zvi, 
Dario Beraldo, Lin Chen, Gurbir Dhillon, Kevin Lin, Sergey Lysenko, and Nick Rozenblyum
for enlightening conversations related to this material.

S.R. was supported by NSF grants DMS-2101984 and DMS-2401526 and a Sloan Research Fellowship.

\section{Renormalizing crystals of categories}\label{s:renorm}

Below $S$ will denote a smooth quasicompact scheme over $k$ unless otherwise specified. This section lays the technical foundations for renormalization in the presence of a crystal structure over $S$, i.e., a $\DD(S)$-module structure. Proposition \ref{rencrys} gives conditions under which the canonical renormalization of a $\DD(S)$-module category with compatible t-structure admits a natural crystal structure. Definition \ref{d:aula} formulates what it means  for a crystal of categories with compatible t-structure to be ``almost ULA generated,"  and Proposition \ref{ularen} says that the canonical renormalization of such a category is ULA generated.

\subsection{Induced crystals} The forgetful functor \[ \oblv_S : \DD(S) \longrightarrow \IndCoh(S) \] is symmetric monoidal, and hence can be viewed as a morphism of $\DD(S)$-modules. Its left adjoint \[ \ind_S : \IndCoh(S) \longrightarrow \DD(S) \] therefore admits an oplax $\DD(S)$-linear structure, which is easily seen to be strict. Given a $\DD(S)$-module category $\sC$, we therefore obtain adjoint functors
\begin{equation}
\ind_{\sC} : \IndCoh(S) \underset{\DD(S)}{\otimes} \sC \rightleftarrows \sC : \oblv_{\sC}.
\label{indadj}
\end{equation}
Moreover, for any $\DD(S)$-linear functor $F : \sC \to \sD$ we have commutative squares \[
\begin{tikzcd}
\IndCoh(S) \otimes_{\DD(S)} \sC \arrow{r}{\id \otimes F} & \IndCoh(S) \otimes_{\DD(S)} \sD  \\
\sC \arrow{r}{F} \arrow{u}{\oblv_{\sC}} & \sD \arrow{u}{\oblv_{\sD}}
\end{tikzcd} \] and \[
\begin{tikzcd}
\IndCoh(S) \otimes_{\DD(S)} \sC \arrow{r}{\id \otimes F} \arrow{d}{\ind_{\sC}} & \IndCoh(S) \otimes_{\DD(S)} \sD \arrow{d}{\ind_{\sD}} \\
\sC \arrow{r}{F} & \sD.
\end{tikzcd} \]

\begin{proposition}

\label{comon}

The adjunction (\ref{indadj}) is both monadic and comonadic.

\end{proposition}

\begin{proof}

For the monadicity it suffices to show that $\oblv_{\sC}$ is conservative, which is Lemma B.2.1 in \cite{R2}.

We claim that the monad $\oblv_{\sC} \circ \ind_{\sC}$ is \emph{effective} in the sense of \cite{R4} Definition 6.9.1, which will imply comonadicity of the adjunction by Remark 6.9.2 of \emph{loc. cit.} The convolution action of the infinitesimal groupoid \[ (S^2)^{\swedge}_{\Delta} \cong S \underset{S_{\dR}}{\times} S \] on $S$ gives rise to an action of $\IndCoh((S^2)^{\swedge}_{\Delta})$ on $\IndCoh(S)$, which in turn induces an equivalence of monoidal categories \[ \IndCoh((S^2)^{\swedge}_{\Delta}) \tilde{\longrightarrow} \End_{\DD(S)}(\IndCoh(S)). \] The natural t-structure on $\IndCoh((S^2)^{\swedge}_{\Delta})$ induces a t-structure on $\End_{\DD(S)}(\IndCoh(S))$, which we shift by $\dim S$ so that the coconnective objects are exactly the left t-exact functors (i.e., we normalize so that the unit $\id_{\IndCoh(S)}$ belongs to the heart).

The monad $\oblv_S \circ \ind_S$ lifts to an associative algebra in $\End_{\DD(S)}(\IndCoh(S))$, which corresponds to the algebra of differential operators in $\IndCoh((S^2)^{\swedge}_{\Delta})$. According to Lemma 6.17.1 of \emph{loc. cit.}, to prove that the monad is effective, it suffices to show that the following conditions are satisfied:
\begin{enumerate}[(i)]
\item $\oblv_S \circ \ind_S$ is coconnective in $\End_{\DD(S)}(\IndCoh(S))$;
\item $\oblv_S \circ \ind_S$ is right flat, i.e., the functor
\begin{align*}
\End_{\DD(S)}(\IndCoh(S)) &\longrightarrow \End_{\DD(S)}(\IndCoh(S)) \\
F &\mapsto \oblv_S \circ \ind_S \circ F
\end{align*}
is left t-exact;
\item the object \[ \cofib(\id_{\IndCoh(S)} \longrightarrow \oblv_S \circ \ind_S) \] in $\End_{\DD(S)}(\IndCoh(S))$ is right flat and eventually connective.
\end{enumerate}
All three of these conditions follow from the existence of the standard filtration on differential operators. Namely, the corresponding filtration on $\oblv_S \circ \ind_S$ in $\End_{\DD(S)}(\IndCoh(S))$ has associated graded given by tensoring with the locally free sheaf $\Sym_{\sO_S}\TT(S)$ (using the action of $\QCoh(S)$ on $\IndCoh(S)$).

\end{proof}

\subsection{Tensor product t-structures} Recall that for any DG categories $\sC$ and $\sD$ equipped with t-structures, their tensor product $\sC \otimes \sD$ admits a canonical t-structure, whose connective subcategory $(\sC \otimes \sD)^{\leq 0}$ is generated under colimits by objects $c \boxtimes d$ with $c$ and $d$ connective. If $\sC = \DD(S_1)$ and $\sD = \DD(S_2)$, then this t-structure makes the natural equivalence \[ \DD(S_1) \otimes \DD(S_2) \tilde{\longrightarrow} \DD(S_1 \times S_2) \] t-exact.

We will use the following technical lemma frequently. Recall that the t-structure on $\sC$ is said to be \emph{compactly generated} if the category $\sC^{\leq 0}$ is compactly generated.

\begin{lemma}

\label{tensorexact}

Let $\sC_1$, $\sC_2$, and $\sD$ be DG categories equipped with t-structures, and suppose we are given a functor $F : \sC_1 \to \sC_2$.
\begin{enumerate}[(i)]
\item If $F$ is right t-exact, then so is \[ F \otimes \id_{\sD} : \sC_1 \otimes \sD \longrightarrow \sC_2 \otimes \sD. \]
\item Suppose that $F$ is left t-exact, and that either
\begin{enumerate}[\indent(a)]
\item the t-structures on $\sC_1$ and $\sC_2$ are compactly generated, or
\item the t-structure on $\sD$ is compactly generated. 
\end{enumerate}
Then $F \otimes \id_{\sD}$ is left t-exact.
\end{enumerate}

\end{lemma}

\begin{proof}

This is a combination of Proposition A.1, Lemma A.6, and Lemma A.9 in \cite{RY}.

\end{proof}

\subsection{Compatible t-structures} The Verdier self-duality equivalence \[ \DD(S)^{\vee} \cong \DD(S) \] interchanges the symmetric monoidal structure whose multiplication is given by $\otimes^!$ with the cocommutative coalgebra structure on $\DD(S)$ whose comultiplication is given by \[ \Delta_{\dR,*} : \DD(S) \longrightarrow \DD(S \times S) \cong \DD(S) \otimes \DD(S). \] In particular, a $(\DD(S),\otimes^!)$-module structure on a DG category $\sC$ is the same datum as a $(\DD(S),\Delta_{\dR,*})$-comodule structure. In what follows we pass freely between the two.

Similarly to the case of D-modules, the Serre self-duality equivalence \[ \IndCoh(S)^{\vee} \cong \IndCoh(S) \] interchanges the symmetric monoidal structure whose multiplication is given by $\otimes^!$ with the cocommutative coalgebra structure on $\IndCoh(S)$ whose comultiplication is given by $\Delta_*$, so that an $(\IndCoh(S),\otimes^!)$-module structure is the same datum as an $(\IndCoh(S),\Delta_*)$-comodule structure.

\begin{definition}

Given a $\DD(S)$-module category $\sC$, we say that a t-structure on $\sC$ is \emph{compatible with the} $\DD(S)$\emph{-module structure} if the corresponding coaction map \[ \sC \longrightarrow \DD(S) \otimes \sC \] is t-exact. Similarly, if $\sD$ is an $\IndCoh(S)$-module category, we say that a t-structure on $\sD$ is \emph{compatible with the} $\IndCoh(S)$\emph{-module structure} if the coaction map \[ \sD \to \IndCoh(S) \otimes \sD \] is t-exact.

\end{definition}

\begin{proposition}

If $\sC$ is a $\DD(S)$-module category equipped with a compatible t-structure, then the category $\IndCoh(S) \otimes_{\DD(S)} \sC$ admits a unique t-structure such that $\ind_{\sC}$ is t-exact. Moreover, this t-structure makes $\oblv_{\sC}$ t-exact and is compatible with the $\IndCoh(S)$-action.

\end{proposition}

\begin{proof}

We claim that the comonad $\ind_{\sC} \circ \oblv_{\sC}$ is t-exact. In view of the previous proposition, this will imply that there is a unique t-structure on $\IndCoh(S) \otimes_{\DD(S)} \sC$ which makes $\ind_{\sC}$ t-exact (because the comonad is left t-exact), and moreover that $\oblv_{\sC}$ is t-exact with respect to this t-structure (because the comonad is t-exact).

The cartesian square \[
\begin{tikzcd}
S \arrow{r} \arrow{d} & S \times S_{\dR} \arrow{d} \\
S_{\dR} \arrow{r} & S_{\dR} \times S_{\dR}.
\end{tikzcd} \]
gives rise to a commutative square \[
\begin{tikzcd}
\IndCoh(S) \arrow{r} \arrow{d}{\ind_S} & \IndCoh(S) \otimes \DD(S) \arrow{d}{\ind_S \otimes \id_{\DD(S)}} \\
\DD(S) \arrow{r} & \DD(S) \otimes \DD(S),
\end{tikzcd} \]
where the functors are given by $\IndCoh$ direct image, and base change implies that the square \[
\begin{tikzcd}
\DD(S) \arrow{r} \arrow{d}{\oblv_S} & \DD(S) \otimes \DD(S) \arrow{d}{\oblv_S \otimes \id_{\DD(S)}} \\
\IndCoh(S) \arrow{r} & \IndCoh(S) \otimes \DD(S)
\end{tikzcd} \]
also commutes. Both of these squares take place in the category of $\DD(S)$-modules, where $\DD(S)$ acts on the right factor of the tensor product categories. Notice also that the horizontal functors are fully faithful and admit $\DD(S)$-linear right adjoints. Thus, tensoring over $\DD(S)$ with $\sC$, we obtain a commutative diagram \[
\begin{tikzcd}
\sC \arrow{r} \arrow{d}{\oblv_{\sC}} & \DD(S) \otimes \sC \arrow{d}{\oblv_S \otimes \id_{\sC}} \\
\IndCoh(S)\otimes_{\DD(S)} \sC \arrow{r} \arrow{d}{\ind_{\sC}} & \IndCoh(S) \otimes \sC \arrow{d}{\ind_S \otimes \id_{\sC}} \\
\sC \arrow{r} & \DD(S) \otimes \sC.
\end{tikzcd} \]
where the horizontal functors are fully faithful. In order to prove that $\ind_{\sC} \circ \oblv_{\sC}$ is t-exact, it therefore suffices to prove that the bottom-left circuit of the outer square is t-exact, since the bottom horizontal functor is the coaction functor and hence t-exact by hypothesis.

Since the upper horizontal functor is t-exact, it is enough to prove that both right vertical functors are t-exact. The functors $\oblv_S$ and $\ind_S$ are both t-exact and the t-structures on $\IndCoh(S)$ and $\DD(S)$ are compactly generated, so this follows from Lemma \ref{tensorexact}.

To see that the t-structure on $\IndCoh(S) \otimes_{\DD(S)} \sC$ is compatible with the $\IndCoh(S)$-action, consider the commutative square \[
\begin{tikzcd}
S \arrow{r} \arrow{d} & S \times S \arrow{d} \\
S_{\dR} \arrow{r} & S_{\dR} \times S_{\dR}.
\end{tikzcd} \]
Passing to ind-coherent sheaves under direct image and tensoring over $\DD(S)$ with $\sC$, we see that the coaction functors fit into a commutative square \[
\begin{tikzcd}
\IndCoh(S) \otimes_{\DD(S)} \sC \arrow{r} \arrow{d}{\ind_{\sC}} & \IndCoh(S) \otimes \IndCoh(S) \otimes_{\DD(S)} \sC \arrow{d}{\ind_S \otimes \ind_{\sC}} \\
\sC \arrow{r} & \DD(S) \otimes \sC.
\end{tikzcd} \]
Since the lower-left circuit of this square is t-exact, it suffices to show that $\ind_S \otimes \ind_{\sC}$ is t-exact and conservative. We have seen above that \[ \ind_S \otimes \id_{\sC} : \IndCoh(S) \otimes \sC \longrightarrow \DD(S) \otimes \sC \] is t-exact. Since $\ind_{\sC}$ is t-exact and the t-structure on $\IndCoh(S)$ is compactly generated, the functor \[ \id_{\IndCoh(S)} \otimes \ind_{\sC} : \IndCoh(S) \otimes \sC \longrightarrow \IndCoh(S) \otimes \IndCoh(S) \underset{\DD(S)}{\otimes} \sC \] is t-exact by Lemma \ref{tensorexact}. Note that the proof of Proposition \ref{comon} shows that $\ind_S \otimes \id_{\sC}$ and $\id_{\IndCoh(S)} \otimes \ind_{\sC}$ are both comonadic and in particular conservative. Thus their composition $\ind_S \otimes \ind_{\sC}$ is t-exact and conservative as claimed.

\end{proof}

\subsection{Almost compact generation} This section borrows heavily from the appendix to \cite{BNP}. We warn the reader that our terminology does not agree with \emph{loc. cit.} in some places. For example, we use the term ``coherent t-structure" for a weaker notion than theirs.

Let $\sC$ be a DG category equipped with a t-structure. Recall that an object $c$ in $\sC$ is called \emph{almost compact} if its truncation $\tau^{\geq n}c$ is compact in $\sC^{\geq n}$ for any $n \in \bZ$. If $c$ is eventually coconnective and almost compact, then $c$ is called \emph{coherent}. Denote by $\sC^{\coh} \subset \sC$ the (non-cocomplete) full subcategory consisting of coherent objects.

If the t-structure on $\sC$ is right complete, then any almost compact object is eventually connective, and hence any coherent object is cohomologically bounded.

We will use the folllowing technical lemma repeatedly.

\begin{lemma}

\label{aclem}

Let $\sC$ and $\sD$ be DG categories equipped with t-structures which are compatible with filtered colimits and right complete. Suppose that $F : \sC \to \sD$ is t-exact and admits a continuous right adjoint, and moreover that $F$ is conservative on $\sC^+$. If $c$ belongs to $\sC^+$ and $F(c)$ is almost compact, then $c$ is almost compact (equivalently, coherent).

\end{lemma}

\begin{proof}

See Step 3 in the proof of Lemma 6.11.2 in \cite{R5}.

\end{proof}

\begin{definition}

We say that $\sC$ is \emph{coherent} with respect to its t-structure if the inclusion \[ \sC^{\geq 0} \longrightarrow \sC^{\geq -1} \] preserves compact objects.

\end{definition}

For example, if $A$ is a classical commutative algebra, then $A\mod$ with its usual t-structure is coherent if and only if $A$ is coherent, i.e., every finitely generated ideal is finitely presented.

\begin{proposition}

Suppose that the t-structure on $\sC$ is compatible with filtered colimits. Then $\sC$ is coherent if and only if almost compact objects in $\sC$ are stable under truncation.

\end{proposition}

In particular, if $\sC$ is coherent then the subcategory $\sC^{\coh}$ is stable under truncation functors, and hence inherits a t-structure from $\sC$. This t-structure extends uniquely to a t-structure on the \emph{canonical renormalization} \[ \sC^{\ren} := \Ind(\sC^{\coh}) \] which is compatible with filtered colimits. The unique continuous extension \[ \sC^{\ren} \longrightarrow \sC \] of the inclusion $\sC^{\coh} \subset \sC$ is t-exact.

\begin{definition}

We call $\sC$ \emph{almost compactly generated} with respect to its t-structure if it satisfies the following conditions:
\begin{enumerate}[(i)]
\item the t-structure is compatible with filtered colimits and right complete;
\item $\sC$ is coherent;
\item $\sC^{\geq 0}$ is compactly generated.
\end{enumerate}

\end{definition}

\begin{rem}

Conditions (i) and (iii) are usually rather easy to verify in practice, while coherence is more subtle. One particularly bothersome example of this is the fact that the tensor product of coherent categories need not be coherent; for instance, it is well-known that a ring of polynomials with coefficients in a coherent ring need not be coherent.

\end{rem}

The canonical renormalization behaves well for almost compactly generated categories.

\begin{proposition}

If $\sC$ is almost compactly generated, then the t-exact functor $\sC^{\ren} \to \sC$ restricts to an equivalence \[ \sC^{\ren,+} \tilde{\longrightarrow} \sC^+. \]

\end{proposition}

The next lemma supplies a convenient criterion for almost compact generation.

\begin{lemma}

Suppose that $\sC$ is equipped with a t-structure such that
\begin{enumerate}[(i)]
\item the t-structure is compatible with filtered colimits and right complete;
\item the category $\sC^{\heartsuit}$ is compactly generated;
\item any compact object in $\sC^{\heartsuit}$ is almost compact in $\sC$.
\end{enumerate}
Then $\sC$ is almost compactly generated.

\label{cohcrit}

\end{lemma}

\subsection{Functorial renormalization} Suppose that $\sC$ and $\sD$ are DG categories equipped with t-structures. We will write \[ \Fun_{\fha}(\sC,\sD) \subset \Fun(\sC,\sD) \] for the full subcategory consisting of exact continuous functors of finite homological amplitude, or equivalently those exact continuous functors which send $\sC^+$ into $\sD^+$. Denote by \[ \Fun_{+\text{-}\!\cts}(\sC^+,\sD^+) \subset \Fun(\sC^+,\sD^+) \] the full subcategory consisting of exact functors which commute with filtered colimits bounded uniformly from below. There is a natural restriction functor
\begin{equation}
\Fun_{\fha}(\sC,\sD) \longrightarrow \Fun_{+\text{-}\!\cts}(\sC^+,\sD^+).
\label{fhaequiv}
\end{equation}

\begin{lemma}

\label{fhaext}

Suppose that $\sC$ is almost compactly generated, and that $\sC^{\ren} \tilde{\to} \sC$. Then the functor (\ref{fhaequiv}) is an equivalence.

\end{lemma}

\begin{proof}

The inverse functor is given by the composite \[ \Fun_{+\text{-}\!\cts}(\sC^+,\sD^+) \longrightarrow \Fun(\sC^{\coh},\sD) \longrightarrow \Fun(\sC,\sD), \] where the first functor is restriction and the second is left Kan extension. Namely, as observed in Example 4.4.4 of \cite{R5}, the fact that $\sC^{\coh}$ is stable under truncation functors implies that any functor in $\Fun_{+\text{-}\!\cts}(\sC^+,\sD^+)$ is left Kan extended from $\sC^{\coh}$.

\end{proof}

Let us write $\DGCat_{\tstr}$ for the category of DG categories equipped with t-structures, with $1$-morphisms given by all continuous functors (we impose no t-exactness conditions on functors here). As described above, the Lurie tensor product of two objects in $\DGCat_{\tstr}$ admits a canonical t-structure, and hence the Lurie tensor product lifts to a symmetric monoidal structure on $\DGCat_{\tstr}$. We denote by $(\DGCat_{\tstr})^{\otimes}$ the associated pseudo-tensor category.\footnote{The term ``pseudo-tensor category," which we have borrowed from Beilinson-Drinfeld, is formally identical to ``colored operad" but plays a different conceptual role. A morphism from a colored operad to a pseudo-tensor category should be thought of as an algebra over the colored operad in the pseudo-tensor category. In particular, for us colored operads are ``small," while pseudo-tensor categories tend to be ``large."}

We define a $1$-full pseudo-tensor subcategory \[ (\DGCat_{\tstr})^{\otimes}_{\fha} \subset (\DGCat_{\tstr})^{\otimes} \] as follows. An object in $\DGCat_{\tstr}$ belongs to $(\DGCat_{\tstr})^{\otimes}_{\fha}$ provided that it is almost compactly generated. An $n$-ary morphism \[ \sC_1 \otimes \cdots \otimes \sC_n \longrightarrow \mathcal{D} \] belongs to $(\DGCat_{\tstr})^{\otimes}_{\fha}$ provided that the image of the composition \[ \sC_1^+ \times \cdots \times \sC_n^+ \longrightarrow \sC_1 \otimes \cdots \otimes \sC_n \longrightarrow \mathcal{D} \] is contained in $\mathcal{D}^+$.

Next, consider the symmetric monoidal category $\DGCat_{\tstr}^{\Delta^1}$ of arrows in $\DGCat_{\tstr}$. Define a $1$-full pseudo-tensor subcategory \[ (\DGCat_{\tstr}^{\Delta^1})^{\otimes}_{\ren} \subset (\DGCat_{\tstr}^{\Delta^1})^{\otimes} \] as follows. An object $\sC' \to \sC$ in $\DGCat_{\tstr}^{\Delta^1}$ belongs to $(\DGCat_{\tstr}^{\Delta^1})^{\otimes}_{\ren}$ provided that the following conditions are satisfied:
\begin{enumerate}[(i)]
\item $\sC' \to \sC$ is a t-exact functor;
\item $\sC$ and $\sC'$ are almost compactly generated;
\item the category $\sC'$ is compactly generated by $(\sC')^{\coh}$;
\item $\sC' \to \sC$ induces an equivalence $(\sC')^+ \tilde{\to} \sC^+$.
\end{enumerate}
(Note that these conditions imply that $\sC' \to \sC$ factors uniqely through a continuous t-exact equivalence $\sC' \tilde{\to} \sC^{\ren}$.) Given objects  \[ \sC_1' \to \sC_1,\cdots \!,\sC_n' \to \sC_n,\sD' \to \sD \] in $(\DGCat_{\tstr}^{\Delta^1})^{\otimes}_{\ren}$, an $n$-ary morphism \[
\begin{tikzcd}
\sC_1' \otimes \cdots \otimes \sC_n' \arrow{r} \arrow{d} & \mathcal{D}' \arrow{d} \\
\sC_1 \otimes \cdots \otimes \sC_n \arrow{r} & \mathcal{D},
\end{tikzcd} \]
in $(\DGCat_{\tstr}^{\Delta^1})^{\otimes}$ is an $n$-ary morphism in $(\DGCat_{\tstr}^{\Delta^1})^{\otimes}_{\ren}$ provided that both horizontal functors are $n$-ary morphisms in $(\DGCat_{\tstr})^{\otimes}_{\fha}$ (in fact, it suffices to impose this condition on the lower horizontal functor).

\begin{proposition}

\label{operadequiv}

The second projection
\begin{align*}
\DGCat_{\tstr}^{\Delta^1} &\longrightarrow \DGCat_{\tstr} \\
(\sC' \to \sC) &\mapsto \sC
\end{align*}
induces an equivalence of pseudo-tensor categories \[ (\DGCat_{\tstr}^{\Delta^1})^{\otimes}_{\ren} \tilde{\longrightarrow} (\DGCat_{\tstr})^{\otimes}_{\fha}. \]

\end{proposition}

\begin{proof}

This functor is essentially surjective on the underlying categories because for any almost compactly generated category $\sC$, the canonical renormalization $\sC^{\ren} \to \sC$ is an object of $(\DGCat_{\tstr}^{\Delta^1})^{\otimes}_{\ren}$ which projects to $\sC$.

To see that this functor is fully faithful on $n$-ary morphisms, suppose that we are given objects \[ \sC_1' \to \sC_1,\cdots \!,\sC_n' \to \sC_n,\mathcal{D}' \to \mathcal{D} \] in $(\DGCat_{\tstr}^{\Delta^1})^{\otimes}_{\ren}$. We claim that any $n$-ary morphism \[ \sC_1 \otimes \cdots \otimes \sC_n \longrightarrow \mathcal{D} \] in $(\DGCat_{\tstr})^{\otimes}_{\fha}$ lifts to an $n$-ary morphism in $(\DGCat_{\tstr}^{\Delta^1})^{\otimes}_{\ren}$ which is unique up to contractible choice. Namely, the image of the composition \[ (\sC_1')^{\coh} \times \cdots \times (\sC_n')^{\coh} \tilde{\longrightarrow} \sC_1^{\coh} \times \cdots \times \sC_n^{\coh} \longrightarrow \mathcal{D} \] is contained in $\mathcal{D}^+$. Since this functor is exact in each variable separately, it follows that there is a unique continuous extension to \[ \sC_1' \otimes \cdots \otimes \sC_n' = \Ind((\sC_1')^{\coh}) \otimes \cdots \otimes \Ind((\sC_n')^{\coh}) \longrightarrow \mathcal{D}. \] Note that this extension defines an $n$-ary morphism in $(\DGCat_{\tstr}^{\Delta^1})^{\otimes}_{\ren}$, since any object in $(\sC_i')^+$ is a filtered colimit of objects in $(\sC_i')^{\coh}$ bounded uniformly from below. This completes the proof.

\end{proof}

We will make frequent use of the endomorphism of pseudo-tensor categories
\begin{align}
\label{renfunctorial}
(\DGCat_{\tstr})^{\otimes}_{\fha} &\longrightarrow (\DGCat_{\tstr})^{\otimes}_{\fha} \\
\sC &\mapsto \sC^{\ren} \nonumber
\end{align}
obtained as the composition \[ (\DGCat_{\tstr})^{\otimes}_{\fha} \tilde{\longrightarrow} (\DGCat_{\tstr}^{\Delta^1})^{\otimes}_{\ren} \longrightarrow (\DGCat_{\tstr})^{\otimes}_{\fha}, \] where the first functor is inverse to the equivalence of Proposition \ref{operadequiv} and the second functor is the first projection $(\sC' \to \sC) \mapsto \sC'$.

\subsection{Renormalization for crystals of categories} We continue to let $S$ denote a smooth quasicompact scheme.

Consider the pseudo-tensor category $(\DD(S)\mod_{\tstr})^{\otimes}$ whose objects are $\DD(S)$-module categories equipped with a compatible t-structure, with $n$-ary morphisms $(\sC_1,\cdots \!,\sC_n) \to \sD$ given by arbitrary $\DD(S)$-linear functors \[ \sC_1 \underset{\DD(S)}{\otimes} \cdots \underset{\DD(S)}{\otimes} \sC_n \longrightarrow \sD. \]

We define a pseudo-tensor subcategory \[ (\DD(S)\mod_{\tstr})^{\otimes}_{\fha} \subset (\DD(S)\mod_{\tstr})^{\otimes} \] a follows. An object $\sC$ of $(\DD(S)\mod_{\tstr})^{\otimes}$ belongs to $(\DD(S)\mod_{\tstr})^{\otimes}_{\fha}$ if $\sC$ is almost compactly generated. An $n$-ary morphism \[ \sC_1 \underset{\DD(S)}{\otimes} \cdots \underset{\DD(S)}{\otimes} \sC_n \longrightarrow \sD \] belongs to $(\DD(S)\mod_{\tstr})^{\otimes}_{\fha}$ provided that the image of the composition \[ \sC_1^+ \times \cdots \times \sC_n^+ \longrightarrow \sC_1 \underset{\DD(S)}{\otimes} \cdots \underset{\DD(S)}{\otimes} \sC_n \longrightarrow \sD \] is contained in $\sD^+$.

\begin{proposition}

\label{rencrys}

The morphism (\ref{renfunctorial}) lifts canonically to a morphism of pseudo-tensor categories
\begin{align*}
(\DD(S)\mod_{\tstr})^{\otimes}_{\fha} &\longrightarrow (\DD(S)\mod_{\tstr})^{\otimes}_{\fha} \\
\sC &\mapsto \sC^{\ren}.
\end{align*}

\end{proposition}

\begin{proof}

Observe that $\DD(S)$ is a commutative algebra object in $(\DGCat_{\tstr})^{\otimes}_{\fha}$ satisfying $\DD(S)^{\ren} \tilde{\to} \DD(S)$. Moreover, for any $\DD(S)$-module category $\sC$, the coaction functor \[ \coact_{\sC} : \sC \longrightarrow \DD(S) \otimes \sC \] is left adjoint to the action functor \[ \act_{\sC} : \DD(S) \otimes \sC \longrightarrow \sC. \] It follows that if $\sC$ is equipped with a compatible t-structure, then $\act_{\sC}$ is left t-exact, and hence the $\DD(S)$-action on $\sC$ takes place in $(\DGCat_{\tstr})^{\otimes}_{\fha}$ provided that $\sC$ is almost compactly generated. All of this implies that (\ref{renfunctorial}) is functorial on $\DD(S)$-modules; we need only check that if the $\DD(S)$-action on $\sC$ is compatible with the t-structure, then so is the induced action on $\sC^{\ren}$.

First, we claim that the functor $\coact_{\sC^{\ren}}$ has finite homological amplitude. Since $(\sC^{\ren})^+$ is generated by compact objects of $\sC^{\ren}$ under filtered colimits bounded uniformly from below, it suffices to show that $\coact_{\sC^{\ren}}$ sends compact objects to eventually coconnective objects. Observe that $\coact_{\sC^{\ren}}$ is left adjoint to the continuous functor $\act_{\sC^{\ren}}$ and hence preserves compact objects, so the claim will follow if we prove that compact objects in $\DD(S) \otimes \sC^{\ren}$ are eventually coconnective. The category $\DD(S) \otimes \sC^{\ren}$ is compactly generated by objects of the form $\sM \boxtimes c$ where $\sM$ and $c$ are compact in $\DD(S)$ in $\sC^{\ren}$ respectively, so we need only prove that objects of this form are eventually coconnective, since they generate all compact objects in $\DD(S) \otimes \sC^{\ren}$ under finite colimits and retracts. But compact objects in $\DD(S)$ and $\sC^{\ren}$ are eventually coconnective, which implies that objects $\sM \boxtimes c$ as above are eventually coconnective, by Lemma \ref{tensorexact}.

To see that $\coact_{\sC^{\ren}}$ is t-exact, consider the commutative square \[
\begin{tikzcd}
\sC^{\ren} \arrow{r} \arrow{d} & \DD(S) \otimes \sC^{\ren} \arrow{d} \\
\sC \arrow{r} & \DD(S) \otimes \sC.
\end{tikzcd} \]
The left vertical and lower horizontal functors are t-exact, and the right vertical functor is t-exact by Lemma \ref{tensorexact}. We have already shown that the upper horizontal functor has finite homological amplitude, and hence we obtain a commutative square \[
\begin{tikzcd}
(\sC^{\ren})^+ \arrow{r} \arrow{d} & (\DD(S) \otimes \sC^{\ren})^+ \arrow{d} \\
\sC^+ \arrow{r} & (\DD(S) \otimes \sC)^+.
\end{tikzcd} \]
Since $(\sC^{\ren})^+ \to \sC^+$ is an equivalence and in particular conservative, Lemma 4.6.2(3) in \cite{R5} implies that the right vertical functor is conservative. It follows immediately that $\coact_{\sC^{\ren}}$ is left t-exact. As for right t-exactness, since the t-structure on $\sC^{\ren}$ is compactly generated, it suffices to show that $\coact_{\sC^{\ren}}$ sends compact objects in $(\sC^{\ren})^{\leq 0}$ into $(\DD(S) \otimes \sC^{\ren})^{\leq 0}$. Since compact objects in $\sC^{\ren}$ are eventually coconnective, this follows from the corresponding property of $\coact_{\sC}$ using the same commutative square.

\end{proof}

\subsection{Almost ULA generation} Suppose that we are given a $\DD(S)$-module category $\sC$ with compatible t-structure.

\begin{definition}\label{d:aula}

An object $c$ in $\sC$ will be called \emph{almost ULA} over $S$ if $\oblv_{\sC}(c)$ is almost compact in $\IndCoh(S) \otimes_{\DD(S)} \sC$. We will say that $\sC$ is \emph{almost ULA generated} if $\sC$ is almost compactly generated, and moreover $(\IndCoh(S) \otimes_{\DD(S)} \sC)^{\geq 0}$ is compactly generated by objects of the form \[ \tau^{\geq 0}(\sF \otimes \oblv_{\sC}(c)) \] where $\sF$ is a perfect complex on $S$ and $c$ is almost ULA in $\sC$.

\end{definition}

As the name suggests, the renormalization of an almost ULA generated category is ULA generated.

\begin{proposition}

\label{ularen}

Suppose that a $\DD(S)$-module category $\sC$ is almost ULA generated. Then the $\DD(S)$-module category $\sC^{\ren}$ is ULA generated by the objects $\tau^{\geq n}c$ in \[ \sC^+ \cong (\sC^{\ren})^+ \subset \sC^{\ren}, \] where $n \in \bZ$ and $c$ is almost ULA in $\sC$.

\end{proposition}

This proposition is immediate from the following lemma.

\begin{lemma}

Let $\sC$ be a $\DD(S)$-module category equipped with compatible t-structure, and suppose that $\sC$ is almost compactly generated. Then $\IndCoh(S) \otimes_{\DD(S)} \sC$ is coherent.

If we assume in addition that $(\IndCoh(S) \otimes_{\DD(S)} \sC)^{\geq 0}$ is compactly generated (e.g. if $\sC$ is almost ULA generated), it follows that $\IndCoh(S) \otimes_{\DD(S)} \sC$ is almost compactly generated. In that case, the functor \[ \IndCoh(S) \underset{\DD(S)}{\otimes} \sC^{\ren} \longrightarrow \IndCoh(S) \underset{\DD(S)}{\otimes} \sC \] induced by $\sC^{\ren} \to \sC$ factors uniquely through a t-exact equivalence \[ \IndCoh(S) \underset{\DD(S)}{\otimes} \sC^{\ren} \tilde{\longrightarrow} (\IndCoh(S) \underset{\DD(S)}{\otimes} \sC)^{\ren}. \] 

\end{lemma}

\begin{proof}

For the first claim, it suffices to show that a compact object $c$ in $(\IndCoh(S) \otimes_{\DD(S)} \sC)^{\geq 0}$ is almost compact (equivalently, coherent). Since $\ind_{\sC}$ is t-exact with continuous and t-exact right adjoint $\oblv_{\sC}$, the object $\ind_{\sC}(c)$ is compact in $\sC^{\geq 0}$ and hence almost compact in $\sC$ by coherence of the t-structure on the latter category. Now Lemma \ref{aclem} implies that $c$ is almost compact in $\IndCoh(S) \otimes_{\DD(S)} \sC$ as desired.

Next, consider the functor \[ \IndCoh(S) \underset{\DD(S)}{\otimes} \sC^{\ren} \longrightarrow \IndCoh(S) \underset{\DD(S)}{\otimes} \sC \] induced by $\sC^{\ren} \to \sC$. This fits into the commutative square \[
\begin{tikzcd}
\IndCoh(S) \otimes_{\DD(S)} \sC^{\ren} \arrow{r} \arrow{d}{\ind_{\sC^{\ren}}} & \IndCoh(S) \otimes_{\DD(S)} \sC \arrow{d}{\ind_{\sC}} \\
\sC^{\ren} \arrow{r} & \sC
\end{tikzcd} \]
and is therefore t-exact, since the vertical functors are t-exact and conservative. For the same reason, both \[ (\IndCoh(S) \underset{\DD(S)}{\otimes} \sC^{\ren})^+ \xrightarrow{\ind_{\sC^{\ren}}} (\sC^{\ren})^+ \] and \[ (\IndCoh(S) \underset{\DD(S)}{\otimes} \sC)^+ \xrightarrow{\ind_{\sC}} \sC^+ \] are comonadic (cf. the Proof of Lemma 3.7.1 in \cite{R5}). Since $(\sC^{\ren})^+ \tilde{\to} \sC^+$, we deduce that the functor under consideration induces an equivalence \[ (\IndCoh(S) \underset{\DD(S)}{\otimes} \sC^{\ren})^+ \tilde{\longrightarrow} (\IndCoh(S) \underset{\DD(S)}{\otimes} \sC)^+. \]

It folows that the functor \[ (\IndCoh(S) \underset{\DD(S)}{\otimes} \sC)^{\ren} \longrightarrow \IndCoh(S) \underset{\DD(S)}{\otimes} \sC \] factors uniquely through a t-exact continuous functor \[ (\IndCoh(S) \underset{\DD(S)}{\otimes} \sC)^{\ren} \longrightarrow \IndCoh(S) \underset{\DD(S)}{\otimes} \sC^{\ren}, \] which we claim is an equivalence. Namely, observe that the triangles \[
\begin{tikzcd}
(\IndCoh(S) \otimes_{\DD(S)} \sC)^{\ren} \arrow{r} \arrow{dr}{\ind_{\sC}^{\ren}} & \IndCoh(S) \otimes_{\DD(S)} \sC^{\ren} \arrow{d}{\ind_{\sC^{\ren}}} \\
& \sC^{\ren}
\end{tikzcd} \]
and \[
\begin{tikzcd}
(\IndCoh(S) \otimes_{\DD(S)} \sC)^{\ren} \arrow{r} & \IndCoh(S) \otimes_{\DD(S)} \sC^{\ren} \\
& \sC^{\ren} \arrow{ul}{\oblv_{\sC}^{\ren}} \arrow{u}{\oblv_{\sC^{\ren}}}
\end{tikzcd} \]
commute, since they commute after restricting to eventually coconnective objects, and the categories $(\IndCoh(S) \otimes_{\DD(S)} \sC)^{\ren}$ and $\sC^{\ren}$ are compactly generated by coherent objects. Since the functor $\ind_{\sC^{\ren}}$ is comonadic by Proposition \ref{comon}, it suffices to show that $\ind_{\sC}^{\ren}$ is comonadic.

For this, observe that $\End_{\DD(S)}(\IndCoh(S))$ is almost compactly generated with respect to the t-structure defined in the proof of Proposition \ref{comon}, and that \[ \End_{\DD(S)}(\IndCoh(S))^{\ren} \tilde{\longrightarrow} \End_{\DD(S)}(\IndCoh(S)). \] Namely, we have an equivalence \[ \End_{\DD(S)}(\IndCoh(S)) \tilde{\longrightarrow} \IndCoh((S^2)^{\swedge}_{\Delta}) \] which is t-exact up to shift, and the t-structure on the right side has these properties. The compatibility of the $D(S)$-action with the t-structure on $\sC$ implies that the action of $\End_{\DD(S)}(\IndCoh(S))$ on $\IndCoh(S) \otimes_{\DD(S)} \sC$ takes place in $(\DGCat_{\tstr})^{\otimes}_{\fha}$, and hence (\ref{renfunctorial}) induces an action on the renormalization $(\IndCoh(S) \otimes_{\DD(S)} \sC)^{\ren}$. The monad induced by the adjunction \[ \ind_{\sC}^{\ren} : (\IndCoh(S) \underset{\DD(S)}{\otimes} \sC)^{\ren} \rightleftarrows \sC^{\ren} : \oblv_{\sC}^{\ren} \] is given by the action of the associative algebra $\oblv_S \circ \ind_S$ in $\End_{\DD(S)}(\IndCoh(S))$, since this is the case on eventually coconnective objects and $(\IndCoh(S) \otimes_{\DD(S)} \sC)^{\ren}$ is compactly generated by coherent objects. As in Proposition \ref{comon}, it follows that this monad is effective and hence that $\ind_{\sC}^{\ren}$ is comonadic.

\end{proof}

\subsection{Change of base: open embedding} Suppose that $j : U \to S$ is an open embedding and $\sC$ is a $\DD(S)$-module category with compatible t-structure. Put \[ \sC_U := \DD(U) \underset{\DD(S)}{\otimes} \sC. \] Then we have an adjunction \[ j^* \otimes \id_{\sC} : \sC = \DD(S) \underset{\DD(S)}{\otimes} \sC \rightleftarrows \sC_U : j_* \otimes \id_{\sC} \] where the right adjoint is fully faithful, and in particular $\sC_U$ admits a unique t-structure such that $j_* \otimes \id_{\sC}$ is left t-exact, or equivalently such that $j^* \otimes \id_{\sC}$ is right t-exact. Here we are abusing notation slightly by writing $j_*$ rather than $j_{\dR,*}$, but since $j$ is an open embedding, this should not cause any confusion.

\begin{lemma}

\label{openresexact}

The functor $j^* \otimes \id_{\sC} : \sC \to \sC_U$ is t-exact. If $j$ is affine, then the right adjoint $j_* \otimes \id_{\sC} : \sC_U \to \sC$ is t-exact.

\end{lemma}

\begin{proof}

To see that $j^* \otimes \id_{\sC}$ is t-exact, it suffices to show it is left t-exact, which is equivalent to left t-exactness of the endofunctor \[ (j_*j^*) \otimes \id_{\sC} : \sC \longrightarrow \sC_U \longrightarrow \sC. \] Similarly, if $(j_*j^*) \otimes \id_{\sC}$ is t-exact, then $j_* \otimes \id_{\sC}$ is t-exact. We have a commutative square \[
\begin{tikzcd}[row sep=large, column sep=large]
\sC \arrow{r}{\coact_{\sC}} \arrow{d}{(j_*j^*) \otimes \id_{\sC}} & \DD(S) \otimes \sC \arrow{d}{(j_*j^*) \otimes \id_{\sC}} \\
\sC \arrow{r}{\coact_{\sC}} & \DD(S) \otimes \sC
\end{tikzcd} \]
where the horizontal functors are fully faithful and t-exact. By Lemma \ref{tensorexact}, the right vertical functor is left t-exact, respectively right t-exact if the functor $j_*j^*$ is such. The lemma follows.

\end{proof}

\begin{proposition}

\label{opencoh}

If $\sC$ is almost ULA generated over $S$, then $\sC_U$ is almost ULA generated over $U$, and the functor \[ (\sC^{\ren})_U \longrightarrow \sC_U \] induced by $\sC^{\ren} \to \sC$ factors uniquely through a t-exact equivalence \[ (\sC^{\ren})_U \tilde{\longrightarrow} (\sC_U)^{\ren}. \]

\end{proposition}

\begin{proof}

Suppose that $\sC$ is almost compactly generated. Since $j_* \otimes \id_{\sC} : \sC_U \to \sC$ is left t-exact and conservative, the t-structure on $\sC_U$ is compatible with filtered colimits. For the same reason, the t-structure on $\sC$ is right separated and hence right complete. To see that $(\sC_U)^{\geq 0}$ is compactly generated, consider the adjunction \[ j^* \otimes \id_{\sC} : \sC^{\geq 0} \rightleftarrows (\sC_U)^{\geq 0} : j_* \otimes \id_{\sC}. \] Since $j_* \otimes \id_{\sC}$ is continuous and fully faithful, hence conservative, it follows that $(\sC_U)^{\geq 0}$ is compactly generated.

To see that $\sC_U$ is almost compactly generated, it remains to show that the t-structure is coherent. We claim that any compact object $c$ in $(\sC_U)^{\geq 0}$ is a direct summand of an object $(j^* \otimes \id_{\sC})(c_0)$ where $c_0$ is a compact object in $\sC^{\geq 0}$. Since $\sC$ is almost compactly generated and $j^* \otimes \id_{\sC}$ preserves almost compact objects, it will follow that $c$ is almost compact in $\sC_U$. To prove the claim, write \[ (j_* \otimes \id_{\sC})(c) \cong \underset{\alpha}{\colim} \ c_{\alpha} \] as a filtered colimit of compact objects $c_{\alpha}$ in $\sC^{\geq 0}$. Then we have \[ c \cong \underset{\alpha}{\colim} \ (j^* \otimes \id_{\sC})(c_{\alpha}) \] in $(\sC_U)^{\geq 0}$. Compactness of $c$ in $(\sC_U)^{\geq 0}$ now implies that $c$ is a direct summand of $(j^* \otimes \id_{\sC})(c_{\alpha})$ for some $\alpha$.

Renormalization of the above adjunction yields a $\DD(S)$-linear adjunction \[ (j^* \otimes \id_{\sC})^{\ren} : \sC^{\ren} \rightleftarrows (\sC_U)^{\ren} : (j_* \otimes \id_{\sC})^{\ren}. \] The right adjoint is continuous and fully faithful, so in particular this adjunction is monadic. Note that $(j_* \otimes \id_{\sC})^{\ren}$ factors through \[ j_* \otimes \id_{\sC^{\ren}} : (\sC^{\ren})_U \longrightarrow \sC^{\ren}, \] since this is the case after restricting to $(\sC_U)^+$. The resulting functor \[ (\sC_U)^{\ren} \longrightarrow (\sC^{\ren})_U \] corresponds to a morphism of (left t-exact) monads on $\sC^{\ren}$, which is an isomorphism because it is such after restricting to $\sC^+$.

The assertion that $\sC_U$ is almost ULA generated over $U$ follows readily from the above.

\end{proof}

\subsection{Change of base: closed embedding} Suppose that we are given a smooth closed subvariety $i : T \to S$. Writing \[ \sC_T := \DD(T) \underset{\DD(S)}{\otimes} \sC \] for any $\DD(S)$-module category $\sC$, note that we have an adjunction \[  i_{\dR,*} \otimes \id_{\sC} : \sC_T \rightleftarrows \sC : i^! \otimes \id_{\sC}, \] with the left adjoint being fully faithful.

\begin{lemma}

The functor $i^! \otimes \id_{\sC} : \sC \to \sC_T$ sends objects ULA over $S$ to objects ULA over $T$. If $\sC$ is ULA generated over $S$, the $\sC_T$ is ULA generated over $T$ by objects of the form $(i^! \otimes \id_{\sC})(c)$, where $c$ is ULA in $\sC$.

\label{shriekresula}

\end{lemma}

\begin{proof}

Consider the commutative square of symmetric monoidal categories \[
\begin{tikzcd}
\DD(S) \arrow{r}{i^!} \arrow{d}{\oblv_S} & \DD(T) \arrow{d}{\oblv_T} \\
\IndCoh(S) \arrow{r}{i^!} & \IndCoh(T).
\end{tikzcd} \]
This can also be viewed as a commutative square of $\DD(S)$-modules, which we can tensor with $\sC$ to obtain \[
\begin{tikzcd}
\sC \arrow{r}{i^! \otimes \id_{\sC}} \arrow{d}{\oblv_S} &\sC_T \arrow{d}{\oblv_T} \\
\IndCoh(S) \otimes_{\DD(S)} \sC \arrow{r}{i^! \otimes \id_{\sC}} & \IndCoh(T) \otimes_{\DD(T)} \sC_T.
\end{tikzcd} \]
Since $S$ and $T$ are smooth, the functor \[ i_* : \IndCoh(T) \longrightarrow \IndCoh(S) \] admits a left adjoint $i^{\IndCoh,*}$, which agrees with $i^!$ up to tensoring with a cohomologically shifted line bundle. Thus $i^!$ admits a $\DD(S)$-linear right adjoint, which implies that the lower horizontal functor in the above square admits a continuous right adjoint and hence preserves compact objects. It follows immediately that the upper horizontal functor preserves ULA objects.

The second claim follows from the observation that the image of \[ i^! \otimes \id_{\sC} : \IndCoh(S) \otimes_{\DD(S)} \sC \longrightarrow \IndCoh(T) \otimes_{\DD(T)} \sC_T \] generates the target under colimits, which is immediate from the corresponding property of \[ i^! : \IndCoh(S) \longrightarrow \IndCoh(T) \] (this being equivalent to the conservativity of its right adjoint, which is isomorphic to $i_*$ tensored with a cohomologically shifted line bundle).

\end{proof}

\begin{lemma}

Suppose that $\sC$ is equipped with a t-structure compatible with the $\DD(S)$-action. Then the image of the fully faithful functor \[ i_{\dR,*} \otimes \id_{\sC} : \sC_T \longrightarrow \sC \] is stable under truncations, so $\sC_T$ admits a unique t-structure such that $i_{\dR,*} \otimes \id_{\sC}$ is t-exact. Moreover, the subcategory $(\sC_T)^{\heartsuit} \subset \sC^{\heartsuit}$ is stable under taking subobjects, and the t-structure on $\sC_T$ is compatible with the $\DD(T)$-action.

\end{lemma}

\begin{proof}

The stability of $\sC_T$ under truncations and of $(\sC_T)^{\heartsuit}$ under taking subobjects both follow formally from Lemma \ref{openresexact} (specifically, the fact that $j^* \otimes \id_{\sC}$ is t-exact). To see that the $\DD(T)$-action on $\sC_T$ is compatible with the t-structure, consider the commutative square \[
\begin{tikzcd}[row sep=large, column sep=large]
\sC_T \arrow{r}{i_{\dR,*} \otimes \id_{\sC}} \arrow{d}{\coact_{\sC}} & \sC \arrow{d}{\coact_{\sC}} \\
\DD(T) \otimes \sC_T \arrow{r}[yshift=0.5em]{i_{\dR,*} \otimes i_{\dR,*} \otimes \id_{\sC}} & \DD(S) \otimes \sC.
\end{tikzcd} \]
The upper horizontal and right vertical functors are t-exact, so it suffices to show that that the lower horizontal functor is conservative and t-exact. In fact, it is fully faithful, being the tensor product of two fully faithful functors which admit continuous right adjoints. To see the t-exactness, note that this functor can be decomposed as \[ \DD(T) \otimes \sC_T \xrightarrow{i_{\dR,*} \otimes \id_{\sC_T}} \DD(S) \otimes \sC_T \xrightarrow{\id_{\DD(S)} \otimes i_{\dR,*} \otimes \id_{\sC}} \DD(S) \otimes \sC. \] Since the t-structures on $\DD(S)$ and $\DD(T)$ are compactly generated, the t-exactness of these two functors follows from that of $i_{\dR,*}$ and $i_{\dR,*} \otimes \id_{\sC} : \sC_T \to \sC$ by Lemma \ref{tensorexact}.

\end{proof}

If $\sC_T$ and $\sC$ are almost ULA generated, then we obtain an adjunction \[ (i_{\dR,*} \otimes \id_{\sC})^{\ren} : (\sC_T)^{\ren} \rightleftarrows \sC^{\ren} : (i^! \otimes \id_{\sC})^{\ren}, \] which takes place in $\DD(S)$-modules by Proposition \ref{rencrys}.

\begin{proposition}

\label{renbasechange}

Suppose that $\sC$ and $\sC_T$ are almost ULA generated (with respect to $S$ and $T$ respectively). Then the functor $(i_{\dR,*} \otimes \id_{\sC})^{\ren}$ factors through an equivalence \[ (\sC_T)^{\ren} \tilde{\longrightarrow} (\sC^{\ren})_T. \]

\end{proposition}

\begin{proof}

Note that $(\sC^{\ren})_T$ is the kernel of the monad $(j_*j^*) \otimes \id_{\sC^{\ren}}$, which fits into the commutative square \[
\begin{tikzcd}[row sep=large, column sep=large]
\sC^{\ren} \arrow{r}[yshift=0.5em]{(j_*j^*) \otimes \id_{\sC^{\ren}}} \arrow{d} & \sC^{\ren} \arrow{d} \\
\sC\arrow{r}[yshift=0.5em]{(j_*j^*) \otimes \id_{\sC}} & \sC.
\end{tikzcd} \]
Since $\sC^{\ren,+} \tilde{\to} \sC^+$, and $(\sC_T)^{\coh} \subset \sC_T$ is annihilated by $(j_*j^*) \otimes \id_{\sC}$, it follows that $(i_{\dR,*} \otimes \id_{\sC})^{\ren}$ factors through \[ i_{\dR,*} \otimes \id_{\sC^{\ren}} : (\sC^{\ren})_T \longrightarrow \sC^{\ren}. \]

To see that the resulting functor \[ (\sC_T)^{\ren} \longrightarrow (\sC^{\ren})_T \] is essentially surjective, observe that Proposition \ref{ularen} and Lemma \ref{shriekresula} together imply that $(\sC^{\ren})_T$ is ULA generated over $T$ by objects of the form $(i^! \otimes \id_{\sC^{\ren}})(c)$, where $c$ is an object in $\sC^{\ren}$ which is ULA over $S$. Now we appeal to the commutative square \[
\begin{tikzcd}[row sep=large, column sep=large]
\sC^{\ren} \arrow{r}[yshift=0.5em]{(i_{\dR,*}i^!) \otimes \id_{\sC^{\ren}}} \arrow{d} & \sC^{\ren} \arrow{d} \\
\sC\arrow{r}[yshift=0.5em]{(i_{\dR,*}i^!)\otimes \id_{\sC}} & \sC,
\end{tikzcd} \]
which shows that $((i_{\dR,*}i^!) \otimes \id_{\sC^{\ren}})(c)$ agrees with $((i_{\dR,*}i^!) \otimes \id_{\sC})(c)$ under the identification $\sC^{\ren,+} \tilde{\to} \sC^+$ (in the latter formula, we view $c$ as an object of $\sC^+$). But the object $((i_{\dR,*}i^!) \otimes \id_{\sC})(c)$ belongs to $(\sC_T)^{\coh} \subset \sC$, which shows that $(i^! \otimes \id_{\sC^{\ren}})(c)$ belongs to the essential image of $(\sC_T)^{\ren}$ as needed.

\end{proof}

\section{Preliminaries on factorization}

This section covers foundational technical material on factorization categories, algebras, and modules. What follows is mostly not original, and our purpose is to establish notation and conventions. The main exception is Proposition \ref{factcatren}, which gives conditions under which it is possible to renormalize a factorization category with respect to a t-structure.

\subsection{Ran space} We denote by $\fSet$ the category of finite sets, and write $\fSet_{\surj} \subset \fSet$ for the subcategory consisting of nonempty finite sets and surjections between them.

Recall that the Ran space of $X$ is the (set-valued) prestack defined by \[ \Map(S,\Ran) = \{ \text{nonempty subsets of } \Map(S,X_{\dR}) \}. \] Note that with this convention, the projection $\Ran \to \Ran_{\dR}$ is an isomorphism.

Given a point $\underline{x} : S \to \Ran$, i.e., $\underline{x} \subset \Map(S,X_{\dR})$, we obtain a closed subscheme \[ \Gamma_{\underline{x}} := \bigcup_{x \in \underline{x}} \Gamma_x \subset X \times S_{\cl,\red}. \] Here we take the scheme-theoretic union in $X \times S_{\cl,\red}$. It follows that $\Gamma_{\underline{x}}$ is a closed subscheme of $X \times S$, but is generally smaller than the scheme-theoretic union taken in $X \times S$; on the other hand, all of our references to $\Gamma_{\underline{x}}$ will in fact depend only on its formal completion in $X \times S$.

For any nonempty finite set $I$, there is a natural map
\begin{align*}
X^I_{\dR} &\longrightarrow \Ran \\
(x_i)_{i \in I} &\mapsto \{ x_i \ | \ i \in I \}.
\end{align*}
For any surjection $I \to J$ we have the associated diagonal map $X^J \to X^I$, which is compatible with the morphisms to $\Ran$. The resulting map \[ \underset{I \in \fSet_{\surj}^{\op}}{\colim} \ X^I_{\dR} \longrightarrow \Ran \] is an isomorphism.

Similarly, we have the \emph{lax} prestack $\Ran^{\untl}$, which parameterizes all subsets of $X_{\dR}$ (including the empty set), with morphisms given by inclusions.

\subsection{Factorization categories} We will denote by $\FactCat$ the category of unital factorization categories on $X$, with morphisms given by strictly unital factorizable functors. We can also consider the category $\FactCat^{\laxfact}$ consisting of lax factorization categories. Both of these categories admit variants in which functors are allowed to be lax unital, and together these assemble into a cartesian square \[
\begin{tikzcd}
\FactCat \arrow{r} \arrow{d} & \FactCat^{\laxfact} \arrow{d} \\
\FactCat_{\laxuntl} \arrow{r} & \FactCat^{\laxfact}_{\laxuntl}.
\end{tikzcd} \]
Here the horizontal functors are fully faithful, while the vertical functors are only $1$-fully faithful but act as the identity on objects.

\begin{rem}

We refer the reader to \S 6 of \cite{R1} for careful definitions of these categories. An alternative but equivalent set of definitions can be found in \S 10 of \cite{charles-lin} (see Remark 10.5.11 there, which compares the two).

\end{rem}

By definition, we have a forgetful functor \[ \FactCat^{\laxfact}_{\laxuntl} \longrightarrow \ShvCat(\Ran^{\untl}). \] The category $\FactCat^{\laxfact}_{\laxuntl}$ admits a natural symmetric monoidal structure which makes the forgetful functor symmetric monoidal. This restricts to a symmetric monoidal structure on each of the $1$-full subcategories of $\FactCat^{\laxfact}_{\laxuntl}$ considered above.

Given an object $\sC$ in $\FactCat^{\laxfact}_{\laxuntl}$ and a prestack $Z$ mapping to $\Ran^{\untl}$, we will write \[ \sC_Z := \Gamma(Z,\sC) \] for the resulting $\QCoh(Z)$-module category. For any finite set $I$, we use the symmetric monoidal equivalences $\QCoh(X^I) \cong \IndCoh(X^I)$ and $\QCoh(X^I_{\dR}) \cong \DD(X^I)$ to view $\sC_{X^I_{\dR}}$ as a $\DD(X^I)$-module and \[ \sC_{X^I} \cong \IndCoh(X^I) \underset{\DD(X^I)}{\otimes} \sC_{X^I_{\dR}} \] as an $\IndCoh(X^I)$-module.

\subsection{A combinatorial presentation of sheaves of categories on $\Ran^{\untl}$} Recall from \cite{R1} that for a lax prestack $Y$, there is a certain $1$-full subcategory \[ \ShvCat(Y)^{\strict} \subset \ShvCat(Y) \] (denoted there by ``naive" rather than ``strict") with the same objects but fewer morphisms. In the case of unital Ran space, the subcategory \[ \ShvCat(\Ran^{\untl})^{\strict} \subset \ShvCat(\Ran^{\untl}) \] contains only those morphisms which are ``strictly unital," i.e., \[ \FactCat^{\laxfact} = \FactCat^{\laxfact}_{\laxuntl} \underset{\ShvCat(\Ran^{\untl})}{\times} \ShvCat(\Ran^{\untl})^{\strict}. \]

From the colimit presentation \[ \Ran = \underset{I \in \fSet^{\op}_{\surj}}{\colim} X_{\dR}^I \] and the $1$-affineness of $X^I_{\dR}$, we obtain a limit presentation \[ \ShvCat(\Ran) = \lim_{I \in \fSet_{\surj}} \DD(X^I)\mod. \] Here the transition functor $\DD(X^I)\mod \to \DD(X^J)\mod$ is given by extension of scalars along the restriction functor $\DD(X^I) \to \DD(X^J)$ for any surjection $I \to J$, corresponding to inverse image of sheaves of categories along $X^J_{\dR} \to X^I_{\dR}$. The following result supplies an analogous presentation of $\ShvCat(\Ran^{\untl})^{\strict}$ as a lax limit over the category $\fSet$.

As a reminder, given an essentially small category $\sI$ and a functor
\begin{align*}
\sI &\longrightarrow \Cat \\
i & \mapsto \sC_i,
\end{align*}
the lax colimit is by definition the cocartesian fibration \[ \underset{i \in \sI}{\operatorname{lax} \, \colim} \ \sC_i \longrightarrow \sI \] obtained by unstraightening. The lax limit is the category of sections \[ \underset{i \in \sI}{\operatorname{lax} \, \lim} \ \sC_i = \Sect(\sI,\underset{i \in \sI}{\operatorname{lax} \, \colim} \ \sC_i). \]

\begin{proposition}

\label{sheafrancomb}

There is a canonical fully faithful embedding of symmetric monoidal categories \[ \ShvCat(\Ran^{\untl})^{\strict} \longrightarrow \underset{I \in \fSet}{\operatorname{lax} \, \lim} \ \DD(X^I)\mod. \] The image consists of those objects $\{ \sC_{X^I_{\dR}} \}$ such that, for any surjection $I \to J$, the functor \[ \DD(X^J) \underset{\DD(X^I)}{\otimes} \sC_{X^I_{\dR}} \longrightarrow \sC_{X^J_{\dR}} \] is an equivalence.

\end{proposition}

\begin{proof}

This is Corollary 4.6.1 in \cite{R1}.

\end{proof}

\subsection{A combinatorial presentation of factorization categories} Following \emph{loc. cit.}, we write $\Part \subset \Part_{\untl}$ for the twisted arrows categories of $\fSet_{\surj} \subset \fSet$. For any $p : I \to J$ in $\Part_{\untl}$, there is a corresponding open subscheme $U(p) \subset X^I$, and given $\epsilon : p_1 \to p_2$ in $\Part_{\untl}$, there is a unique morphism $U(\epsilon) : U(p_2) \to U(p_1)$ which makes the following square commute: \[
\begin{tikzcd}
U(p_2) \arrow{r}{U(\epsilon)} \arrow{d} & U(p_1) \arrow{d} \\
X^{I_2} \arrow{r} & X^{I_1}.
\end{tikzcd} \]
Thus we obtain a functor $U$ from $\Part_{\untl}^{\op}$ into schemes.

If we equip $\Part_{\untl}$ with the symmetric monoidal structure given by $(p_1,p_2) \mapsto p_1 \sqcup p_2$, then $U$ is canonically oplax symmetric monoidal, since \[ U(p_1 \sqcup p_2) \subset U(p_1) \times U(p_2) \] as subschemes of $X^{I_1 \sqcup I_2} = X^{I_1} \times X^{I_2}$. This endows the resulting functor
\begin{align*}
\Part_{\untl} &\longrightarrow \Cat \\
p &\mapsto \DD(U(p))\mod
\end{align*}
with a lax symmetric monoidal structure, and hence its unstraightening
\begin{equation}
\underset{p \in \Part_{\untl}}{\operatorname{lax} \, \colim} \ \DD(U(p))\mod \longrightarrow \Part_{\untl}
\label{partstraight}
\end{equation}
is naturally symmetric monoidal (in particular, the lax colimit admits a symmetric monoidal structure).

\begin{proposition}

\label{factcatcomb}

There is canonical equivalence from $\FactCat^{\laxfact}$ to the full subcategory of lax symmetric monoidal sections of (\ref{partstraight}) satisfying the following conditions:
\begin{enumerate}[(i)]
\item morphisms in $\Part \subset \Part_{\untl}$ map to cocartesian morphisms;
\item morphisms between objects of the form $\varnothing \to I$ map to cocartesian morphisms.
\end{enumerate}

\end{proposition}

\begin{proof}

This is a slight variant of Proposition-Construction 8.4.1 (cf. also \S 8.12) in \cite{R1}.

\end{proof}

\subsection{Renormalization of factorization categories} We now formulate conditions under which it is possible to renormalize a factorization category with respect to a given t-structure.

\begin{proposition}

\label{sheafranren}

Let $\sC$ be a sheaf of categories on $\Ran^{\untl}$, and suppose we are given a t-structure on $\sC_{X^I_{\dR}}$ for each finite set $I$, compatible with the $\DD(X^I)$-action. Assume also that the following conditions are satisfied:
\begin{enumerate}[(i)]
\item $\sC_{X^I_{\dR}}$ is almost ULA generated for any $I$ in $\fSet$;
\item for any surjection $I \to J$, the functor $\sC_{X^J_{\dR}} \to \sC_{X^I_{\dR}}$ left adjoint to the structure functor is t-exact;
\item for any injection $I \to J$, the structure functor $\sC_{X^I_{\dR}} \to \sC_{X^J_{\dR}}$ has finite homological amplitude.
\end{enumerate}
Then there exists a sheaf of categories $\sC^{\ren}$ on $\Ran^{\untl}$ equipped with a morphism $\sC^{\ren} \to \sC$, such that for any finite set $I$ the functor $(\sC^{\ren})_{X^I_{\dR}} \to \sC_{X^I_{\dR}}$ realizes $(\sC^{\ren})_{X^I_{\dR}}$ as the renormalization of $\sC_{X^I_{\dR}}$ constructed in Proposition \ref{rencrys}.

\end{proposition}

\begin{proof}

The almost ULA generation condition implies in particular that each $\sC_{X^I_{\dR}}$ is almost compactly generated, and conditions (ii-iii) imply that the structure functor $\DD(X^I) \to \DD(X^J)$ has finite homological amplitude for any map of finite sets $I \to J$. Applying the construction of Proposition \ref{rencrys}, we obtain an object of \[ \underset{I \in \fSet}{\operatorname{lax} \, \lim} \ \DD(X^I)\mod, \] which lies in the image of $\ShvCat(\Ran^{\untl})^{\strict}$ (cf. Proposition \ref{sheafrancomb}) by Propositions \ref{ularen} and \ref{renbasechange}.

\end{proof}

\begin{proposition}

\label{factcatren}

Let $\sC$ be an object of $\FactCat^{\laxfact}$. Suppose that we are given a t-structure on $\sC_{X^I_{\dR}}$ for any finite set $I$ satisfying the conditions of Proposition \ref{sheafranren}, and that for any finite sets $I_1,I_2$, the image of the composite functor \[ \sC_{X^{I_1}_{\dR}}^+ \times \sC_{X^{I_2}_{\dR}}^+ \longrightarrow \sC_{X^{I_1}_{\dR}} \otimes \sC_{X^{I_2}_{\dR}} \longrightarrow \DD((X^{I_1} \times X^{I_2})_{\disj}) \underset{\DD(X^{I_1 \sqcup I_2})}{\otimes} \sC_{X^{I_1 \sqcup I_2}_{\dR}} \] is contained in \[ (\DD((X^{I_1} \times X^{I_2})_{\disj}) \underset{\DD(X^{I_1 \sqcup I_2})}{\otimes} \sC_{X^{I_1 \sqcup I_2}_{\dR}})^+. \] Then the sheaf of categories $\sC^{\ren}$ on $\Ran^{\untl}$ from Proposition \ref{sheafranren} naturally lifts to an object of $\FactCat^{\laxfact}$.

\end{proposition}

\begin{proof}

Apply Proposition \ref{factcatcomb} to view $\sC$ as a lax symmetric monoidal section of (\ref{partstraight}). By Proposition \ref{opencoh}, for any object $p : I \to J$ in $\Part_{\untl}$, the category \[ \sC_{U(p)} = \DD(U(p)) \underset{\DD(X^I)}{\otimes} \sC_{X^I_{\dR}} \] is almost ULA generated over $U(p)$ with respect to the unique t-structure that makes $\sC_{X^I_{\dR}} \to \sC_{U(p)}$ t-exact. Given objects $p_1 : I_1 \to J_1$, $p_2 : I_2 \to J_2$ and a morphism $\epsilon : p_1 \to p_2$ in $\Part_{\untl}$, we claim that the corresponding structure functor $\sC_{U(p_1)} \to \sC_{U(p_2)}$ has finite homological amplitude. This follows from hypotheses (ii-iii) of Proposition \ref{sheafranren} and the commutative square \[
\begin{tikzcd}
\sC_{X^{I_1}_{\dR}} \arrow{r} \arrow{d} & \sC_{X^{I_2}_{\dR}} \arrow{d} \\
\sC_{U(p_1)} \arrow{r} & \sC_{U(p_2)},
\end{tikzcd} \]
since the vertical functors are t-exact with fully faithful and left t-exact right adjoints. Similarly, the structure functor $\sC_{U(p_1)} \otimes \sC_{U(p_2)} \to \sC_{U(p_1 \sqcup p_2)}$ coming from the lax symmetric monoidal structure sends $\sC_{U(p_1)}^+ \times \sC_{U(p_2)}^+$ into $\sC_{U(p_1 \sqcup p_2)}^+$, which we can see from the commutative square \[
\begin{tikzcd}
\sC_{X^{I_1}_{\dR}} \otimes \sC_{X^{I_1}_{\dR}} \arrow{r} \arrow{d} & \DD((X^{I_1} \times X^{I_2})_{\disj}) \underset{\DD(X^{I_1 \sqcup I_2})}{\otimes} \sC_{X^{I_1 \sqcup I_2}_{\dR}} \arrow{d} \\
\sC_{U(p_1)} \otimes \sC_{U(p_2)} \arrow{r} & \sC_{U(p_1 \sqcup p_2)}.
\end{tikzcd} \]

Combining the above observations, it follows that applying (\ref{renfunctorial}) termwise to the categories $\sC_{U(p)}$ yields another lax symmetric monoidal section of (\ref{partstraight}). This section satisfies condition (i) of Proposition \ref{factcatcomb} by Proposition \ref{renbasechange}, and condition (ii) is trivial to verify.

\end{proof}

Let $(\FactCat^{\laxfact})^{\otimes}_{\tstr}$ be the pseudo-tensor category whose objects consist of $\sC$ in $\FactCat^{\laxfact}$ together with a t-structure on $\sC_{X^I_{\dR}}$ for every finite set $I$, and whose $n$-ary morphisms are the same as those in the symmetric monoidal category $\FactCat^{\laxfact}$, i.e., we do not impose any t-exactness conditions on functors.

Denote by \[ (\FactCat^{\laxfact})^{\otimes}_{\aULA} \subset (\FactCat^{\laxfact})^{\otimes}_{\tstr} \] the pseudo-tensor subcategory whose objects $\sC$ satisfy the conditions of Propositions \ref{sheafranren} and \ref{factcatren}, and whose $n$-ary morphisms $(\sC_1,\cdots \!,\sC_n) \to \sD$ satisfy the condition that the image of \[ (\sC_1)_{X^I_{\dR}}^+ \times \cdots \times (\sC_n)_{X^I_{\dR}}^+ \longrightarrow (\sC_1)_{X^I_{\dR}} \underset{\DD(X^I)}{\otimes} \cdots \underset{\DD(X^I)}{\otimes} (\sC_n)_{X^I_{\dR}} \longrightarrow \sD_{X^I_{\dR}} \] is contained in $\sD_{X^I_{\dR}}^+$ for any finite set $I$.

\begin{proposition}

\label{monfactcatren}

The assignment $\sC \mapsto \sC^{\ren}$ of Proposition \ref{factcatren} naturally lifts to a morphism of pseudo-tensor categories \[ (\FactCat^{\laxfact})^{\otimes}_{\aULA} \longrightarrow (\FactCat^{\laxfact})^{\otimes}_{\aULA}. \]

\end{proposition}

\begin{proof}

This is immediate, since the construction of Proposition \ref{factcatren} is induced by the morphism of pseudo-tensor categories (\ref{renfunctorial}).

\end{proof}

\subsection{Dualizable factorization categories} The following is an explicit dualizability criterion for factorization categories.

\begin{proposition}

Suppose that $\sC$ in $\FactCat$ has the property that $\sC_{X^I_{\dR}}$ is dualizable as a $\DD(X^I)$-module for any finite set $I$. Assume also that for any injection $J \to I$, the unital structure functor \[ \DD(X^I) \underset{\DD(X^J)}{\otimes} \sC_{X^J_{\dR}} \longrightarrow \sC_{X^I_{\dR}} \] admits a $\DD(X^I)$-linear right adjoint. Then $\sC$ is dualizable as an object of $\FactCat_{\laxuntl}$.

\label{dualfactprop}

\end{proposition}

\begin{proof}

First, observe that the hypothesis in the proposition implies that \[ \DD(X^I) \underset{\DD(X^J)}{\otimes} \sC_{X^J_{\dR}} \longrightarrow \sC_{X^I_{\dR}} \] admits a $\DD(X^I)$-linear right adjoint for any (not necessarily injective) map $I \to J$ in $\fSet$. Namely, any such map can written as the composition of a surjection and an injection, and the functor in question is an equivalence if $I \to J$ is surjective. It follows that the sheaf of categories underlying $\sC$ is dualizable as an object of $\ShvCat(\Ran^{\untl})$, with dual object given by dualizing termwise and then passing to left adjoints of the structure functors.

A lax factorization structure on $\sC$ induces an oplax factorization structure on $\sC^{\vee}$, and if the factorization structure on $\sC$ is strict then so is the one on $\sC^{\vee}$. It is not difficult to see that the evaluation \[ \sC^{\vee} \otimes \sC \longrightarrow \Vect \] and coevaluation \[ \Vect \longrightarrow \sC^{\vee} \otimes \sC \] morphisms lift from $\ShvCat(\Ran^{\untl})$ to $\FactCat_{\laxuntl}$.

\end{proof}

\begin{rem}

We emphasize that it is necessary to take the dual of $\sC$ in $\FactCat_{\laxuntl}$ rather than $\FactCat$. The point is that the coevaluation functor \[ \Vect \longrightarrow \sC^{\vee} \otimes \sC \] is typically only lax unital.

\end{rem}

\subsection{Commutative factorization categories}\label{comfactcat} As shown in \cite{R1} \S 7.5, one can functorially attach an object of $\FactCat$ to any symmetric monoidal DG category. We denote the associated factorization category by the same symbol to avoid cluttering the notation.

Let $\sC$ be a symmetric monoidal category. For any finite set $I$, we have a symmetric monoidal functor \[ \Loc_{X^I} : \sC^{\otimes I} \longrightarrow \sC_{X^I_{\dR}}, \] defined in Remark 6.8.4 of \cite{R2}.

\begin{proposition}

\label{ulagenprop}

Suppose that $\sC$ is a rigid symmetric monoidal category whose underlying DG category is compactly generated.
\begin{enumerate}[(i)]

\item The functor $\Loc_{X^I}$ sends compact objects to objects ULA over $X^I$. If $\{ c_{\alpha} \}$ is a set of compact generators for $\sC^{\otimes I}$, then $\{ \Loc_{X^I}(c_{\alpha}) \}$ is a set of ULA generators for $\sC_{X^I_{\dR}}$.

\item The factorization category attached to $\sC$ is dualizable and self-dual as an object of $\FactCat_{\laxuntl}$.

\end{enumerate}

\end{proposition}

\begin{proof}

For (i), see Proposition 6.16.1 and the proof of Theorem 6.7.1 in \emph{loc. cit.}

By Corollary 6.18.2 of \cite{R2}, the $\DD(X^I)$-module category $\sC_{X^I_{\dR}}$ is dualizable and self-dual for any finite set $I$. According to Proposition \ref{dualfactprop}, to prove (ii) it remains to show that for any injection $J \to I$, the structure map \[ \DD(X^I) \underset{\DD(X^J)}{\otimes} \sC_{X^J_{\dR}} \longrightarrow \sC_{X^I_{\dR}} \] admits a $\DD(X^I)$-linear right adjoint, or equivalently preserves ULA objects. This follows from the commutative square \[
\begin{tikzcd}
\sC^{\otimes J} \arrow{r} \arrow{d}{\id \otimes \Loc_{X^J}} & \sC^{\otimes I} \arrow{d}{\Loc_{X^I}} \\
\DD(X^I) \underset{\DD(X^J)}{\otimes} \sC_{X^J_{\dR}} \arrow{r} & \sC_{X^I_{\dR}},
\end{tikzcd} \]
where the upper horizontal arrow acts as the identity on $\sC^{\otimes J}$ and inserts the unit in $\sC^{\otimes (I \setminus J)}$.

\end{proof}

The following lemma will also be useful.

\begin{lemma}

\label{rigidrightadj}

Let $F : \sC \to \sD$ be a symmetric monoidal functor. Assume that $\sC$ is rigid and compactly generated, and that the unit object in $\sD$ is compact, so that in particular the right adjoint $F^{\RR}$ of $F$ is continuous. Then the morphism in $\FactCat$ arising from $F$ admits a right adjoint in $\FactCat_{\laxuntl}$, which in particular agrees with $F^{\RR}$ fiberwise.

\end{lemma}

\begin{proof}

It suffices to prove that for any finite set $I$, the $\DD(X^I)$-linear functor \[ F : \sC_{X^I_{\dR}} \longrightarrow \sD_{X^I_{\dR}} \] preserves ULA objects. Using the commutative square \[
\begin{tikzcd}
\sC^{\otimes I} \arrow{r}{F^{\otimes I}} \arrow{d}{\Loc_{X^I}} & \sD^{\otimes I} \arrow{d}{\Loc_{X^I}} \\
\sC_{X^I_{\dR}} \arrow{r}{F} & \sD_{X^I_{\dR}}
\end{tikzcd} \]
and Proposition \ref{ulagenprop}, it is enough to prove that the composition \[ \sC^{\otimes I} \xrightarrow{F^{\otimes I}} \sD^{\otimes I} \xrightarrow{\Loc_{X^I}} \sD_{X^I_{\dR}} \] sends compact objects to objects ULA over $X^I$. This functor is symmetric monoidal, and the proof of Lemma 6.16.2 in \cite{R2} shows that the unit object in $\sD_{X^I_{\dR}}$ is ULA, so the claim follows from the rigidity of $\sC^{\otimes I}$.

\end{proof}

\subsection{Factorization algebras}\label{factalgsec} Fix $\sC$ in $\FactCat^{\laxfact}$. The category $\FactAlg(\sC)$ of (unital) factorization algebras in $\sC$ has several equivalent definitions available in the literature; for our use in the sequel, we will record still another. It is a unital variant of the definition given in \cite{FG}.

The lax prestack $\Ran^{\untl}$ is a commutative monoid in correspondences under the binary operation given by \[
\begin{tikzcd}
& (\Ran^{\untl} \times \Ran^{\untl})_{\disj} \arrow{dl}[swap]{j} \arrow{dr}{u \circ j} & \\
\Ran^{\untl} \times \Ran^{\untl} & & \Ran^{\untl}.
\end{tikzcd} \]
Here \[ (\Ran^{\untl} \times \Ran^{\untl})_{\disj} \subset \Ran^{\untl} \times \Ran^{\untl} \] is the open subspace consisting of pairs $(\underline{x},\underline{x}')$ such that \[ \Gamma_{\underline{x}} \cap \Gamma_{\underline{x}'} = \varnothing, \] $j$ is the open inclusion, and $u : \Ran^{\untl} \times \Ran^{\untl} \to \Ran^{\untl}$ is the operation of union $(\underline{x},\underline{x}') \mapsto \underline{x} \cup \underline{x}'$.

Note that $u \circ j$ is fibered in groupoids (as opposed to categories), and is moreover schematic \'{e}tale. It follows that we have a well-behaved direct image functor along $u \circ j$, and hence the category $\sC_{\Ran^{\untl}}$ acquires a symmetric monoidal structure $\otimes^{\ch}$ with binary operation given by pull-push along the above correspondence.

Then $\FactAlg(\sC)$ can be realized as the following full subcategory of $\CocomCoalg^{\otimes^{\ch}}(\sC_{\Ran^{\untl}})$. Given $A$ in the latter category, the binary operation is a morphism \[ A \longrightarrow (u \circ j)_*j^!(A \boxtimes A), \] and since $u \circ j$ is \'{e}tale this corresponds by adjunction to a morphism \[ (u \circ j)^!A \longrightarrow j^!(A \boxtimes A). \] The object $A$ belongs to $\FactAlg(\sC)$ if and only if the latter map is an isomorphism.

Suppose $\sC$ that is a constant commutative factorization category, i.e., it arises from the same-named symmetric monoidal DG category as in \S\ref{comfactcat}. For any commutative algebra $A$ in the symmetric monoidal category $\sC$, there is an associated factorization algebra $A_{\Ran}$ in the factorization category $\sC$, which we will simply denote by $A$ when no confusion is likely to result. Note that $A_{\Ran}|^!_X = A \otimes \omega_X$.

\subsection{Factorization modules} We continue to let $\sC$ denote a fixed object of $\Fact^{\laxfact}$.

Given a map of (lax) prestacks $p : Z \to \Ran^{\untl}$, we denote by $\Ran^{\untl}_Z$ the associated factorization module space. Explicitly, it is the lax prestack defined as follows: a point $S \to \Ran^{\untl}_Z$ consists of a pair $(z,\underline{x})$ where $z : S \to Z$, $\underline{x} : S \to \Ran^{\untl}$, and $p(z) \subset \underline{x}$, with a unique morphism $(z,\underline{x}) \to (z',\underline{x}')$ if $z = z'$ and $\underline{x} \subset \underline{x}'$. The space $\Ran_Z^{\untl}$ admits a natural action of $\Ran^{\untl}$ in correspondences with respect to the above monoid structure. Namely, the action is given by the correspondence \[
\begin{tikzcd}
& (\Ran^{\untl} \times \Ran^{\untl}_Z)_{\disj} \arrow{dl}[swap]{j} \arrow{dr}{u \circ j} & \\
\Ran^{\untl} \times \Ran^{\untl}_Z & & \Ran^{\untl}_Z,
\end{tikzcd} \]
where we have abused notation slightly by using the same symbols for the obvious variants of the maps $u$ and $j$.

Again $u \circ j$ is schematic \'{e}tale, whence $\sC_{\Ran^{\untl}}$ acts on $\sC_{\Ran^{\untl}_Z}$ with respect to $\otimes^{\ch}$, with the action map given by pull-push along the above correspondence. The category of factorization $A$-modules $A\modfact(\sC)_Z$ over $Z$ can be realized as the full subcategory of \[ A\comod^{\otimes^{\ch}}(\sC_{\Ran^{\untl}_Z}) \] consisting of those $M$ such that the map \[ (u \circ j)^!M \longrightarrow j^!(A \boxtimes M) \] determined by the comodule structure is an isomorphism. The projection $\Ran^{\untl}_Z \to Z$ induces an action of $\QCoh(Z)$ on $A\modfact(\sC)_Z$.

Writing $\Delta_Z : Z \to \Ran^{\untl}_Z$ for the diagonal map, we think of the composite \[ \oblv_A : A\modfact(\sC)_Z \xrightarrow{\oblv} \sC_{\Ran^{\untl}_Z} \xrightarrow{\Delta^!_Z} \sC_Z \] as the forgetful functor; in particular, it is conservative. For $A = \unit_{\sC}$, the functor \[ \oblv_{\unit_{\sC}} : \unit_{\sC}\modfact(\sC) \longrightarrow \sC \] is an equivalence.

Letting $Z$ vary over all affine schemes mapping to $\Ran^{\untl}$, the categories $A\modfact(\sC)_Z$ assemble into an object $A\modfact(\sC)$ of $\FactCat^{\laxfact}$, with the (lax) factorization structure given by ``external fusion" (cf. \S 8.14 of \cite{R1}). 

There is a canonical object $\Vac_{A,Z}$ in $A\modfact(\sC)_Z$ with \[ \oblv_A(\Vac_{A,Z}) = A_Z := A|^!_Z\] in $\sC_Z$, called the vacuum factorization $A$-module over $Z$. The object $\Vac_{A,\Ran^{\untl}}$, which we will denote simply by $\Vac_A$ when there is no risk of confusion, is the unit for the factorization structure on $A\modfact(\sC)$, and in particular has a natural structure of factorization algebra in $A\modfact(\sC)$.

\subsection{Functoriality of factorization modules} Consider the category $(\FactCat_{\laxuntl}^{\laxfact})^{\FactAlg^{\op}}$. At the level of objects and $1$-morphisms, it consists of pairs $(\sC,A)$ with $\sC$ in $\FactCat^{\laxfact}$ and $A$ in $\FactAlg(\sC)$, and a morphism $(\sC,A) \to (\sD,B)$ is given by $F : \sC \to \sD$ in $\FactCat_{\laxuntl}^{\laxfact}$ and $B \to F(A)$ in $\FactAlg(\sD)$.

The construction $(\sC,A) \mapsto A\modfact(\sC)$ extends to a functor
\begin{equation}
(\FactCat_{\laxuntl}^{\laxfact})^{\FactAlg^{\op}} \longrightarrow \FactCat^{\laxfact}_{\laxuntl}.
\label{factmodfunct}
\end{equation}
In particular, given a morphism $(\sC,A) \to (\sD,B)$ consisting of $F : \sC \to \sD$ and $B \to F(A)$, the functor obtained by applying (\ref{factmodfunct}) can be expressed as a composition \[ A\mod^{\fact}(\sC) \longrightarrow F(A)\mod^{\fact}(\sD) \longrightarrow B\mod^{\fact}(\sD). \] Here the first functor makes the square \[
\begin{tikzcd}
A\mod^{\fact}(\sC) \arrow{r} \arrow{d}{\oblv_A} & F(A)\mod^{\fact}(\sD) \arrow{d}{\oblv_{F(A)}} \\
\sC \arrow{r}{F} & \sD
\end{tikzcd} \]
commute, and the second functor is restriction of scalars along $B \to F(A)$ (cf. \cite{charles-lin} \S 10.6).

\subsection{A transitivity principle} Fix $\sC$ in $\FactCat^{\laxfact}$ and a morphism $A \to B$ in $\FactAlg(\sC)$. Apply (\ref{factmodfunct}) to the morphism \[  (B\modfact(\sC),\Vac_B) \longrightarrow (A\modfact(\sC),\res^B_A(\Vac_B)) \] to obtain a morphism
\begin{equation}
\label{factmodtranseq}
B\modfact(\sC) \cong \Vac_B\modfact(B\modfact(\sC)) \longrightarrow \res^B_A(\Vac_B)\modfact(A\modfact(\sC))
\end{equation}
in $\FactCat^{\laxfact}$.

\begin{lemma}

\label{factmodtrans}

The functor (\ref{factmodtranseq}) is an equivalence.

\end{lemma}

\begin{proof}

Note that the forgetful functor $\oblv_A : A\modfact(\sC) \to \sC$ sends $\res^B_A(\Vac_B) \mapsto B$. Applying (\ref{factmodfunct}) to the resulting morphism \[ (A\modfact(\sC),\res^B_A(\Vac_B)) \longrightarrow (\sC,B) \] yields the inverse to (\ref{factmodtranseq}).

\end{proof}

\subsection{Functoriality of monoidal structures} We observe that $(\FactCat_{\laxuntl}^{\laxfact})^{\FactAlg^{\op}}$ admits a natural symmetric monoidal structure, given on objects by \[ (\sC,A) \otimes (\sD,B) = (\sC \otimes \sD,A \boxtimes B). \] Moreover, the functor (\ref{factmodfunct}) is naturally lax symmetric monoidal: in particular, for $(\sC,A)$ and $(\sD,B)$ in $(\FactCat_{\laxuntl}^{\laxfact})^{\FactAlg^{\op}}$ we have a canonical morphism \[ A\mod^{\fact}(\sC) \otimes B\mod^{\fact}(\sD) \longrightarrow (A \boxtimes B)\mod^{\fact}(\sC \otimes \sD) \] in $\FactCat^{\laxfact}$, which intertwines the forgetful functors to $\sC \otimes \sD$.

Observe that the projection $(\FactCat_{\laxuntl}^{\laxfact})^{\FactAlg^{\op}} \to \FactCat_{\laxuntl}^{\laxfact}$ is symmetric monoidal, and in particular induces a functor \[ \AssocAlg((\FactCat_{\laxuntl}^{\laxfact})^{\FactAlg^{\op}}) \longrightarrow \AssocAlg(\FactCat_{\laxuntl}^{\laxfact}). \] This functor admits a canonical section
\begin{equation}
\AssocAlg(\FactCat_{\laxuntl}^{\laxfact}) \longrightarrow \AssocAlg((\FactCat_{\laxuntl}^{\laxfact})^{\FactAlg^{\op}})
\label{monfactcatsect}
\end{equation}
given on objects by $\sA \mapsto (\sA,\e_{\sA})$, where $\e_{\sA}$ denotes the monoidal unit of $\sA$ viewed as a factorization algebra.

The lax symmetric monoidal structure on (\ref{factmodfunct}) gives rise to a functor \[ \AssocAlg((\FactCat_{\laxuntl}^{\laxfact})^{\FactAlg^{\op}}) \longrightarrow \AssocAlg(\FactCat^{\laxfact}_{\laxuntl}). \] Note that the composition of this functor with (\ref{monfactcatsect}) factors through the $1$-full subcategory \[ \AssocAlg(\FactCat^{\laxfact}) \longrightarrow \AssocAlg(\FactCat^{\laxfact}_{\laxuntl}), \] and hence can be viewed as a functor
\begin{align}
\label{factmodunit}
\AssocAlg(\FactCat_{\laxuntl}^{\laxfact}) &\longrightarrow \AssocAlg(\FactCat^{\laxfact}) \\
\sA &\mapsto \e_{\sA}\modfact(\sA). \nonumber
\end{align}
Namely, the factorization unit and the monoidal unit in $\e_{\sA}\modfact(\sA)$ coincide, both being the vacuum module $\Vac_{\e_{\sA}}$.

\subsection{Unital vs. non-unital factorization algebras}\label{nonuntlalgsec} In \S\S\ref{nonuntlalgsec}-\ref{chirenvsec}, we fix $\sC$ in $\FactCat^{\laxfact}$.

We will write \[ \Ran_{\varnothing} := \Ran \sqcup \{ \varnothing \}, \] which is the monoid obtained from the semigroup $\Ran$ by adjoining an identity element. Note that unital factorization algebras on $\Ran_{\varnothing}$ are equivalent to non-unital factorization algebras on $\Ran$.

Consider the canonical morphism of lax prestacks $\phi : \Ran_{\varnothing} \to \Ran^{\untl}$. The inverse image functor \[ \phi^! : \sC_{\Ran^{\untl}} \longrightarrow \sC_{\Ran_{\varnothing}} \] induces a forgetful functor \[ \OblvUnit : \FactAlg(\sC) \longrightarrow \FactAlg^{\nonuntl}(\sC). \]

\begin{proposition}

\label{addunitprop}

The functor $\OblvUnit$ admits a left adjoint $\AddUnit$. The comonad \[ \AddUnit \circ \OblvUnit : \FactAlg(\sC) \longrightarrow \FactAlg(\sC) \] is given by $A \mapsto A \times \unit_{\sC}$.

\end{proposition}

\begin{proof}

For the existence of the left adjoint $\AddUnit$, we apply Proposition 4.4.2 of \cite{atiyah-bott}, which says that\footnote{We warn the reader that the version of $\Ran^{\untl}$ used in \emph{loc. cit.} does not include $\varnothing$, which causes some formulas to appear different from ours.} $\phi^!$ admits a left adjoint. Both functors are compatible with factorization structures and therefore lift to the desired adjunction \[ \AddUnit : \FactAlg^{\nonuntl}(\sC) \rightleftarrows \FactAlg(\sC) : \OblvUnit. \]

Given a unital factorization algebra $A$ in $\sC$, there is a natural morphism \[ \AddUnit(\OblvUnit(A)) \longrightarrow \unit_{\sC} \] in $\FactAlg(\sC)$ corresponding to the zero morphism \[ \OblvUnit(A) \longrightarrow \OblvUnit(\unit_{\sC}) \] in $\FactAlg^{\nonuntl}(\sC)$. To show that the resulting morphism \[ \AddUnit(\OblvUnit(A)) \longrightarrow A \times \unit_{\sC} \] is an isomorphism, it suffices to check that its restriction along the main diagonal is an isomorphism in $\sC_X$. The (appropriately modified) formula in Proposition 4.3.6 of \emph{loc. cit.} shows that \[ \AddUnit(\OblvUnit(A))|^!_X = A_X \times \unit_{\sC_X}, \] as needed.

\end{proof}

\subsection{Unital vs. non-unital factorization modules} Fix $Z \to \Ran$.

\begin{proposition}

\label{untlmodequiv}

For any $A$ in $\FactAlg^{\nonuntl}(\sC)$, the composition
\begin{align*}
\AddUnit(A)\modfact(\sC)_Z &\longrightarrow \OblvUnit(\AddUnit(A))\mod^{\fact,\nonuntl}(\sC)_Z \\
&\longrightarrow A\mod^{\fact,\nonuntl}(\sC)_Z
\end{align*}
is an equivalence.

\end{proposition}

\begin{proof}

Similarly to the proof of Proposition \ref{addunitprop}, we see that \[ \phi_Z^! : \sC_{\Ran^{\untl}_Z} \longrightarrow \sC_{\Ran_Z} \] admits a left adjoint. This left adjoint is moreover compatible with factorization $\sC$-module structures and the operation $\AddUnit$, hence lifts a functor \[ A\mod^{\fact,\nonuntl}(\sC)_Z \longrightarrow \AddUnit(A)\modfact(\sC)_Z \] inverse to the functor in the proposition.

\end{proof}

\begin{lemma}

For any $A_1$, $A_2$ in $\FactAlg(\sC)$, the functor \[ \res^{A_1 \times A_2}_{A_1} \oplus \res^{A_1 \times A_2}_{A_2} : A_1\modfact(\sC)_Z \times A_2\modfact(\sC)_Z \longrightarrow (A_1 \times A_2)\modfact(\sC)_Z \] is an equivalence.

\end{lemma}

\begin{proof}

A formal argument, valid for comodules over any coalgebra, shows that \[ \res^{A_1 \times A_2}_{A_1} \oplus \res^{A_1 \times A_2}_{A_2} : A_1\comod^{\otimes^{\ch}}(\sC_{\Ran_Z}) \times A_2\comod^{\otimes^{\ch}}(\sC_{\Ran_Z}) \longrightarrow (A_1 \times A_2)\comod^{\otimes^{\ch}}(\sC_{\Ran_Z}) \] is an equivalence. One then checks that the factorization condition is preserved under this equivalence.

\end{proof}

Combining the lemma with Propositions \ref{addunitprop} and \ref{untlmodequiv}, we obtain the following useful result.

\begin{proposition}

\label{oblvunitff}

For any $A$ in $\FactAlg(\sC)$, the forgetful functor \[ A\modfact(\sC)_Z \longrightarrow \OblvUnit(A)\mod^{\fact,\nonuntl}(\sC)_Z \] factors canonically as \[
\begin{tikzcd}
A\modfact(\sC)_Z \arrow{r}{(\id,0)} \arrow{dr} & A\modfact(\sC)_Z \times \sC_Z \arrow{d} \\
& \OblvUnit(A)\mod^{\fact,\nonuntl}(\sC)_Z
\end{tikzcd} \]
where the vertical functor is an equivalence. In particular, this forgetful functor is fully faithful.

\end{proposition}

\subsection{Chiral algebras} Replacing $\Ran^{\untl}$ by $\Ran$ and repeating the relevant portions of \S\ref{factalgsec} \emph{mutatis mutandi}, we obtain a non-unital symmetric monoidal structure $\otimes^{\ch}$ on $\sC_{\Ran}$ (cf. \cite{FG}). In particular, we can realize the category of non-unital factorization algebras in $\sC$ as the full subcategory \[ \FactAlg^{\nonuntl}(\sC) \subset \CocomCoalg^{\otimes^{\ch},\nonuntl}(\sC_{\Ran}) \] of coalgebras satisfying the factorization condition, as in \emph{loc. cit.}

On the other hand, we can consider the category of Lie algebras $\LieAlg^{\otimes^{\ch}}(\sC_{\Ran})$. The category of chiral algebras is defined by \[ \LieAlg^{\ch}(\sC) := \LieAlg^{\otimes^{\ch}}(\sC_{\Ran}) \underset{\sC_{\Ran}}{\times} \sC_{X_{\dR}}, \] i.e., it is the full subcategory of $\LieAlg^{\otimes^{\ch}}(\sC_{\Ran})$ consisting of Lie algebras whose underlying object of $\sC_{\Ran}$ is supported on the main diagonal $\Delta : X_{\dR} \to \Ran$.

\begin{rem}

From our perspective chiral algebras, being Lie algebras, are necessarily viewed as non-unital.

\end{rem}

In \emph{loc. cit.}, it is shown that the functor \[ \Delta^![-1] : \sC_{\Ran} \longrightarrow \sC_{X_{\dR}} \] lifts to an equivalence \[ \FactAlg^{\nonuntl}(\sC) \tilde{\longrightarrow} \LieAlg^{\ch}(\sC), \] which will be used frequently below. The inverse is given by the (non-unital) homological Chevalley complex with respect to $\otimes^{\ch}$, which we denote by $B \mapsto \CC^{\ch}_+(B)$.

\subsection{Chiral modules} For any $Z \to \Ran$, we can define the non-unital factorization module space $\Ran_Z$ similarly to $\Ran^{\untl}_Z$. The category $\sC_{\Ran_Z}$ is a non-unital module for $\sC_{\Ran}$ with respect to $\otimes^{\ch}$. In particular, given a non-unital factorization algebra $A$, the category of non-unital factorization $A$-modules over $Z$ can be realized as the full subcategory \[ A\mod^{\fact,\nonuntl}(\sC)_Z \subset A\comod^{\otimes^{\ch},\nonuntl}(\sC_{\Ran_Z}) \] consisting of comodules satisfying the factorization condition, analogously to the unital case.

Given a chiral algebra $B$ in $\sC$, the category of chiral $B$-modules over $Z$ is by definition \[ B\mod^{\ch}(\sC)_Z := B\mod^{\otimes^{\ch}}(\sC_{\Ran_Z}) \underset{\sC_{\Ran_Z}}{\times} \sC_Z, \] i.e., it is the full subcategory of $B\mod^{\otimes^{\ch}}(\sC_{\Ran_Z})$ consisting of those modules whose underlying object of $\sC_{\Ran_Z}$ is supported on the diagonal $\Delta_Z : Z \to \Ran_Z$.

It is shown in \cite{FG} that for any $A$ in $\FactAlg^{\nonuntl}(\sC)$, the functor \[ \Delta_Z^! : \sC_{\Ran_Z} \longrightarrow \sC_Z \] lifts to a $\QCoh(Z)$-linear equivalence \[ A\mod^{\fact,\nonuntl}(\sC)_Z \tilde{\longrightarrow} (\Delta^!A[-1])\mod^{\ch}(\sC)_Z. \] The inverse is given by $M \mapsto \CC^{\ch}_+(\Delta^!A[-1],M)$, the (non-unital) homological Chevalley complex of the chiral algebra $\Delta^!A[-1]$ with coefficients.

\subsection{Lie-$*$ algebras and modules} The union map $u : \Ran \times \Ran \to \Ran$ endows $\Ran$ with another commutative semigroup structure. If $\sC$ is a \emph{commutative} factorization category, then there is an associated non-unital symmetric monoidal structure on $\sC_{\Ran}$ denoted by $\otimes^*$. The category of Lie-$*$ algebras in $\sC$ is defined by \[ \LieAlg^*(\sC) := \LieAlg^{\otimes^*}(\sC_{\Ran}) \underset{\sC_{\Ran}}{\times} \sC_{X_{\dR}}. \]

Given $Z \to \Ran$, the category $\sC_{\Ran_Z}$ acquires a structure of non-unital $\sC_{\Ran}$-module with respect to $\otimes^*$. For any $L$ in $\LieAlg^*(\sC)$, we can therefore consider the $\QCoh(Z)$-module category \[ L\mod^*(\sC)_Z := L\mod^{\otimes^*}(\sC_{\Ran_Z}) \underset{\sC_{\Ran_Z}}{\times} \sC_Z. \] Like factorization modules, these assemble into an object $L\mod^*(\sC)$ of $\FactCat^{\laxfact}$. This object has the feature that the functor \[ \triv_L : \sC \longrightarrow  L\mod^*(\sC) \] assigning the trivial $L$-module structure is a morphism in $\FactCat^{\laxfact}$, and in particular the unit object in $L\mod^*(\sC)$ is $\triv_L(\unit_{\sC})$.

\subsection{Chiral envelope}\label{chirenvsec} There is a natural map $\otimes^* \to \otimes^{\ch}$, which gives rise to a forgetful functor \[ \oblv^{\ch \to *} : \LieAlg^{\ch}(\sC) \longrightarrow \LieAlg^*(\sC). \] This functor admits a left adjoint \[ \UU^{* \to \ch} : \LieAlg^*(\sC) \longrightarrow \LieAlg^{\ch}(\sC), \] the functor of chiral enveloping algebra.

For any $B$ in $\LieAlg^{\ch}(\sC)$ and $Z \to \Ran$, the map $\otimes^* \to \otimes^{\ch}$ gives rise to a $\QCoh(Z)$-linear forgetful functor \[ B\mod^{\ch}(\sC)_Z \longrightarrow \oblv^{\ch \to *}(B)\mod^*(\sC)_Z. \] For $L$ in $\LieAlg^*(\sC)$, the composite \[ \oblv_L^{\ch \to *} : \UU^{* \to \ch}(L)\mod^{\ch}(\sC)_Z \longrightarrow \oblv^{\ch \to *}(\UU^{* \to \ch}(L))\mod^*(\sC))_Z \xrightarrow{\res} L\mod^*(\sC)_Z \] admits a left adjoint \[ \ind_L^{* \to \ch} : L\mod^*(\sC)_Z \longrightarrow \UU^{* \to \ch}(L)\mod^{\ch}(\sC)_Z. \]

We define \[ \UU^{* \to \fact}(L) := \AddUnit(\CC_+^{\ch}(\UU^{* \to \ch}(L))), \] an object in $\FactAlg(\sC)$. By Proposition \ref{untlmodequiv}, we have an equivalence \[ \UU^{* \to \ch}(L)\mod^{\ch}(\sC)_Z \tilde{\longrightarrow} \UU^{* \to \fact}(L)\modfact(\sC)_Z, \] and in particular as $Z$ varies we can upgrade $\UU^{* \to \ch}(L)\mod^{\ch}(\sC)$ to an object of $\FactCat^{\laxfact}$. With respect to this factorization structure, the functors above lift to an adjunction \[ \ind_L^{* \to \ch} : L\mod^*(\sC) \rightleftarrows \UU^{* \to \ch}(L)\mod^{\ch}(\sC) : \oblv^{\ch \to *} \] in $\FactCat^{\laxfact}_{\laxuntl}$, with $\ind^{* \to \ch}$ being strictly unital.

\section{A local acyclity theorem}

The main goal of this section is to prove Theorem \ref{vacacthm}, which says in particular that the vacuum factorization module for a connective commutative algebra almost of finite type is almost ULA. From this we deduce Corollary \ref{factmodcoh}, which says that the category of factorization modules over a connective commutative algebra of finite type satisfies the conditions of Proposition \ref{factcatren} and hence can be renormalized, provided that the same is true for the category of commutative modules over that algebra.

\subsection{Representations as a factorization category} Below we collect some facts about the factorization category $\Rep(H)$, where $H$ is an algebraic group.

The forgetful functor $\oblv_H : \Rep(H) \to \Vect$ is symmetric monoidal, and hence gives rise to a morphism in $\FactCat$.

\begin{proposition}

For any finite set $I$, the functor \[ \oblv_H : \Rep(H)_{X^I_{\dR}} \longrightarrow \DD(X^I) \] is conservative, and there is a unique t-structure on $\Rep(H)_{X^I_{\dR}}$ which makes this functor t-exact. The morphism $\oblv_H$ admits a right adjoint $\coind_H$ in $\FactCat_{\laxuntl}$, and for any finite set $I$, the functor \[ \coind_H : \DD(X^I) \longrightarrow \Rep(H)_{X^I_{\dR}} \] is t-exact.

\end{proposition}

\begin{proof}

See Propositions 6.22.1 and 6.24.1 of \cite{R2}.

\end{proof}

\begin{proposition}

\label{repulagen}

The factorization category $\Rep(H)$, with the t-structure on $\Rep(H)_{X^I_{\dR}}$ defined as above for each finite set $I$, satisfies the conditions of Proposition \ref{factcatren}, and \[ \Rep(H)^{\ren} \longrightarrow \Rep(H) \] is an isomorphism.

\end{proposition}

\begin{proof}

The t-structure on $\Rep(H)_{X^I_{\dR}}$ inherits compatibility with filtered colimits and right completeness from $\DD(X^I)$ via the conservative and t-exact functor $\oblv_H$. To see that this t-structure is almost compactly generated, it remains to prove that if an object $M$ is compact in $\Rep(H)_{X^I_{\dR}}^{\geq 0}$, then it is almost compact in $\Rep(H)_{X^I_{\dR}}$. By the previous proposition, we have the $\DD(X^I)$-linear adjunction \[ \oblv_H : \Rep(H)_{X^I_{\dR}} \rightleftarrows \DD(X^I) : \coind_H \] with both functors t-exact. Thus $\oblv_H(M)$ is compact in $\DD(X^I)^{\geq 0}$, hence almost compact because $\DD(X^I)$ is almost compactly generated. Now Lemma \ref{aclem} implies that $M$ is almost compact.

If $\{ V_{\alpha} \}$ is a set of compact generators for $\Rep(H)^{\otimes I}$, then $\{ \Loc_{X^I}(V_{\alpha}) \}$ is a set of ULA generators for $\Rep(H)_{X^I_{\dR}}$. Note that $\Loc_{X^I}$ is t-exact up to shift, and in particular each of these objects is eventually coconnective. It follows that $\Rep(H)_{X^I_{\dR}}$ is almost ULA generated and that \[ (\Rep(H)_{X^I_{\dR}})^{\ren} \tilde{\longrightarrow} \Rep(H)_{X^I_{\dR}}. \]

It is not difficult to see that the other conditions in Propositions \ref{sheafranren} and \ref{factcatren} are are inherited from $\DD(X^I)$ via $\oblv_H$.

\end{proof}

\subsection{The t-structure on factorization modules} Fix a commutative algebra $A$ in $\Rep(H)^{\leq 0}$. Recall that we denote the associated factorization algebra in $\Rep(H)$ by the same symbol, and in particular we can consider $A\modfact(\Rep(H))$, an object in $\FactCat^{\laxfact}$.

\begin{proposition}

\label{tstrcomalg}

For any finite set $I$, there is a unique t-structure on $A\modfact(\Rep(H))_{X^I_{\dR}}$ which makes the forgetful functor \[ \oblv^{\fact} : A\modfact(\Rep(H))_{X^I_{\dR}} \longrightarrow \Rep(H)_{X^I_{\dR}} \] t-exact.

\end{proposition}

We will need the following lemma.

\begin{lemma}

\label{finflattexact}

If $Y$ and $Z$ are classical local complete intersection schemes and $f : Y \to Z$ is finite flat, then $f^! : \DD(Z) \to \DD(Y)$ is t-exact.

\end{lemma}

\begin{proof}

First, we claim that the forgetful functor \[ \oblv_Y : \DD(Y) \longrightarrow \IndCoh(Y) \] is t-exact. Recall that for any scheme $Y$ the functor $\oblv_Y$ is left t-exact, so it suffices to prove right t-exactness. Since the claim is Zariski local on $Y$, we can assume that there exists a regular closed embedding $i : Y \to W$ where $W$ is smooth. The composite functor \[ \DD(Y) \xrightarrow{i_{\dR,*}} \DD(W) \xrightarrow{\oblv_W} \IndCoh(W) \] is t-exact, and admits a nonnegative filtration with associated graded \[ \Sym_{\sO_W}(\NN(Y/W)) \otimes i_*\oblv_Y(-) \] (here $\NN(Y/W)$ denotes the normal bundle, and the tensor product denotes the action of $\QCoh(W)$ on $\IndCoh(W)$). Since this associated graded contains $i_*\oblv_Y(-)$ as a direct summand and $i_*$ is t-exact and conservative, it follows that $\oblv_Y$ is right t-exact.

Using the commutative square \[
\begin{tikzcd}
\DD(Z) \arrow{r}{f^!} \arrow{d}{\oblv_Z} & \DD(Y) \arrow{d}{\oblv_Y} \\
\IndCoh(Z) \arrow{r}{f^!} & \IndCoh(Y),
\end{tikzcd} \]
it is therefore enough to prove that $f^! : \IndCoh(Z) \longrightarrow \IndCoh(Y)$ is t-exact. This functor is left t-exact because $f$ is finite. On the other hand, since $Y$ is Gorenstein, so is the morphism $f$, which implies right t-exactness.

\end{proof}

Now we can prove the proposition.

\begin{proof}[Proof of Proposition \ref{tstrcomalg}]

First, we claim that if such a t-structure exists on \[ A\mod^{\fact,\nonuntl}(\Rep(H))_{X^I_{\dR}}, \] then the proposition will follow. Namely, by Proposition \ref{oblvunitff} we have \[ A\mod^{\fact,\nonuntl}(\Rep(H))_{X^I_{\dR}} \cong A\modfact(\Rep(H))_{X^I_{\dR}} \times \Rep(H)_{X^I_{\dR}}. \] It therefore suffices to observe that the direct factor $0 \times  \Rep(H)_{X^I_{\dR}}$ is stable under truncation functors.

Thus we are reduced to proving the proposition for non-unital factorization modules. In that situation, we have the equivalence \[ A\mod^{\fact,\nonuntl}(\Rep(H))_{X^I_{\dR}} \tilde{\longrightarrow} (A \otimes \omega_X[-1])\mod^{\ch}(\Rep(H))_{X^I_{\dR}}, \] which intertwines the forgetful functors to $\Rep(H)_{X^I_{\dR}}$. If $M$ is an object of \[ (A \otimes \omega_X[-1])\mod^{\ch}(\Rep(H))_{X^I_{\dR}}, \] we claim that the truncation $\tau^{\leq 0} M$ taken in $\Rep(H)_{X^I_{\dR}}$ admits a chiral $(A \otimes \omega_X[-1])$-module structure compatible with the morphism $\tau^{\leq 0}M \to M$. It will then follow that\[ \Hom_{(A \otimes \omega_X[-1])\mod^{\ch}(\Rep(H))_{X^I_{\dR}}}(N,\tau^{\leq 0} M) \tilde{\longrightarrow} \Hom_{(A \otimes \omega_X[-1])\mod^{\ch}(\Rep(H))_{X^I_{\dR}}}(N,M) \] for any $N$ which is connective in $\Rep(H)_{X^I_{\dR}}$, since \[ \Hom_{\Rep(H)_{X^I_{\dR}}}(N,\tau^{>0}M) = 0 \] and $\oblv_{A \otimes \omega_X[-1]}$ is conservative.

Let $i : Z_I \to X \times X^I$ denote the inclusion of the incidence divisor and $j : (X \times X^I)_{\disj} \to X \times X^I$ its complement, and write $p : Z_I \to X^I$ for the projection. Note that by Lemma \ref{finflattexact}, the functor \[ i_{\dR,*}p^! : \Rep(H)_{X^I_{\dR}} \longrightarrow \Rep(H)_{X_{\dR}} \otimes \Rep(H)_{X^I_{\dR}} \] is t-exact. The structure map of $M$ is a morphism \[ j_*j^*((A \otimes \omega_X[-1]) \boxtimes M) \longrightarrow i_{\dR,*}p^! M \] in $\Rep(H)_{X_{\dR}} \otimes \Rep(H)_{X^I_{\dR}}$. The object \[ j_*j^*((A \otimes \omega_X[-1]) \boxtimes \tau^{\leq 0} M) \] is connective, and hence the composition \[ j_*j^*((A \otimes \omega_X[-1]) \boxtimes \tau^{\leq 0} M) \longrightarrow j_*j^*((A \otimes \omega_X[-1]) \boxtimes M) \longrightarrow i_{\dR,*}p^! M \] factors uniquely through \[ \tau^{\leq 0}i_{\dR,*}p^! M = i_{\dR,*}p^!\tau^{\leq 0} M. \] The same reasoning applies to $n$-ary chiral operations for $n > 1$, which yields the desired chiral module structure on $\tau^{\leq 0}M$.

\end{proof}

\subsection{Commutative factorization modules} We continue to assume that $A$ is a commutative algebra in $\Rep(H)^{\leq 0}$.

\begin{lemma}

\label{tstrlem1}

For any finite set $I$, we have $H^i(A_{X^I_{\dR}}) = 0$ for all $i > -\# I$. If $A$ is eventually coconnective, then so is $A_{X^I_{\dR}}$.

\end{lemma}

\begin{proof}

It suffices to prove both claims after applying the t-exact and conservative functor \[ \oblv_H : \Rep(H)_{X^I_{\dR}} \longrightarrow \DD(X^I), \] so we can assume that $H$ is trivial.

The first assertion can be proved by the same argument as Lemma 6.24.3 of \cite{R2}.

The second assertion is clear from the Cousin filtration, whose subquotients have the form $A^{\otimes J} \otimes j_*j^!\omega_{Z(p)}$ where $p : I \to J$ is surjective and $j : Z(p) \to X^I$ is the inclusion of the corresponding stratum.

\end{proof}

The symmetric monoidal functor \[ \ind_A : \Rep(H) \longrightarrow A\mod(\Rep(H)) \] gives rise to a morphism of factorization categories. Since these symmetric monoidal categories are rigid, Lemma \ref{rigidrightadj} says that $\ind_A$ admits a right adjoint \[ \oblv_A : A\mod(\Rep(H)) \longrightarrow \Rep(H) \] in $\FactCat_{\laxuntl}$.

\begin{proposition}

\label{comfactmodtstr}

For any finite set $I$, the category $A\mod(\Rep(H))_{X^I_{\dR}}$ admits a unique t-structure such that \[ \oblv_A : A\mod(\Rep(H))_{X^I_{\dR}} \longrightarrow \Rep(H)_{X^I_{\dR}} \] is t-exact.

\end{proposition}

\begin{proof}

The functor $\oblv_A$ admits the $\DD(X^I)$-linear left adjoint $\ind_A$, and the monad $\oblv_A \circ \ind_A$ on $\Rep(H)_{X^I_{\dR}}$ is given by tensoring with the commutative algebra $A_{X^I}$. Since the tensor product in $\DD(X^I)$ is has cohomological amplitude $\# I$, the same holds for $\Rep(H)_{X^I_{\dR}}$, so Lemma \ref{tstrlem1} implies that this monad is right t-exact. The proposition follows.

\end{proof}

The forgetful functor from commutative to factorization $A$-modules is a morphism \[ \oblv^{\comfact} : A\mod(\Rep(H)) \longrightarrow A\modfact(\Rep(H)) \] in $\FactCat^{\laxfact}$. It is compatible with the forgetful functors to $\Rep(H)$, and in particular the functor \[ \oblv^{\comfact} : A\mod(\Rep(H))_{X^I_{\dR}} \longrightarrow A\modfact(\Rep(H))_{X^I_{\dR}} \] is t-exact for any finite set $I$.

Note that the unitality of $\oblv^{\comfact}$ implies that \[ \oblv^{\comfact}A_{\Ran} = \Vac_A. \]

\subsection{Factorization modules in the heart} Let $\sA$ be a symmetric monoidal DG category, $\sC$ an $\sA$-module category, and $\sC_0 \subset \sC$ a (not necessarily stable) full subcategory. For any associative algebra $A$ in $\sA$, we can consider \[ A\mod(\sC_0) := A\mod(\sC) \underset{\sC}{\times} \sC_0. \]

We will use the following formal lemma.

\begin{lemma}

\label{resequivlem}

Suppose that $A \to B$ is a morphism of associative algebras in $\sA$ with the property that \[ \Hom_{\sC}(B^{\otimes n} \otimes M,N) \longrightarrow \Hom_{\sC}(A^{\otimes n} \otimes M,N) \] is an isomorphism for all $n \geq 1$ and all $M,N$ in $\sC_0$. Then the restriction of scalars functor \[ B\mod(\sC_0) \longrightarrow A\mod(\sC_0) \] is an equivalence.

\end{lemma}

Let us fix a finite set $I$.

\begin{lemma}

\label{tstrlem2}

Fix $m \geq 0$. Let $A \to B$ be a morphism of connective commutative algebras in $\Rep(H)$ such that $H^i(A) \tilde{\to} H^i(B)$ for all $i \geq -m$. We have \[ H^i(A_{X^I_{\dR}}) \tilde{\longrightarrow} H^i(B_{X^I_{\dR}}) \] for all $i \geq -\# I - m$.

\end{lemma}

\begin{proof}

As in the proof of Lemma \ref{tstrlem1}, we can assume that $H$ is trivial.

Let us treat the case $\# I = 2$ for simplicity, the general case being similar. Fix $i \geq -2-m$. Our hypothesis implies that \[ H^{i+1}(A \otimes \Delta_{\dR,*}\omega_X) = H^{i+2}(A) \otimes \Delta_{\dR,*}\omega_X[-1] \longrightarrow H^{i+2}(B) \otimes \Delta_{\dR,*}\omega_X[-1] = H^{i+1}(B \otimes \Delta_{\dR,*}\omega_X)\] is an isomorphism. Since $A$ and $B$ are connective, we have \[ H^{i+2}(A \otimes A) \tilde{\longrightarrow} H^{i+2}(B \otimes B), \] from which we deduce that \[ H^i(A \otimes A \otimes j_*j^*\omega_{X^2}) \tilde{\longrightarrow} H^i(B \otimes B \otimes j_*j^*\omega_{X^2}). \] On the other hand, the surjectivity of $H^{i+2}(A \otimes A) \to H^{i+2}(A)$ implies that the connecting homomorphism \[ H^i(A \otimes A \otimes j_*j^*\omega_{X^2}) \longrightarrow H^{i+1}(A \otimes \Delta_{\dR,*}\omega_X) \] for the Cousin triangle is surjective, and likewise for $B$, which implies the desired isomorphism.

\end{proof}

\begin{lemma}

In the situation of the previous lemma, the restriction of scalars functors \[ B\mod(\Rep(H))_{X^I_{\dR}}^{\leq 0, \geq -m} \longrightarrow A\mod(\Rep(H))_{X^I_{\dR}}^{\leq 0, \geq -m} \] and \[ B\modfact(\Rep(H))_{X^I_{\dR}}^{\leq 0, \geq -m} \longrightarrow A\modfact(\Rep(H))_{X^I_{\dR}}^{\leq 0, \geq -m} \] are equivalences.

\label{tstrlem3}

\end{lemma}

\begin{proof}

For the commutative case, we saw in the proof of Proposition \ref{comfactmodtstr} that \[ A_{X^I_{\dR}}\mod(\Rep(H)_{X^I_{\dR}}) \tilde{\longrightarrow} A\mod(\Rep(H))_{X^I_{\dR}}. \] For any $M$ in $\Rep(H)_{X^I_{\dR}}^{\leq 0,\geq -m}$ and any $n \geq 1$, Lemmas \ref{tstrlem1} and \ref{tstrlem2} imply that \[ H^i(A_{X^I_{\dR}}^{\otimes n} \overset{!}{\otimes} M) \tilde{\longrightarrow} H^i(B_{X^I_{\dR}}^{\otimes n} \overset{!}{\otimes} M) \] for all $i \geq -m$, since the tensor product in $\Rep(H)_{X^I_{\dR}}$ has cohomological amplitude $\#I$. Now apply Lemma \ref{resequivlem}.

For the second equivalence, note that it suffices to prove the corresponding claim for non-unital factorization modules. Namely, by Proposition \ref{oblvunitff} and the definition of the t-structure on unital modules, we have \[ (A\modfact_{X^I_{\dR}})^{\leq 0, \geq -m} = A\modfact_{X^I_{\dR}} \cap (A\mod^{\fact,\nonuntl}_{X^I_{\dR}})^{\leq 0, \geq -m}. \]

As for the non-unital case, we use the equivalence \[ A\mod^{\fact,\nonuntl}(\Rep(H))_{X^I_{\dR}} \tilde{\longrightarrow} (A \otimes \omega_X[-1])\mod^{\ch}(\Rep(H))_{X^I_{\dR}},\] which commutes with the forgetful functors to $\Rep(H)_{X^I_{\dR}}$ and in particular is t-exact. The hypotheses imply that $A \otimes \omega_X[-1]$ and $B \otimes \omega_X[-1]$ are connective in $\DD(X)$, and that \[ H^i(A \otimes \omega_X[-1]) \tilde{\longrightarrow} H^i(B \otimes \omega_X[-1]) \] for $i \geq -m$. As previously remarked, the map \[ (\Ran \times \Ran_{X^I})_{\disj} \longrightarrow \Ran_{X^I} \] given by disjoint union is \'{e}tale. It follows that for any $M$ in $\Rep(H)_{X^I_{\dR}}^{\leq 0}$, any $n \geq 1$, and any nonempty finite set $I'$, we have \[ H^i(((A \otimes \omega_X[-1])^{\otimes n} \overset{\ch}{\otimes} M)|^!_{X^{I' \sqcup I}}) \tilde{\longrightarrow} H^i(((B \otimes \omega_X[-1])^{\otimes n} \overset{\ch}{\otimes} M)|^!_{X^{I' \sqcup I}}) \] for all $i \geq -m$. Here we are restricting along the natural map $X^{I' \sqcup I} \to \Ran_{X^I}$, and $\overset{\ch}{\otimes}$ denotes the chiral action of $\Rep(H)_{\Ran}$ on $\Rep(H)_{\Ran_{X^I_{\dR}}}$. Now the claim follows from Lemma \ref{resequivlem}.

\end{proof}

\begin{proposition}

\label{comfactequivheart}

For any commutative algebra $A$ in $\Rep(H)^{\leq 0}$, the t-exact functor $\oblv^{\comfact}$ restricts to an equivalence \[ (A\mod(\Rep(H))_{X^I_{\dR}})^{\heartsuit} \tilde{\longrightarrow} (A\modfact(\Rep(H))_{X^I_{\dR}})^{\heartsuit}. \]

\end{proposition}

\begin{proof}

Consider the commutative square \[
\begin{tikzcd}[row sep=large, column sep=large]
H^0(A)\mod(\Rep(H))_{X^I_{\dR}}^{\heartsuit} \arrow{r}[yshift=0.5em]{\oblv^{\comfact}} \arrow{d} & H^0(A)\modfact(\Rep(H))_{X^I_{\dR}}^{\heartsuit} \arrow{d} \\
A\mod(\Rep(H))_{X^I_{\dR}}^{\heartsuit} \arrow{r}{\oblv^{\comfact}} & A\modfact(\Rep(H))_{X^I_{\dR}}^{\heartsuit}
\end{tikzcd} \]
where the vertical functors, given by restriction of scalars along $A \to H^0(A)$, are equivalences by Lemma \ref{tstrlem3}. Thus we can assume that $A$ is classical, in which case it is immediate from the definitions that $\oblv^{\comfact}$ is fully faithful at the level of abelian categories. It remains to prove essential surjectivity.

Recall the t-exact equivalence \[ A\mod^{\fact,\nonuntl}(\Rep(H))_{X^I_{\dR}} \tilde{\longrightarrow} (A \otimes \omega_X[-1])\mod^{\ch}(\Rep(H))_{X^I_{\dR}}, \] and that unital factorization modules form a full subcategory of the left side. Writing \[ i : Z_I \to X \times X^I \quad \text{and} \quad p : Z_I \to X^I \] for the inclusion of the incidence divisor and the projection as before, the structure map of a chiral module $M$ has the form \[ j_*j^*((A \otimes \omega_X[-1]) \boxtimes M) \longrightarrow i_{\dR,*}p^! M. \] Objects in the essential image of \[ \oblv^{\com \to \fact} : A\mod(\Rep(H))_{X^I_{\dR}}^{\heartsuit} \longrightarrow A\modfact(\Rep(H))_{X^I_{\dR}}^{\heartsuit} \] correspond under the above equivalence to those $M$ such that the composite map
\begin{equation}
(A \otimes \omega_X[-1]) \boxtimes M \longrightarrow j_*j^*((A \otimes \omega_X[-1]) \boxtimes M) \longrightarrow i_{\dR,*}p^! M
\label{comobstr}
\end{equation}
vanishes. But notice that the functor $i^*_{\dR}$ is well-defined on $(A \otimes \omega_X[-1]) \boxtimes M$, and moreover \[ i^*_{\dR}((A \otimes \omega_X[-1]) \boxtimes M) = A \otimes M[1]. \] Since $A$ is classical and the D-module underlying $M$ belongs to $\DD(X^I)^{\heartsuit}$, the map \[ A \otimes M[1] \longrightarrow M \] corresponding to (\ref{comobstr}) vanishes. The proposition follows immediately.

\end{proof}

\subsection{A local acyclity criterion} As above, we let $A$ denote a commutative algebra in $\Rep(H)^{\leq 0}$. For any finite set $I$, the functor \[ \oblv^{\comfact} : A\mod(\Rep(H))_{X^I_{\dR}} \longrightarrow A\modfact(\Rep(H))_{X^I_{\dR}} \] is t-exact, and hence its (not necessarily continuous) right adjoint $\coind^{\fact \to \com}$ is left t-exact.

\begin{proposition}

\label{vacacchar}

For any finite set $I$, the following conditions are equivalent:
\begin{enumerate}[(i)]
\item the $\DD(X^I)$-linear functor \[ \oblv^{\comfact} : A\mod(\Rep(H))_{X^I_{\dR}} \longrightarrow A\modfact(\Rep(H))_{X^I_{\dR}} \] preserves almost ULA objects;
\item the functor \[ \coind^{\fact \to \com} :  A\modfact(\Rep(H))_{X^I_{\dR}}^+ \longrightarrow A\mod(\Rep(H))_{X^I_{\dR}}^+ \] preserves filtered colimits bounded uniformly from below and is $\DD(X^I)^+$-linear;
\item the vacuum module $\Vac_{A,X^I_{\dR}}$ is almost ULA in \[ A\modfact(\Rep(H))_{X^I_{\dR}}. \]
\end{enumerate}
Moreover, these equivalent conditions hold for the pair $(A,H)$ if and only if they hold for $(A,1)$, i.e., $A$ viewed as a commutative algebra in $\Vect^{\leq 0}$.

\end{proposition}

We will apply the following lemma.

\begin{lemma}

\label{aulapreslem}

Suppose that $\sC$ and $\sD$ are $\DD(S)$-modules with compatible t-structures, and that $F : \sC \to \sD$ is $\DD(S)$-linear and t-exact. In particular, the (not necessarily continuous) right adjoint $F^{\RR} : \sD \longrightarrow \sC$ is left t-exact. If \[ F^{\RR} : \sD^+ \longrightarrow \sC^+ \] preserves filtered colimits bounded uniformly from below and is $\DD(S)^+$-linear, then $F$ preserves almost ULA objects. If we assume in addition that $\sC$ is ULA generated, then the converse holds.

\end{lemma}

\begin{proof}

Cf. Proposition B.7.1 in \cite{R2} for the corresponding statement about ULA objects. Assume that $F^{\RR}$ preserves filtered colimits bounded uniformly from below, which implies that $F$ preserves almost compact objects. We have the commutative square of t-exact functors \[
\begin{tikzcd}
\IndCoh(S) \otimes_{\DD(S)} \sC \arrow{r}{\id \otimes F} \arrow{d}{\ind_{\sC}} & \IndCoh(S) \otimes_{\DD(S)} \sD \arrow{d}{\ind_{\sD}} \\
\sC \arrow{r}{F} & \sD,
\end{tikzcd} \]
where the vertical functors are conservative and have the continuous right adjoints $\oblv_{\sC}$ and $\oblv_{\sD}$. Thus Lemma \ref{aclem} implies that the upper horizontal functor preserves almost compact objects, from which it follows that $F$ preserves almost ULA objects.

Now assume that $\sC$ is ULA generated and that $F$ preserves almost ULA objects. The latter assumption implies that for any ULA object $c$ in $\sC$, the object \[ F(\oblv_{\sC}(c)) = \oblv_{\sD}(F(c)) \] is almost compact. Since $\IndCoh(S) \otimes_{\DD(S)} \sC$ is generated as an $\IndCoh(S)$-module by objects of the form $\oblv_{\sC}(c)$, where $c$ is ULA in $\sC$, it follows that \[ \id \otimes F : \IndCoh(S) \underset{\DD(S)}{\otimes} \sC \longrightarrow \IndCoh(S) \underset{\DD(S)}{\otimes} \sD  \] preserves almost compact objects, and hence $F$ does as well by the above commutative square. It follows immediately that \[ F^{\RR} : \sD^+ \longrightarrow \sC^+ \] preserves filtered colimits bounded uniformly from below. This functor is automatically lax $\DD(S)^+$-linear, so it remains to prove that this structure is strict. Passing to right adjoints in the square above, we obtain another commutative square \[
\begin{tikzcd}
\sD^+ \arrow{r}{F^{\RR}} \arrow{d}{\oblv_{\sC}} & \sC^+ \arrow{d}{\oblv_{\sD}} \\
(\IndCoh(S) \otimes_{\DD(S)} \sD)^+ \arrow{r}{(\id \otimes F)^{\RR}} & (\IndCoh(S) \otimes_{\DD(S)} \sC)^+,
\end{tikzcd} \]
where the vertical functors are conservative and $\DD(S)^+$-linear. Thus it suffices to show that the lower horizontal functor is strictly $\IndCoh(S)^+$-linear. The claim is local on $S$, which we can therefore assume is affine. But then $\IndCoh(S)^+$ is generated by the unit object under finite colimits and filtered colimits bounded uniformly from below.

\end{proof}

\begin{proof}[Proof of Proposition \ref{vacacchar}]

Since $A\mod(\Rep(H))_{X^I_{\dR}}$ is ULA generated, Lemma \ref{aulapreslem} implies that (i) is equivalent to (ii).

Condition (i) implies (iii) because the unit object in $A\mod(\Rep(H))_{X^I_{\dR}}$ is ULA.

Consider the commutative square of $\DD(X^I)$-modules \[
\begin{tikzcd}
A\mod(\Rep(H))_{X^I_{\dR}} \arrow{d}{\oblv^{\comfact}} \arrow{r}{\oblv_H} & A\mod_{X^I_{\dR}} \arrow{d}{\oblv^{\comfact}} \\ 
A\modfact(\Rep(H))_{X^I_{\dR}} \arrow{r}{\oblv_H} & A\modfact_{X^I_{\dR}}.
\end{tikzcd} \]
The horizontal functors are t-exact and conservative, and both admit $\DD(X^I)$-linear and t-exact right adjoints, induced by the right adjoint \[ \coind_H : \DD(X^I) \longrightarrow \Rep(H)_{X^I_{\dR}} \] of $\oblv_H$. In particular, the horizontal functors preserve almost ULA objects. Applying Lemma \ref{aclem}, it follows immediately that (ii) implies (i), and that (iii) is equivalent to (iv).

As shown in the proof of Proposition \ref{comfactmodtstr}, the object $A_{X^I}$ is ULA, hence almost ULA, in $A\modfact(\Rep(H))_{X^I_{\dR}}$. Thus (i) implies (iii), and for the same reason (ii) implies (iv).

To complete the proof, it suffices to show that (iv) implies (ii). We will prove that the $\IndCoh(X^I)$-linear functor \[ \oblv^{\comfact} : A\mod_{X^I} \longrightarrow A\modfact_{X^I} \] preserves almost compact objects, which will follow if we can show that \[ \oblv^{\comfact} : A\mod_{X^I}^{\geq m} \longrightarrow (A\modfact_{X^I})^{\geq m} \] preserves compact objects for any $m \in \bZ$. As the claim is local on $X$, we can assume for simplicity that $X$ is affine. Then the DG category $A\mod_{X^I}$ is compactly generated by the object $A_{X^I}$, and hence the (non-stable) category $A\mod_{X^I}^{\geq m}$ is compactly generated by the objects $\tau^{\geq m}(A_{X^I}[-i])$ for $i \geq 0$. Since $\oblv^{\comfact}$ is t-exact, we have \[ \oblv^{\comfact}\tau^{\geq m}(A_{X^I}[-i]) = \tau^{\geq m}(\Vac_{A,X^I}), \] and the right side is compact in $(A\modfact_{X^I})^{\geq m}$ by hypothesis.

\end{proof}

\subsection{Local acyclicity of the vacuum module} Continue to let $A$ denote a commutative algebra in $\Rep(H)^{\leq 0}$. We will say that $A$ is (almost) of finite type if the underlying commutative algebra in $\Vect^{\leq 0}$ is.

\begin{theorem}

If $A$ is almost of finite type, then the equivalent conditions of Proposition \ref{vacacchar} hold.

\label{vacacthm}

\end{theorem}

Under some additional hypotheses, the theorem will allow us to renormalize the factorization category $A\modfact(\Rep(H))$.

\begin{corollary}

\label{factmodcoh}

If $A$ is of finite type and $A\mod(\Rep(H))_{X^I_{\dR}}$ is coherent for every finite set $I$, then $A\modfact(\Rep(H))$ satisfies the hypotheses of Theorem \ref{factcatren}. In particular, under these hypotheses $A\modfact(\Rep(H))_{X^I_{\dR}}$ is almost ULA generated for every finite set $I$.

\end{corollary}

\begin{proof}

First, observe that the hypothesis that $A\mod(\Rep(H))_{X^I_{\dR}}$ is coherent implies that it is almost ULA generated. This follows from Proposition \ref{repulagen} and the $\DD(X^I)$-linear adjunction \[ \ind_A : \Rep(H)_{X^I_{\dR}} \rightleftarrows A\mod(\Rep(H))_{X^I_{\dR}} : \oblv_A. \]

Next, we claim that $A\modfact(\Rep(H))_{X^I_{\dR}}$ is coherent. By Lemma \ref{cohcrit}, it suffices to show that if $M$ is compact in $A\modfact(\Rep(H))_{X^I_{\dR}}^{\heartsuit}$, then $M$ almost compact. By Proposition \ref{comfactequivheart}, we have $M = \oblv^{\com \to \fact} N$ for some $N$ in $A\mod(\Rep(H))_{X^I_{\dR}}^{\heartsuit}$, which is therefore almost compact since $A\mod(\Rep(H))_{X^I_{\dR}}$ was assumed coherent. Now the theorem implies that $M$ is almost compact.

Proposition \ref{comfactequivheart} implies that the essential image of \[ \oblv^{\com \to \fact} : A\mod(\Rep(H))_{X^I_{\dR}}^{\geq 0} \longrightarrow A\modfact(\Rep(H))_{X^I_{\dR}}^{\geq 0} \] generates the target under colimits, which implies that $A\modfact(\Rep(H))_{X^I_{\dR}}$ is almost ULA generated.

Condition (ii) of Proposition \ref{sheafranren} is satisfied by $A\modfact(\Rep(H))$ because it holds for $\Rep(H)$, and the morphism \[ \oblv_A : A\modfact(\Rep(H)) \longrightarrow \Rep(H) \] in $\ShvCat(\Ran^{\untl})$ is conservative over $X^I_{\dR}$ for every finite set $I$. As for condition (iii), note that Lemma \ref{tstrlem1} implies that it holds for $A\mod(\Rep(H))$ because $A$ was assumed to be of finite type, and in particular eventually coconnective. Since \[ \oblv^{\com \to \fact} : A\mod(\Rep(H)) \longrightarrow A\modfact(\Rep(H)) \] is a morphism in $\ShvCat(\Ran^{\untl})$, the condition for $A\modfact(\Rep(H))$ follows from Proposition \ref{comfactequivheart}. The remaining condition in Proposition \ref{factcatren} follows from the corresponding property of $A\mod(\Rep(H))$ in a similar fashion.

\end{proof}

We remark that the hypothesis in the corollary holds if $\Spec A$ is a homogeneous space for $H$. Indeed, in that case we have \[ A\mod(\Rep(H)) \cong \Rep(H_y) \] for any $y : \Spec k \to \Spec A$ with stabilizer $H_y$. On the other hand, if $H$ is the trivial group, then $A\mod_{X^I_{\dR}}$ is generally not coherent.

\subsection{The case of a symmetric algebra} \label{symalgsec} For the remainder of this section, we assume that $H$ is trivial unless otherwise specified. As shown in Proposition \ref{vacacchar}, there no loss of generality when proving Theorem \ref{vacacthm} under this assumption.

Below we prove the key special case of Theorem \ref{vacacthm} where $A$ is a symmetric algebra with finitely many generators.

\begin{lemma}

\label{symchirenv}

Let $A := \Sym V$ where $V$ is a complex of vector spaces. Write $A^+ := \Sym^{>0} V$ for the augmentation ideal of $A$, viewed as a non-unital commutative algebra. Then we have a canonical identification of chiral algebras \[ \UU^{* \to \ch}(V \otimes \omega_X[-1]) \tilde{\longrightarrow} A^+ \otimes \omega_X[-1], \] where $V \otimes \omega_X[-1]$ has the abelian Lie-$*$ algebra structure and $A^+ \otimes \omega_X[-1]$ has the chiral algebra structure induced by the non-unital commutative algebra structure on $A^+$.

\end{lemma}

\begin{proof}

Since $A^+ \otimes \omega_X[-1]$ is commutative as a chiral algebra, the underlying Lie-$*$ algebra \[ \oblv^{\ch \to *}(A^+ \otimes \omega_X[-1]) \] is abelian, and we have a morphism of Lie-$*$ algebras \[ V \otimes \omega_X[-1] \longrightarrow \oblv^{\ch \to *}(A^+ \otimes \omega_X[-1]). \] By adjunction, we obtain a morphism of chiral algebras \[ \UU^{* \to \ch}(V \otimes \omega_X[-1]) \longrightarrow A^+ \otimes \omega_X[-1], \] which we claim is an isomorphism. This can be seen by passing to the associated graded of the PBW filtration on the chiral enveloping algebra, cf. \cite{FG} Corollary 6.5.2.

\end{proof}

\begin{lemma}

\label{symalglem}

If $A = \Sym V$ where $V$ is a connective complex of vector spaces with finite-dimensional total cohomology, then $\Vac_{A,X^I_{dR}}$ is almost ULA in $A\modfact_{X^I_{\dR}}$.

\end{lemma}

\begin{proof}

Let $A^+ := \Sym^{> 0} V$ be as in Lemma \ref{symchirenv}. By the previous lemma, we have an equivalence \[ A\modfact_{X^I_{\dR}} \tilde{\longrightarrow} (A^+ \otimes \omega_X[-1])\mod^{\ch}_{X^I_{\dR}}, \] which is compatible with the forgetful functors to $\DD(X^I)$ and in particular is t-exact. Now consider the $\DD(X^I)$-linear adjunction \[ \ind^{* \to \ch} : (V \otimes \omega_X[-1])\mod^*_{X^I_{\dR}} \rightleftarrows (A^+ \otimes \omega_X[-1])\mod^{\ch}_{X^I_{\dR}} : \oblv^{\ch \to *}. \] The functor $\oblv^{\ch \to *}$ intertwines the forgetful functors to $\DD(X^I)$ and hence is t-exact. It follows formally that $\ind^{* \to \ch}$ sends objects which are almost ULA and eventually coconnective to almost ULA objects. Since $\ind^{* \to \ch}$ is unital, it therefore suffices to show that $\triv_{V \otimes \omega_X[-1]}(\omega_{X^I})$ is almost ULA in $(V \otimes \omega_X[-1])\mod^*_{X^I_{\dR}}$.

For this, we recall that $\triv_{V \otimes \omega_X[-1]}(\omega_{X^I})$ has the natural nonnegative filtration whose $n^{\text{th}}$ associated graded piece is \[ M_n := \ind_{V \otimes \omega_X[-1]}(\Sym^{*,n}(V \otimes \omega_X) \overset{*}{\otimes} (\Delta_{X^I})_{\dR,*} \omega_{X^I}), \] where $\Delta_{X^I} : X^I \to \Ran_{X^I}$ is the main diagonal and \[ \ind_{V \otimes \omega_X[-1]} : \DD(\Ran_{X^I}) \longrightarrow (V \otimes \omega_X[-1])\mod^{\otimes^*}(\DD(\Ran_{X^I})) \] is the induction functor.

The lemma will follow if we prove these two assertions:
\begin{enumerate}[(i)] 
\item For any $n \geq 0$, the object $M_n$ is ULA over $X^I$.
\item Given $m \in \bZ$, for $n$ sufficiently large we have \[ \Hom(M_n,N) = 0 \] for any $N$ in $((V \otimes \omega_X[-1])\mod^*_{X^I_{\dR}})^{\geq m}$.
\end{enumerate}

For claim (i), note that $\ind_{V \otimes \omega_X[-1]}$ preserves ULA objects, being left adjoint to the $D(X^I)$-linear functor $\oblv_{V \otimes \omega_X[-1]}$. The object \[ (V \otimes \omega_X)^{\overset{*}{\otimes} n} \overset{*}{\otimes} (\Delta_{X^I})_{\dR,*} \omega_{X^I} \] is the de Rham direct image of $V^{\otimes n} \otimes \omega_{X^n \times X^I}$ along the proper map \[ X^n \times X^I \longrightarrow \Ran \times \Ran_{X^I} \longrightarrow \Ran_{X^I}, \] hence is ULA over $X^I$. It follows that the object \[ \Sym^{*,n}(V \otimes \omega_X) \overset{*}{\otimes} (\Delta_{X^I})_{\dR,*} \omega_{X^I} \] obtained by taking coinvariants for the symmetric group $\Sigma_n$ is also ULA over $X^I$.

Finally, for (ii) it is enough to show that given $m \in \bZ$, for $n$ sufficiently large we have \[ \Hom_{\DD(\Ran_{X^I})}((V \otimes \omega_X)^{\overset{*}{\otimes} n} \overset{*}{\otimes} \Delta_{I,*} \omega_{X^I}, (\Delta_{X^I})_{\dR,*}N) = 0 \] for any $N \in \DD(X^I)^{\geq m}$ (since the same will therefore hold for the $\Sigma_n$-coinvariants \[ \Sym^{*,n}(V \otimes \omega_X) \overset{*}{\otimes}  (\Delta_I)_{\dR,*} \omega_{X^I} \] of the first object). Given nonempty finite $I'$, define $Z$ to be the fiber product \[
\begin{tikzcd}
Z \arrow{r} \arrow{d} & X^n \times X^I \arrow{d} \\
X^{I' \sqcup I} \arrow{r} & \Ran_{X^I}.
\end{tikzcd} \]
By base change, the object \[ ((V \otimes \omega_X)^{\overset{*}{\otimes} n} \overset{*}{\otimes}  (\Delta_{X^I})_{\dR,*} \omega_{X^I})|^!_{X^{I' \sqcup I}} \] is the de Rham direct image of $V^{\otimes n} \otimes \omega_{X^n \times X^I}$ along the finite map $Z \to X^{I' \sqcup I}$, hence is concentrated in cohomological degrees $\leq -n-\# I$. Considering the same cartesian square in the case $n = 0$ shows that $((\Delta_{X^I})_{\dR,*}N)|^!_{X^{I' \sqcup I}}$ is concentrated in cohomological degrees $\geq -m$. The claim follows.

\end{proof}

\subsection{Filtrations and the associated graded} Let $A$ be a commutative algebra in $\Vect^{\leq 0}$ equipped with a nonnegative filtration. The functor \[ \gr : A\mod^{\fil} \to \gr(A)\mod^{\gr} \] which takes the associated graded module is symmetric monoidal, as is the forgetful functor \[ \oblv^{\fil} : A\mod^{\fil} \longrightarrow A\mod. \] In particular, both give rise to morphisms in $\FactCat$.

We will write \[ A\mod^{\fil}_{\geq 0} \subset A\mod^{\fil} \] for the full symmetric monoidal subcategory consisting of nonnegatively filtered modules (not to be confused with a constraint on cohomological degree, which we indicate with a superscript), and likewise for graded modules.

Let $\Vect^{\gr} := \Rep(\bG_m)$ denote the category of graded vector spaces. Observe that the category $(\Vect^{\gr})_{X^I_{\dR}}$ has a grading by ``total degree fiberwise over $X^I$," characterized by the condition that the natural conservative functor \[ (\Vect^{\gr})_{X^I_{\dR}} \longrightarrow \Vect^{\gr} \otimes D(X^I) \] respects this grading. We will write $M_{\td = i}$ for the graded component of $M$ of total degree $i$.

\begin{lemma}

\label{limgr}

Let $B$ be a nonnegatively graded connective commutative algebra, and suppose we are given an inverse system \[ \cdots \to M_3 \to M_2 \to M_1 \] in $B\mod^{\fact}(\Vect^{\gr})_{X^I_{\dR}}$ such that for each $i \in \bZ$, the inverse system \[ \cdots \to (\oblv_B M_3)_{\td = i} \to (\oblv_B M_2)_{\td = i} \to (\oblv_B M_1)_{\td = i} \] in $(\Vect^{\gr})_{X^I_{\dR}}$ stabilizes. Then the forgetful functor \[ \oblv_B : B\mod^{\fact}(\Vect^{\gr})_{X^I_{\dR}} \longrightarrow (\Vect^{\gr})_{X^I_{\dR}} \] preserves the limit $\lim_n M_n$.

\end{lemma}

\begin{proof}

By Proposition \ref{oblvunitff}, the inclusion \[ B\mod^{\fact}(\Vect^{\gr})_{X^I_{\dR}} \longrightarrow B\mod^{\fact,\nonuntl}(\Vect^{\gr})_{X^I_{\dR}} \] preserves limits. Applying the equivalence \[ B\mod^{\fact,\nonuntl}(\Vect^{\gr})_{X^I_{\dR}} \tilde{\longrightarrow} (B \otimes \omega_X[-1])\mod^{\ch}(\Vect^{\gr})_{X^I_{\dR}}, \] it suffices to check the corresponding claim in the category of chiral modules. The forgetful functor \[ (B \otimes \omega_X[-1])\mod^{\otimes^{\ch}}((\Vect^{\gr})_{\Ran_{X^I_{\dR}}}) \longrightarrow (\Vect^{\gr})_{\Ran_{X^I_{\dR}}} \] preserves limits, so it is enough to prove that the limit $\lim_n \oblv_B M_n$ taken in $(\Vect^{\gr})_{\Ran_{X^I_{\dR}}}$ belongs to the full subcategory $(\Vect^{\gr})_{X^I_{\dR}}$. But this is clear, since for each graded component of fixed total degree the limit stabilizes and hence belongs to $(\Vect^{\gr})_{X^I_{\dR}}$.

\end{proof}

\begin{corollary}

\label{limgrcor}

With $B$ as in the lemma, suppose we are given a sequence of objects $Q_1,Q_2,\cdots$ in $B\mod^{\fact}(\Vect^{\gr})_{X^I_{\dR}}$ such that for each $i \geq 1$, the object $\oblv_B Q_i$ in $(\Vect^{\gr})_{X^I_{\dR}}$ is concentrated in total degree $\geq i$. Then the map \[ \bigoplus_{i \geq 1} Q_i \longrightarrow \prod_{i \geq 1} Q_i \] in $B\mod^{\fact}(\Vect^{\gr})_{X^I_{\dR}}$ is an isomorphism.

\end{corollary}

\begin{proof}

Applying the lemma to the inverse system with terms \[ \prod_{i=1}^n Q_i, \]  we see that the product \[ \prod_{i \geq 1} Q_i \] is preserved by $\oblv_B$. Now the corollary follows from the observation that the map \[ \bigoplus_{i \geq 1} \oblv_B Q_i \longrightarrow \prod_{i \geq 1} \oblv_B Q_i \] in $\Vect^{\gr}_{X^I_{\dR}}$ is an isomorphism, since in each graded component of fixed total degree, all but finitely many factors in the product vanish.

\end{proof}

\begin{lemma}

\label{assocgrlem}

Fix a nonnegatively filtered connective commutative algebra $A$. If the conditions of Proposition \ref{vacacchar} hold for $\gr(A)$, then they hold for $\oblv^{\fil} A$.

\end{lemma}

\begin{proof}

\emph{Step 1:} With $B$ as in Lemma \ref{limgr}, fix $m \in \bZ$ and suppose that $P$ is a compact object in $B\mod^{\fact}(\Vect^{\gr})_{X^I_{\dR}}^{\geq m}$. We claim that there exists $r \geq 0$ such that for any object $Q$ of $B\mod^{\fact}(\Vect^{\gr})_{X^I_{\dR}}^{\geq m}$ with $\oblv_B Q$ concentrated in total degree $\geq r$, we have \[ \Hom_{B\mod^{\fact}(\Vect^{\gr})_{X^I_{\dR}}}(P,Q) = 0. \] Otherwise we could choose a sequence $Q_1,Q_2,\cdots$ in $B\mod^{\fact}(\Vect^{\gr})_{X^I_{\dR}}^{\geq m}$ such that $\oblv_B Q_i$ is concentrated in total degree $\geq i$ and \[ \Hom_{B\mod^{\fact}(\Vect^{\gr})_{X^I_{\dR}}}(P,Q_i) \neq 0 \] for any $i \geq 1$. By Corollary \ref{limgrcor}, this contradicts the compactness of $P$ in $B\mod^{\fact}(\Vect^{\gr})_{X^I_{\dR}}^{\geq m}$.

\emph{Step 2:} By Lemma \ref{rigidrightadj}, the functor \[ \gr : (\Vect^{\fil})_{X^I_{\dR}} \longrightarrow (\Vect^{\gr})_{X^I_{\dR}} \] admits a right adjoint which is $\DD(X^I)$-linear and compatible with factorization. In particular, this lifts to a $\DD(X^I)$-linear adjunction \[ A\modfact(\Vect^{\fil})_{X^I_{\dR}} \rightleftarrows \gr(A)\modfact(\Vect^{\gr})_{X^I_{\dR}}. \]

We define a functorial \emph{descending} filtration on any $M$ in $A\modfact(\Vect^{\fil})_{X^I_{\dR}}$ as follows. Let $F_0 := M$, and for any $n \geq 0$ define \[ F_{n+1} := \fib(F_n \to \gr F_n) \] (here we are suppressing from the notation the right adjoint of $\gr$, which commutes with the forgetful functors to $\DD(X^I)$). We claim that if $M$ is concentrated in total degree $\geq 0$ then \[ (\oblv_A M)_{\td = 0} \longrightarrow (\oblv_A \gr M)_{\td = 0} \] is an isomorphism. It will then follow inductively that $F_n$ is concentrated in total degree $\geq n$ for any $n \geq 0$, and hence that \[ M \tilde{\longrightarrow} \lim_n M/F_n \] by Lemma \ref{limgr}.

Recall that $X^I$ is stratified with strata $Z(p)$ indexed by surjections $p : I \to J$. Using the Cousin filtration, it is enough to check the claim over each $Z(p)$. There we are working in the multifiltered category $(\Vect^{\fil})^{\otimes J} \otimes \DD(Z(p))$ with its standard grading by total degree, and the claim is evident.

\emph{Step 3:} Let $V := \ind_{X^I}(\Vac_{A,X^I})$, an object of $A\modfact(\Vect^{\fil})_{X^I_{\dR}}$, so the lemma will follow if we prove that $\oblv^{\fil} V$ is almost compact in $A\modfact_{X^I_{\dR}}$. Fixing $m \in \bZ$, we must show that \[ \Hom_{A\modfact_{X^I_{\dR}}}(\oblv^{\fil}V,N) \] preserves filtered colimits when viewed as a functor of $N$ in $(A\modfact_{X^I_{\dR}})^{\geq m}$. Note that the functor \[ \oblv^{\fil} : \Vect^{\fil}_{\geq 0} \longrightarrow \Vect \] admits a symmetric monoidal right adjoint $\triv^{\fil}_{\geq 0}$, which equips a vector space with the trivial nonnegative filtration. In particular, this lifts to a $\DD(X^I)$-linear adjunction \[ \oblv^{\fil} : A\modfact(\Vect^{\fil}_{\geq 0})_{X^I_{\dR}} \rightleftarrows A\modfact_{X^I_{\dR}} : \triv^{\fil}_{\geq 0}. \] Thus we have \[ \Hom_{A\modfact_{X^I_{\dR}}}(\oblv^{\fil}V,N) = \Hom_{A\modfact(\Vect^{\fil})_{X^I_{\dR}}}(V,\triv^{\fil}_{\geq 0}N). \]

Fix $r$ as in Step 1 for $B := \gr(A)$ and $P := \tau^{\geq m} \gr(V)$. Recalling the descending filtration on $M := \triv^{\fil}_{\geq 0} N$ constructed in Step 2, we claim that \[ \Hom_{A\modfact(\Vect^{\fil})_{X^I_{\dR}}}(V,F_r) = 0. \] Since $F_r = \lim_k F_r/F_{r+k}$, it suffices to show that \[ \Hom_{A\modfact(\Vect^{\fil})_{X^I_{\dR}}}(V,F_n/F_{n+1}) = 0 \] for any $n \geq r$. By construction we have 
\begin{align*}
\Hom_{A\modfact(\Vect^{\fil})_{X^I_{\dR}}}(V,F_n/F_{n+1}) &= \Hom_{A\modfact(\Vect^{\fil})_{X^I_{\dR}}}(V,\gr F_n) \\
&= \Hom_{\gr(A)\modfact(\Vect^{\gr})_{X^I_{\dR}}}(\gr(V),\gr F_n).
\end{align*}
Note that $M = \triv^{\fil}_{\geq 0} N$ implies that $\gr F_n$ is concentrated in cohomological degrees $\geq m$ for any $n \geq 0$ (this can be checked over each stratum $Z(p)$, where it becomes an obvious statement about the trivial nonnegative multifiltration). Now the claim follows from Step 1.

\emph{Step 4:} Continuing the notation of Step 3, we have \[ \Hom_{A\modfact_{X^I_{\dR}}}(\oblv^{\fil}V,N) = \Hom_{A\modfact(\Vect^{\fil})_{X^I_{\dR}}}(V,M/F_r). \] Thus it suffices to show that for any $n \geq 0$, \[ \Hom_{A\modfact(\Vect^{\fil})_{X^I_{\dR}}}(V,F_n/F_{n+1}) = \Hom_{\gr(A)\modfact(\Vect^{\gr})_{X^I_{\dR}}}(\gr(V),\gr F_n) \] preserves filtered colimits when viewed as a functor of $N$ in $(A\modfact_{X^I_{\dR}})^{\geq m}$. Since $\gr F_n$ belongs to $\gr(A)\modfact(\Vect^{\gr})_{X^I_{\dR}}^{\geq m}$ by Step 3, we are done by the hypothesis on $\gr(A)$.

\end{proof}

\subsection{Finitely presented algebras} A connective commutative algebra $A$ is called $0$\emph{-finitely presented} if it is isomorphic to a classical finitely generated polynomial algebra. For $n \geq 1$, we say that $A$ is $n$\emph{-finitely presented} if \[ A \cong B \underset{\Sym(V[n-1])}{\otimes} k \] where $B$ is $(n-1)$-finitely presented and $V$ is a finite-dimensional classical vector space equipped with a $k$-linear map $V[n-1] \to B$. If $A$ is $n$-finitely presented for some $n \geq 0$, we say that $A$ is \emph{finitely presented}.

For example, an algebra $A$ is $1$-finitely presented if and only if its spectrum is a derived global complete intersection.

\begin{lemma}

\label{fpthm}

The conditions of Proposition \ref{vacacchar} hold for $A$ finitely presented.

\end{lemma}

\begin{proof}

We prove by induction on $n$ that the conditions hold for any $A$ of the form $B \otimes \Sym V$, where $B$ is $n$-finitely presented and $V$ is a connective complex of vector spaces with finite-dimensional total cohomology (for varying $n$, this statement is equivalent to the assertion of the lemma). The case $n = 0$ is precisely Lemma \ref{symalglem}. In general, we have \[ B \cong C \underset{\Sym(W[n-1])}{\otimes} k \] where $C$ is $(n-1)$-finitely presented and $W$ is finite-dimensional and classical. Thus $B$ admits a nonnegative filtration with associated graded algebra \[ \gr B \cong C \otimes \Sym(W[n]). \] Consequently $A$ admits a nonnegative filtration with associated graded \[ \gr A \cong C \otimes \Sym(W[n]) \otimes \Sym V \cong C \otimes \Sym(W[n] \oplus V). \] Theorem \ref{vacacthm} holds for $\gr A$ by the inductive hypothesis, whence the theorem holds for $A$ by Lemma \ref{assocgrlem}.

\end{proof}

\subsection{Conclusion of the proof} We will use the following (straightforward) approximation lemma to prove Theorem \ref{vacacthm}.

\begin{lemma}

\label{fpres}

If $A$ is connective and almost of finite type, there exists a directed system \[ B_0 \longrightarrow B_1 \longrightarrow B_2 \longrightarrow \cdots \] of connective commutative algebras mapping to $A$ such that
\begin{enumerate}
\item $B_n$ is $n$-finitely presented for every $n \geq 0$,
\item $H^{-i}(B_n) \tilde{\to} H^{-i}(A)$ for $i < n$, 
\item $H^{-n}(B_n) \to H^{-n}(A)$ is surjective.
\end{enumerate}
In particular $\underset{n}{\colim} \ B_n \tilde{\to} A$.

\end{lemma}

Finally, we can prove the theorem.

\begin{proof}[Proof of Theorem \ref{vacacthm}]

Choose a filtration \[ \underset{n}{\colim} \ B_n \tilde{\to} A \] as in Lemma \ref{fpres}. Fix $m \geq 0$. Lemma \ref{tstrlem3} implies that for $n > m$, the restriction of scalars induces an equivalence \[ (A\modfact_{X^I_{\dR}})^{\leq 0, \geq -m} \tilde{\longrightarrow} (B_n\modfact_{X^I_{\dR}})^{\leq 0, \geq -m}. \] Moreover, by Lemmas \ref{tstrlem1} and \ref{tstrlem2}, this equivalence sends \[ \tau^{\geq -m}(\Vac_{A,X^I_{\dR}}[-\# I]) \mapsto \tau^{\geq -m}(\Vac_{B_n,X^I_{\dR}}[-\# I]). \] The inclusion \[ (A\modfact_{X^I})^{\leq 0, \geq -m} \longrightarrow (A\modfact_{X^I})^{\geq -m} \] preserves compact objects because the t-structure on $A\modfact_{X^I}$ is compatible with filtered colimits. Since $\Vac_{B_n,X^I_{\dR}}$ is almost ULA by Lemma \ref{fpthm}, it follows that \[ \tau^{\geq -m}(\Vac_{A,X^I}[-\# I]) \] is compact in $(A\modfact_{X^I})^{\geq -m}$. Thus $\Vac_{A,X^I_{\dR}}$ is almost ULA as desired.

\end{proof}

\section{Spectral Hecke categories}\label{s:spec-hecke}

In this section, we construct the factorization categories which appear on the spectral (a.k.a. Langlands dual) side. These arise as renormalizations of certain categories of factorization modules.

\subsection{The factorization algebra $\Upsilon(\mathfrak{n}_P)$} Recall that the Lie bracket on $\mathfrak{n}_P$ induces the structure of Lie-$*$ algebra on the constant sheaf $\mathfrak{n}_P \otimes k_X$, and hence we have the object $(\mathfrak{n}_P \otimes k_X)\text{-mod}^*(\Rep(M))$ in $\FactCat^{\laxfact}$. On the other hand, we can consider the object of $\FactCat$ attached to the symmetric monoidal category $\mathfrak{n}_P\mod(\Rep(M))$.

\begin{proposition}

There is a canonical isomorphism \[ (\mathfrak{n}_P \otimes k_X)\mod^*(\Rep(M)) \tilde{\longrightarrow} \mathfrak{n}_P\mod(\Rep(M)) \] $\FactCat^{\laxfact}$. In particular, the left side belongs to $\FactCat$.

\end{proposition}

\begin{proof}

This is Lemma 7.7.1 in \cite{R2}.

\end{proof}

In what follows we will pass freely between $(\mathfrak{n}_P \otimes k_X)\text{-mod}^*(\Rep(M))$ and $\mathfrak{n}_P\mod(\Rep(M))$.

Following \cite{R2}, we write \[ \Upsilon(\mathfrak{n}_P,-) : \mathfrak{n}_P\mod(\Rep(M)) \longrightarrow \Rep(M) \] for the morphism in $\FactCat_{\laxuntl}$ given by the composition \[ \mathfrak{n}_P\mod(\Rep(M)) \xrightarrow{\ind^{* \to \ch}} \UU^{* \to \ch}(\mathfrak{n}_P \otimes k_X)\text{-mod}^{\ch}(\Rep(M)) \xrightarrow{\oblv} \Rep(M) \] (this is the chiral analogue of the $\bE_2$-monoidal functor denoted by $\Chev_{\Upsilon}$ in the introduction).

Denote by $\Upsilon(\mathfrak{n}_P)$ the factorization algebra in $\Rep(M)$ obtained by applying $\Upsilon(\mathfrak{n}_P,-)$ to the unit object in $\mathfrak{n}_P\mod(\Rep(M))$. In particular, applying (\ref{factmodfunct}) to $\Upsilon(\mathfrak{n}_P,-)$ yields a morphism
\begin{equation}
\label{indups}
\mathfrak{n}_P\mod(\Rep(M)) \longrightarrow \Upsilon(\mathfrak{n}_P)\modfact(\Rep(M))
\end{equation}
in $\FactCat^{\laxfact}$.

Since $\ind^{* \to \ch}$ is strictly unital, we have an identification \[ \UU^{* \to \fact}(\mathfrak{n}_P \otimes k_X) \tilde{\longrightarrow} \Upsilon(\mathfrak{n}_P) \] in $\FactAlg(\Rep(M))$ and a commutative triangle \[
\begin{tikzcd}
& \mathfrak{n}_P\mod(\Rep(M)) \arrow{dl}[swap]{(\ref{indups})} \arrow{dr}{\ind^{* \to \ch}} & \\
\Upsilon(\mathfrak{n}_P)\modfact(\Rep(M)) \arrow{rr}{\sim} & & \UU^{* \to \ch}(\mathfrak{n}_P \otimes k_X)\mod^{\ch}(\Rep(M))
\end{tikzcd} \]
in $\FactCat^{\laxfact}$.

Making this identification, we have an adjunction
\begin{equation}
\label{chiralind}
\ind^{* \to \ch} : \mathfrak{n}_P\mod(\Rep(M)) \rightleftarrows \Upsilon(\mathfrak{n}_P)\modfact(\Rep(M)) : \oblv^{\ch \to *}
\end{equation}
in $\FactCat_{\laxuntl}^{\laxfact}$. Note that the right adjoint is compatible with the forgetful functors to $\Rep(M)$, and in particular is conservative. This adjunction is therefore monadic, which implies that the category $\Upsilon(\mathfrak{n}_P)\modfact(\Rep(M))$ belongs to $\FactCat$, i.e., factorizes strictly. In particular, the adjunction (\ref{chiralind}) takes place in $\FactCat_{\laxuntl}$, with the left adjoint being strictly unital.

\subsection{An integrability condition on $\Upsilon(\mathfrak{n}_P)$-modules} Note that the symmetric monoidal functor \[ \res^{N_P}_{\mathfrak{n}_P} : \Rep(P) \longrightarrow \mathfrak{n}_P\mod(\Rep(M)) \] is fully faithful because $\mathfrak{n}_P$ is nilpotent.

\begin{lemma}

For any finite set $I$, the functor \[ \res^{N_P}_{\mathfrak{n}_P} : \Rep(P)_{X^I_{\dR}} \longrightarrow \mathfrak{n}_P\mod(\Rep(M))_{X^I_{\dR}} \] is fully faithful.

\end{lemma}

\begin{proof}

Since the symmetric monoidal category $\Rep(P)$ is rigid and the unit object in $\mathfrak{n}_P\mod(\Rep(M))$ is compact, Lemma \ref{rigidrightadj} says that corresponding morphism in $\FactCat$ admits a right adjoint in $\FactCat_{\laxuntl}$. The unit of this adjunction is an isomorphism because it is compatible with factorization and is an isomorphism over $X$.

\end{proof}

In what follows, we will freely identify $\Rep(P)_{X^I_{\dR}}$ with its essential image in $\mathfrak{n}_P\mod(\Rep(M))_{X^I_{\dR}}$.

\begin{proposition}

\label{indchint}

The monad on $\mathfrak{n}_P\mod(\Rep(M))_{X^I_{\dR}}$ arising from the adjunction (\ref{chiralind}) preserves the full subcategory $\Rep(P)_{X^I_{\dR}}$ for any finite set $I$.

\end{proposition}

We will apply a general lemma, which we now prepare to state. Let $\sC$ be a an object of $\ComAlg(\FactCat^{\laxfact})$ and $L$ a Lie-$*$ algebra in $\sC$. Then we have adjoint functors \[ \ind^{* \to \ch} : L\mod^*(\sC) \rightleftarrows \UU^{* \to \ch}(L)\mod^{\ch}(\sC) : \oblv^{\ch \to *} \] in $\FactCat_{\laxuntl}^{\laxfact}$, with the left adjoint being strictly unital.

The factorization structure on $L\mod^*(\sC)$ is commutative, i.e., it lifts to $\ComAlg(\FactCat^{\laxfact})$, compatibly with the forgetful functor \[ \oblv_L : L\mod^*(\sC) \longrightarrow \sC. \] In particular, for any $Z \to \Ran$ the category \[ L\mod^*(\sC)_Z \] is equipped with a symmetric monoidal structure lifting the one on $\sC_Z$, which will also be denoted by $\otimes^!$.

Viewing $L$ with the adjoint action as an object of the symmetric monoidal category $L\mod^*(\sC)_{X_{\dR}}$, the symmetric algebra $\Sym^!(L[1])$ is a commutative algebra in this category and hence gives rise to a commutative factorization algebra in $L\mod^*(\sC)$.

\begin{lemma}

\label{indchfiltr}

The monad $\oblv^{\ch \to *}\ind^{* \to \ch}$ on $L\mod^*(\sC)_{X^I_{\dR}}$ admits a canonical filtration with associated graded isomorphic to \[ \Sym^!(L[1])_{X^I_{\dR}} \overset{!}{\otimes} (-). \]

\end{lemma}

\begin{proof}

First, we record the apparently weaker assertion that the functor \[ \oblv_{\UU^{* \to \ch}(L)}\ind^{* \to \ch} : L\mod^*(\sC)_{X^I_{\dR}} \longrightarrow \sC_{X^I_{\dR}} \] is filtered with associated graded \[ \oblv_L(\Sym^!(L[1])_{X^I_{\dR}} \overset{!}{\otimes} (-)). \] This is a variant of Corollary 6.5.2 in \cite{FG} and is proved in the same way.

Now we will deduce the lemma. Observe that the adjoint action canonically upgrades $L$ to a Lie-$*$ algebra in $L\mod^*(\sC)$. It follows that $\UU^{* \to \ch}(L)$ upgrades to a chiral algebra in $L\mod^*(\sC)$. Moreover, the morphism \[ \oblv_L : L\mod^*(\sC) \longrightarrow \sC \] in $\ComAlg(\FactCat^{\laxfact})$ induces isomorphisms \[ L\mod^*(L\mod^*(\sC)) \tilde{\longrightarrow} L\mod^*(\sC) \] and \[ \UU^{* \to \ch}(L)\mod^{\ch}(L\mod^*(\sC)) \tilde{\longrightarrow} \UU^{* \to \ch}(L)\mod^{\ch}(\sC). \] These fit into a commutative square \[
\begin{tikzcd}
L\mod^*(L\mod^*(\sC)) \arrow{r}{\ind^{* \to \ch}} \arrow{d}[anchor=south, rotate=90]{\sim} & \UU^{* \to \ch}(L)\mod^{\ch}(L\mod^*(\sC)) \arrow{d}[anchor=south, rotate=90]{\sim} \\
L\mod^*(\sC) \arrow{r}{\ind^{* \to \ch}} & \UU^{* \to \ch}(L)\mod^{\ch}(\sC)
\end{tikzcd} \]
in $\FactCat^{\laxfact}$, so the lemma follows from the assertion above applied to $L$ as a Lie-$*$ algebra in $L\mod^*(\sC)$.

\end{proof}

\begin{proof}[Proof of Proposition \ref{indchint}]

Since $\mathfrak{n} \otimes k_X$ belongs to $\Rep(P)_{X_{\dR}}$, it follows that \[ \Sym^!(\mathfrak{n} \otimes k_X[1])_{X^I_{\dR}} \] belongs to $\Rep(P)_{X^I_{\dR}}$. Now apply Lemma \ref{indchfiltr} to deduce the proposition.

\end{proof}

For any finite set $I$, define \[ \Upsilon(\mathfrak{n}_P)\mod_0^{\fact}(\Rep(M))_{X^I_{\dR}} \subset \Upsilon(\mathfrak{n}_P)\modfact(\Rep(M))_{X^I_{\dR}} \] to be the full subcategory generated by the image of $\Rep(P)_{X^I_{\dR}}$ under $\ind^{* \to \ch}$. In view of Proposition \ref{indchint}, it follows that these assemble into an object of $\FactCat$, and we have a monadic adjunction
\begin{equation}
\label{chiralind2}
\ind^{* \to \ch} : \Rep(P) \rightleftarrows \Upsilon(\mathfrak{n}_P)\mod_0^{\fact}(\Rep(M)) : \oblv^{\ch \to *}
\end{equation}
in $\FactCat_{\laxuntl}$ with strictly unital left adjoint.

\subsection{The t-structure on integrable $\Upsilon(\mathfrak{n}_P)$-modules} We equip \[ \Upsilon(\mathfrak{n}_P)\mod_0^{\fact}(\Rep(M))_{X^I_{\dR}} \] with the unique t-structure such that $\oblv^{\ch \to *}$ is left t-exact. Equivalently, the connective objects are generated by the image of $\Rep(P)^{\leq 0}$ under $\ind^{* \to \ch}$. Proposition 7.11.1(2) of \cite{R2} says that $\ind^{* \to \ch}$ is actually t-exact.

\begin{proposition}

\label{upscoh}

For any finite set $I$, the t-structure on \[ \Upsilon(\mathfrak{n}_P)\mod_0^{\fact}(\Rep(M))_{X^I_{\dR}} \] is coherent.

\end{proposition}

We will apply the following general lemma.

\begin{lemma}

\label{upscohlem}

Let $F : \sC \to \sD$ be a t-exact functor whose restriction \[ F : \sC^{\heartsuit} \longrightarrow \sD^{\heartsuit} \] is fully faithful. Assume also that the right adjoint $F^{\RR}$ is continuous and conservative. If the t-structure on $\sC$ is coherent, then so is the t-structure on $\sD$.

\end{lemma}

\begin{proof}

A standard argument shows that any compact object in $\sD^{\heartsuit}$ admits a finite filtration with subquotients of the form $F(c)$, where $c$ is compact in $\sC^{\heartsuit}$. If the t-structure on $\sC$ is coherent, then any compact object in $\sC^{\heartsuit}$ is almost compact in $\sC$. Since $F$ preserves almost compact objects and any extension of almost compact objects is almost compact, the lemma follows.

\end{proof}

\begin{proof}[Proof of Proposition \ref{upscoh}]

Lemma 7.16.1 of \cite{R2} says that \[ \ind^{* \to \ch} : \Rep(P)_{X^I_{\dR}} \longrightarrow \Upsilon(\mathfrak{n}_P)\mod_0^{\fact}(\Rep(M))_{X^I_{\dR}} \] is fully faithful on the heart, and the t-structure on $\Rep(P)_{X^I_{\dR}}$ is coherent by Proposition \ref{repulagen}. Thus the hypotheses of Lemma \ref{upscohlem} apply to the adjunction (\ref{chiralind2}).

\end{proof}

\subsection{The factorization algebra $\Upsilon(\mathfrak{n}_P,\sO_G)$} We will abuse notation slightly by also writing $\Upsilon(\mathfrak{n}_P,-)$ for the composition \[ \Rep(G) \xrightarrow{\res^G_P} \Rep(P) \xrightarrow{\res^{N_P}_{\mathfrak{n}_P}} \mathfrak{n}_P\mod(\Rep(M)) \xrightarrow{\Upsilon(\mathfrak{n}_P,-)} \Rep(M). \] This is a morphism in $\FactCat_{\laxuntl}$, and hence can be viewed as a factorization algebra in \[ \uHom_{\FactCat_{\laxuntl}}(\Rep(G),\Rep(M)). \] But $\Rep(G)$ is canonically self-dual in $\FactCat_{\laxuntl}$ by Proposition \ref{ulagenprop}, so we can identify \[ \uHom_{\FactCat_{\laxuntl}}(\Rep(G),\Rep(M)) \cong \Rep(G) \otimes \Rep(M) \cong \Rep(G \times M). \]

We will write $\Upsilon(\mathfrak{n}_P,\sO_G)$ for the factorization algebra in $\Rep(G \times M)$ corresponding to $\Upsilon(\mathfrak{n}_P,-)$. As the notation suggests, this factorization algebra can also be obtained by applying the morphism \[ \id_{\Rep(G)} \otimes \Upsilon(\mathfrak{n}_P,-) : \Rep(G) \otimes \Rep(G) \longrightarrow \Rep(G) \otimes \Rep(M) \cong \Rep(G \times M) \] in $\FactCat_{\laxuntl}$ to $\sO_G$, viewed as a commutative factorization algebra in $\Rep(G) \otimes \Rep(G)$.

By construction, there is a canonical morphism of factorization algebras \[ \unit_{\Rep(G)} \boxtimes \Upsilon(\mathfrak{n}_P) \longrightarrow \Upsilon(\mathfrak{n}_P,\sO_G) \] in $\Rep(G \times M)$. By Lemma \ref{factmodtrans}, this induces an equivalence
\begin{equation}
\label{upsilonfactmodtrans}
\Upsilon(\mathfrak{n}_P,\sO_G)\modfact(\Rep(G \times M)) \tilde{\longrightarrow} \Upsilon(\mathfrak{n}_P,\sO_G)\modfact(\Rep(G) \otimes \Upsilon(\mathfrak{n}_P)\modfact(\Rep(M))).
\end{equation}

\subsection{An integrability condition on $\Upsilon(\mathfrak{n}_P,\sO_G)$-modules} By definition, the factorization algebra $\Upsilon(\mathfrak{n}_P,\sO_G)$ is the image of $\res^{G \times G}_{G \times P}$ under the functor \[ \id_{\Rep(G)} \otimes (\Upsilon(\mathfrak{n}_P,-) \circ \res^{N_P}_{\mathfrak{n}_P}) : \Rep(G \times P) \longrightarrow \Rep(G) \otimes \Upsilon(\mathfrak{n}_P)\modfact(\Rep(M))), \] and in particular belongs to the full subcategory \[ \Upsilon(\mathfrak{n}_P,\sO_G)\modfact(\Rep(G) \otimes \Upsilon(\mathfrak{n}_P)\mod_0^{\fact}(\Rep(M))). \]

We define \[ \Upsilon(\mathfrak{n}_P,\sO_G)\mod_0^{\fact}(\Rep(G \times M)) \subset \Upsilon(\mathfrak{n}_P,\sO_G)\modfact(\Rep(G \times M)) \] to be the full subcategory corresponding to \[ \Upsilon(\mathfrak{n}_P,\sO_G)\modfact(\Rep(G) \otimes \Upsilon(\mathfrak{n}_P)\mod_0^{\fact}(\Rep(M))) \] under the equivalence (\ref{upsilonfactmodtrans}).

The adjunction (\ref{chiralind2}) induces an adjunction \[ \Rep(G \times P) \rightleftarrows \Rep(G) \otimes \Upsilon(\mathfrak{n}_P)\mod_0^{\fact}(\Rep(M)) \] in $\FactCat_{\laxuntl}$ after tensoring with $\id_{\Rep(G)}$. Applying the construction (\ref{factmodfunct}) yields an adjunction
\begin{equation}
\label{psconngen}
\sO_{G}\modfact(\Rep(G \times P)) \rightleftarrows \Upsilon(\mathfrak{n}_P,\sO_{G})\mod_0^{\fact}(\Rep(G \times M))
\end{equation}
in $\FactCat^{\laxfact}_{\laxuntl}$, with the left adjoint belonging to $\FactCat^{\laxfact}$. Note that by construction, the right adjoint makes the square \[
\begin{tikzcd}
\Upsilon(\mathfrak{n}_P,\sO_{G})\mod_0^{\fact}(\Rep(G \times M)) \arrow{r} \arrow{d}{\oblv_{\Upsilon(\mathfrak{n}_P,\sO_G)}} & \sO_{G}\modfact(\Rep(G \times P)) \arrow{d}{\oblv_{\sO_G}} \\
\Rep(G \times M) & \Rep(G \times P) \arrow{l}{\id \otimes \res^P_M}
\end{tikzcd} \]
commute, and in particular is conservative over $X^I_{\dR}$ for any finite set $I$.

\subsection{The t-structure on integrable $\Upsilon(\mathfrak{n}_P,\sO_G)$-modules} Recall that by Proposition \ref{tstrcomalg}, the category \[ \sO_{G}\modfact(\Rep(G \times P))_{X^I_{\dR}} \] admits a unique t-structure which makes the forgetful functor to $\Rep(G \times P)_{X^I_{\dR}}$ t-exact.

\begin{lemma}

\label{psequivheart}

For any finite set $I$, there is a unique t-structure on the category \[ \Upsilon(\mathfrak{n}_P,\sO_G)\mod_0^{\fact}(\Rep(G \times M))_{X^I_{\dR}} \] which makes the left adjoint in (\ref{psconngen}) t-exact over $X^I_{\dR}$, and this functor moreover restricts to an equivalence on the hearts of the t-structures.

\end{lemma}

\begin{proof}

Since the right adjoint in (\ref{psconngen}) is conservative, there is a unique t-structure which makes the left adjoint right t-exact. To see that it is left t-exact, it suffices to check that the corresponding monad on \[ \sO_G\modfact(\Rep(G \times P))_{X^I_{\dR}} \] is left t-exact. The forgetful functor \[ \sO_G\modfact(\Rep(G \times P))_{X^I_{\dR}} \longrightarrow \Rep(G \times P)_{X^I_{\dR}} \] is t-exact and intertwines the monad in question with the monad \[ \id_{\Rep(G)} \otimes (\oblv^{\ch \to *} \circ \ind^{* \to \ch} \circ \res^{N_P}_{\mathfrak{n}_P}) \] on the target. Thus it suffices to show that the latter monad is left t-exact, which follows from the PBW theorem for factorization modules as in the proof of Proposition 7.11.1 in \cite{R3}.

The left adjoint in (\ref{psconngen}) is an equivalence on hearts by a similar calculation using the PBW theorem (see the proof of Lemma 7.16.1 in \emph{loc. cit.}; cf. also Proposition \ref{comfactequivheart} above).

\end{proof}

By Proposition \ref{comfactequivheart}, the functor \[ \Rep(P)_{X^I_{\dR}} \cong \sO_G\mod(\Rep(G \times P))_{X^I_{\dR}} \xrightarrow{\oblv^{\com \to \fact}} \sO_G\modfact(\Rep(G \times P))_{X^I_{\dR}} \] restricts to an equivalence on the hearts of the t-structures. Combining this with Lemma \ref{psequivheart}, we obtain an equivalence \[ (\Rep(P)_{X^I_{\dR}})^{\heartsuit} \tilde{\longrightarrow} (\Upsilon(\mathfrak{n}_P,\sO_G)\mod_0^{\fact}(\Rep(G \times M))_{X^I_{\dR}})^{\heartsuit}. \]

\begin{proposition}

\label{specpscoh}

For any finite set $I$, the t-structure on \[ \Upsilon(\mathfrak{n}_P,\sO_G)\mod_0^{\fact}(\Rep(G \times M))_{X^I_{\dR}} \] satisfies the hypotheses of Proposition \ref{factcatren}. In particular, this category is almost ULA generated over $X^I$.

\end{proposition}

\begin{proof}

First, we show that the t-structure is almost ULA generated. Coherence follows from Lemma \ref{psequivheart} using the criterion of Lemma \ref{cohcrit}, since the left adjoint in (\ref{psconngen}) preserves almost compact objects, the category \[ \Rep(P)_{X^I_{\dR}}^{\heartsuit} \] is compactly generated, and \[ \sO_G\modfact(\Rep(G \times P))_{X^I_{\dR}} \] is almost compactly generated by Corollary \ref{factmodcoh}. Recall that \[ \sO_G\modfact(\Rep(G \times P))_{X^I_{\dR}} \] admits a set of almost ULA generators by Corollary \ref{factmodcoh}, and since the right adjoint in (\ref{psconngen}) is left t-exact, conservative, and $\DD(X^I)$-linear, the claim follows.

Hypothesis (ii) of Proposition \ref{sheafranren} follows from the corresponding property of \[ \sO_G\modfact(\Rep(G \times P)) \] via the right adjoint in (\ref{psconngen}), which is left t-exact and conservative over $X^I_{\dR}$ for all $I$. As for hypothesis (iii), by right completeness it suffices to check that the structure functors in question are left t-exact on the heart of the t-structure. Since \[ \sO_G\modfact(\Rep(G \times P)) \] has this property by Corollary \ref{factmodcoh}, the claim follows from Lemma \ref{psequivheart}.

Similarly, the hypothesis in Proposition \ref{factcatren} is satisfied because of Lemma \ref{psequivheart} and the corresponding property of \[ \sO_G\modfact(\Rep(G \times P)). \]

\end{proof}

\subsection{Definition of the spectral Hecke categories} Applying Propositions \ref{factcatren} and \ref{specpscoh}, we obtain an object \[ \Sph^{\spec}_{G,P} := \Upsilon(\mathfrak{n}_P,\sO_G)\mod_0^{\fact}(\Rep(G \times M))^{\ren} \] in $\FactCat^{\laxfact}$. Composing the morphism \[ \Rep(P) \cong \sO_G\mod(\Rep(G \times P)) \xrightarrow{\oblv^{\comfact}} \sO_G\modfact(\Rep(G \times P)) \] in $\FactCat^{\laxfact}$ with the left adjoint in (\ref{psconngen}), we obtain a morphism
\begin{equation}
\label{psgen}
\Rep(P) \longrightarrow \Upsilon(\mathfrak{n}_P,\sO_G)\mod_0^{\fact}(\Rep(G \times M)).
\end{equation}
This functor is t-exact over $X^I_{\dR}$ for every finite set $I$, and hence renormalizes to a morphism
\begin{equation}
\label{psgenren}
\Rep(P) \longrightarrow \Sph^{\spec}_{G,P}
\end{equation}
in $\FactCat^{\laxfact}$, where we used Proposition \ref{repulagen} to identify $\Rep(P)^{\ren} \cong \Rep(P)$. By Lemma \ref{psequivheart}, this induces an equivalence \[ \Rep(P)_{X^I_{\dR}}^{\heartsuit} \tilde{\longrightarrow} (\Sph^{\spec}_{G,P})_{X^I_{\dR}}^{\heartsuit} \] for any finite set $I$.

\begin{proposition}

The morphism (\ref{psgenren}) admits a right adjoint in $\FactCat^{\laxfact}_{\laxuntl}$, which is conservative over $X^I_{\dR}$ for any finite set $I$. In particular, the adjunction is monadic, and hence $\Sph^{\spec}_{G,P}$ belongs to $\FactCat$.

\end{proposition}

\begin{proof}

Theorem \ref{vacacthm}, together with the existence of the adjunction (\ref{psconngen}), implies that \[ \Rep(P)_{X^I_{\dR}} \longrightarrow (\Sph^{\spec}_{G,P})_{X^I_{\dR}} \] preserves ULA objects for any finite set $I$. Moreover, the essential image of this functor generates its target, since it induces an equivalence on the hearts of the t-structures. Since $\Rep(P)_{X^I_{\dR}}$ is ULA generated by Proposition \ref{repulagen}, we obtain the desired right adjoint. It follows that $\Sph^{\spec}_{G,P}$ factorizes strictly, since $\Rep(P)$ does and we have just shown that (\ref{psgenren}) has a monadic right adjoint.

\end{proof}

\subsection{Monoidal and bimodule structures} In the case $G = P$, we will write \[ \Sph^{\spec}_G := \Sph^{\spec}_{G,G} \] to simplify the notation. Here the factorization algebra $\Upsilon(\mathfrak{n}_P,\sO_G) = \sO_G$ is the monoidal unit in \[ \Rep(G \times G) \cong \uEnd_{\FactCat_{\laxuntl}}(\Rep(G)). \] The object \[ \sO_G\modfact(\Rep(G \times G)) \] is obtained by applying the construction (\ref{factmodunit}) to $\Rep(G \times G)$, hence inherits an associative algebra structure in $\FactCat^{\laxfact}$. Similarly, the object \[ \Upsilon(\mathfrak{n}_P,\sO_G)\modfact(\Rep(G \times M)) \] is naturally a bimodule for the pair \[ (\sO_G\modfact(\Rep(G \times G)),\sO_M\modfact(\Rep(M \times M))). \]

\begin{proposition}

\label{specheckemult}

The object $\Sph^{\spec}_G$ admits a unique structure of associative algebra in $\FactCat$ compatible with the morphism \[ \Sph^{\spec}_G \longrightarrow \sO_G\modfact(\Rep(G \times G)). \] Similarly, the object $\Sph^{\spec}_{G,P}$ admits a unique structure of $(\Sph^{\spec}_G,\Sph^{\spec}_M)$-bimodule in $\FactCat$ compatible with the morphism \[ \Sph^{\spec}_{G,P} \longrightarrow \Upsilon(\mathfrak{n}_P,\sO_G)\modfact(\Rep(G \times M)) \] in $\FactCat^{\laxfact}$.

\end{proposition}

\begin{proof}

For the first claim, by Proposition \ref{monfactcatren} it suffices to show that the unit and multiplication morphisms for the associative algebra \[ \sO_G\modfact(\Rep(G \times G)) \] are unary and binary morphisms, respectively, in the pseudo-tensor category $(\FactCat^{\laxfact})^{\otimes}_{\aULA}$. This means that for any finite set $I$, the monoidal unit object in \[ \sO_G\modfact(\Rep(G \times G))_{X^I_{\dR}} \] is eventually coconnective, and that eventually coconnective objects are stable under the monoidal operation. By Lemmas \ref{tstrlem1} and \ref{tstrlem2}, the unit object $\Vac_{\sO_G,X^I_{\dR}}$ is concentrated in cohomological degree $-\# I$, and in particular is eventually coconnective. As for the multiplication, the t-exact monoidal functor \[ \oblv^{\com \to \fact} : \Rep(G)_{X^I_{\dR}} \longrightarrow \sO_G\modfact(\Rep(G \times G))_{X^I_{\dR}} \] induces an equivalence on the hearts by Proposition \ref{comfactequivheart}, and monoidal structure on $\Rep(G)_{X^I_{\dR}}$ is left t-exact in each variable separately. The claim now follows by right completeness of the t-structure on \[ \sO_G\modfact(\Rep(G \times G))_{X^I_{\dR}}. \]

Recall that the morphism (\ref{psgen}) in $\FactCat^{\laxfact}$ is t-exact and induces an equivalence on the hearts over each $X^I_{\dR}$. By right completeness of the t-structures, it follows that the action of the pair \[ (\sO_G\modfact(\Rep(G \times G)),\sO_M\modfact(\Rep(M \times M))) \] preserves the full subcategory \[ \Upsilon(\mathfrak{n}_P,\sO_G)\mod_0^{\fact}(\Rep(G \times M)) \subset \Upsilon(\mathfrak{n}_P,\sO_G)\modfact(\Rep(G \times M)), \] and that the action of the pair on this subcategory takes place in $(\FactCat^{\laxfact})^{\otimes}_{\aULA}$. Then we are done by Proposition \ref{monfactcatren}.

\end{proof}

\section{Construction of the derived Satake transform}\label{s:functor}

In this section, we construct the factorization categories that appear on the geometric (a.k.a. automorphic) side. We then construct functors to the spectral side, and formulate our main results as Theorems \ref{mainthm} and \ref{mainthmps}.

\subsection{Canonical twists} Let $\rho$ denote the half-sum of the simple coroots of $G$. Choose a square root $(\Omega^1_X)^{\otimes \frac12}$ of the canonical line bundle $\Omega^1_X$ on $X$, and put \[ \rho(\Omega^1_X) := 2\rho((\Omega^1_X)^{\otimes \frac12}), \] which is a well-defined $T$-bundle on $X$ because $2\rho \in \Lambda$ is an integral coweight.

In what follows, we replace $G$ by its \emph{canonical twist} by $\rho(\Omega^1_X)$. By definition, this is the group scheme of automorphisms of the $G$-bundle on $X$ defined by \[ G \overset{T}{\times} \rho(\Omega^1_X). \] Similarly, we replace $P$ and $P^-$ by their canonical twists, defined in the same way (the corresponding twist of $T$ itself is trivial because $T$ is commutative). These are subgroup schemes of the canonical twist of $G$. We remark that by construction, all of these twists are pure inner forms (cf. \S 2.15 of \cite{R3} of the corresponding constant group schemes $G$, $P$, and $P^-$. In particular, replacing each of these groups with its canonical twist does not affect the corresponding space of principal bundles.

The canonical twist of $N$ is defined as \[ N \overset{T}{\times} \rho(\Omega^1_X) \] using the adjoint action of $T$ on $N$, and similarly for $N^-$. These are normal subgroup schemes of the canonical twists of $B$ and $B^-$ respectively, with the quotient in either case being the constant group scheme $T$. The canonical twists of $N$ and $N^-$ are \emph{not} pure inner forms of the constant group schemes, and their spaces of principal bundles generally differ. One can similarly define the canonical twist of any subgroup of $N$ stable under the adjoint action of $T$, such as $N_P$ or $N_M := N \cap M$.

These twists will be suppressed below in order to simplify the notation. Their presence can be safely ignored in most situations, with the notable exception of defining the Whittaker condition.

\subsection{The spherical Hecke stack} Write $\sH_G := \mathfrak{L}^+G \backslash \mathfrak{L}G/\mathfrak{L}^+G$ for the spherical Hecke stack, a groupoid in corr-unital factorization spaces (i.e., factorization spaces with a unit correspondence rather than a single morphism, cf. \cite{charles-lin} \S 10). The groupoid structure on $\sH_G$ induces a monoidal structure on \[ \DD(\sH_{G,X^I}) := \DD((\mathfrak{L}G)_{X^I})^{\mathfrak{L}^+G \times \mathfrak{L}^+G} \] for each finite set $I$, and these assemble into an object $\DD(\sH_G)$ of $\AssocAlg(\FactCat)$ satisfying \[ \DD(\sH_G)_{X^I_{\dR}} = \DD(\sH_{G,X^I}). \]

Inversion on the group $\fL G$ induces an anti-automorphism of $\DD(\sH)$, which allows us to exchange left and right $\DD(\sH)$-module structures. In what follows, we will make use of this without comment.

For any finite set $I$, the Beilinson-Drinfeld affine Grassmannian \[ \Gr_{G,X^I} = (\mathfrak{L}G/\mathfrak{L}^+G)_{X^I} \] is an ind-scheme locally of finite type, and in particular $\DD(\Gr_{G,X^I})$ is equipped with a natural t-structure. We give $\DD(\sH_{G,X^I})$ the unique t-structure such that the forgetful functor \[ \DD(\sH_{G,X^I}) = \DD(\Gr_{G,X^I})^{\mathfrak{L}^+G} \longrightarrow \DD(\Gr_{G,X^I}) \] is t-exact.

\subsection{Definition of the spherical Hecke category} The na\"{i}ve Satake transform, as constructed in \S 6 of \cite{R2}, is a morphism \[ \Sat^{\nv}_G : \Rep(\check{G}) \longrightarrow \DD(\sH_G) \] in $\AssocAlg(\FactCat)$. We recall from \emph{loc. cit.} that for any finite set $I$, the corresponding functor \[ \Rep(\check{G})_{X^I_{\dR}} \longrightarrow \DD(\sH_G)_{X^I_{\dR}} \] is t-exact and induces an equivalence \[ \Rep(\check{G})_{X^I_{\dR}}^{\heartsuit} \tilde{\longrightarrow} \DD(\sH_G)_{X^I_{\dR}}^{\heartsuit}. \]

\begin{lemma}

\label{naivesatula}

For any finite set $I$, the functor \[ \Sat^{\nv}_G : \Rep(\check{G})_{X^I_{\dR}} \longrightarrow \DD(\sH_G)_{X^I_{\dR}} \] preserves almost ULA objects.

\end{lemma}

\begin{proof}

It suffices to show that the $\IndCoh(X^I)$-linear functor \[ \Rep(\check{G})_{X^I} \longrightarrow \DD(\sH_G)_{X^I} \] preserves almost compact objects. By Proposition \ref{repulagen}, almost compact objects in $\Rep(\check{G})_{X^I}$ are compact. Moreover, the symmetric monoidal structure is rigid, which means that compactness in $\Rep(\check{G})_{X^I}$ is equivalent to dualizability. Since $\Sat^{\nv}_G$ is monoidal, it preserves dualizable objects, so it remains only to show that the unit object $\delta_{\mathfrak{L}^+ G,X^I}$ is almost compact in $\DD(\sH_G)_{X^I}$. For this, recall that the forgetful functor \[ \oblv_{\mathfrak{L}^+G} : \DD(\sH_G)_{X^I} \longrightarrow \IndCoh(X^I) \otimes_{\DD(X^I)} \DD(\Gr_{G,X^I}) =: \DD(\Gr_G)_{X^I} \] is t-exact, conservative, and admits an $\IndCoh(X^I)$-linear right adjoint, namely $*$-averaging with respect to $\mathfrak{L}^+G$. The image of $\delta_{\mathfrak{L}^+ G,X^I}$ under this functor is the direct image of $\omega_{X^I}$ along the unit section \[ X^I \longrightarrow X^I \times_{X^I_{\dR}} (\Gr_{G,X^I})_{\dR}, \] which is ind-proper. It follows that $\oblv_{\mathfrak{L}^+G}(\delta_{\mathfrak{L}^+ G,X^I})$ is compact in $\DD(\Gr_G)_{X^I}$, and hence $\delta_{\mathfrak{L}^+ G,X^I}$ is almost compact in $\DD(\sH_G)_{X^I}$ by Lemma \ref{aclem}.

\end{proof}

\begin{proposition}

\label{sphren}

The t-structure on $\DD(\sH_G)_{X^I_{\dR}}$ satisfies the hypotheses of Proposition \ref{factcatren}. Moreover, the object \[ \Sph_G := \DD(\sH_G)^{\ren} \] belongs to $\FactCat$ and admits a unique associative algebra structure compatible with \[ \Sph_G \longrightarrow \DD(\sH_G). \]

\end{proposition}

\begin{proof}

We first show that that the t-structure on $\DD(\sH_G)_{X^I_{\dR}}$ is coherent, or equivalently that any object compact in $\DD(\sH_G)_{X^I_{\dR}}^{\geq 0}$ is almost compact in $\DD(\sH_G)_{X^I_{\dR}}$. The t-structure on $\DD(\Gr_{G,X^I})$ is coherent because $\Gr_{G,X^I}$ is an ind-scheme locally of finite type. Thus the claim follows from Lemma \ref{aclem} applied to the functor \[ \oblv_{\mathfrak{L}^+G} : \DD(\sH_G)_{X^I_{\dR}} = \DD(\sH_{G,X^I}) \longrightarrow \DD(\Gr_{G,X^I}), \] which is t-exact, conservative, and admits a continuous right adjoint.

The functor \[ \Sat^{\nv}_G : \Rep(\check{G})_{X^I_{\dR}} \longrightarrow \DD(\sH_G)_{X^I_{\dR}} \] induces an equivalence on the hearts and preserves almost ULA objects by Lemma \ref{naivesatula}, and $\Rep(\check{G})_{X^I_{\dR}}$ is ULA generated by Proposition \ref{repulagen}. It follows readily that $\DD(\sH_G)_{X^I_{\dR}}$ is almost ULA generated. Conditions (ii-iii) of Proposition \ref{sheafranren}, as well as the hypothesis of Proposition \ref{factcatren}, can be deduced from the corresponding properties of $\DD(\Gr_{G,X^I})$.

By Proposition \ref{repulagen}, renormalizing $\Sat^{\nv}_G$ yields a morphism \[ \Rep(\check{G}) \longrightarrow \Sph_G \] in $\AssocAlg(\FactCat^{\laxfact})$. Moreover, by Lemma \ref{naivesatula}, this morphism admits a right adjoint in $\FactCat^{\laxfact}_{\laxuntl}$. Since $\Sat^{\nv}_G$ is an equivalence on the hearts, this right adjoint is conservative and hence monadic over each $X^I_{\dR}$. It follows that $\Sph_G$ belongs to $\FactCat$, i.e., factorizes strictly, since $\Rep(\check{G})$ has this property.

For the last claim, by Proposition \ref{monfactcatren} it suffices to show that $\DD(\sH_G)$ is an associative algebra in the pseudo-tensor category $(\FactCat^{\laxfact})^{\otimes}_{\aULA}$. This follows from right completeness and the corresponding property of $\Rep(\check{G})$, once more using the fact that the monoidal functor \[ \Sat^{\nv}_G : \Rep(\check{G})_{X^I_{\dR}} \longrightarrow \Sph_{G,X^I_{\dR}} \] is t-exact and induces an equivalence on the hearts for any finite set $I$.

\end{proof}

\begin{rem}

The above result may be compared with \cite{nocera} Corollary 1.2
in the Betti setting.

\end{rem}

\subsection{Whittaker (co)invariants} We have an isomorphism \[ N^-/[N^-,N^-] \cong \prod_{i \in \sI_G} \bG_a^{\Omega^1_X} \] of group schemes over $X$, where \[ \bG_a^{\Omega^1_X} := \bG_a \overset{\bG_m}{\times} \Omega^1_X \] is the canonical twist of $\bG_a$. Summing over $\sI_G$, we obtain a homomorphism $N^- \to \bG_a^{\Omega^1_X}$ and hence a homomorphism of factorizable group ind-schemes \[ \mathfrak{L}N^- \longrightarrow \mathfrak{L}(\bG_a^{\Omega^1_X}). \] Finally, we compose with the canonical residue homomorphism \[ \mathfrak{L}(\bG_a^{\Omega^1_X}) \longrightarrow \bG_a \] to obtain \[ \psi : \mathfrak{L}N^- \longrightarrow \bG_a. \]

In particular, we obtain a multiplicative D-module (a.k.a. character D-module of rank 1) $\psi^!\exp[1]$ on $\mathfrak{L}N^-$. For any category $\sC$ acted on by $\DD((\mathfrak{L}G)_{X^I})$, we can consider the associated Whittaker invariants \[ \sC^{\mathfrak{L}N^-,\psi} := \sC^{\mathfrak{L}N^-,\psi^!\exp[1]}, \] and similarly for coinvariants (we suppress $X^I$ from the notation here for simplicity). Theorem 2.2.1 of \cite{whit} asserts that there is a canonical equivalence \[ \sC_{\mathfrak{L}N^-,\psi} \tilde{\longrightarrow} \sC^{\mathfrak{L}N^-,\psi}. \]

\subsection{The spherical Whittaker category} As shown in \cite{R3} Corollary 2.31.2, the $\DD(X^I)$-module categories \[ \DD(\Gr_{G,X^I})^{\mathfrak{L}N^-,\psi} \] assemble into an object \[ \DD(\Gr_G)^{\mathfrak{L}N^-,\psi} \] of $\FactCat$, with unit given by the so-called Whittaker vacuum object. Convolution defines a right action of $\DD(\sH_G)$ on this category, and in particular, we obtain a morphism \[ \DD(\sH_G) \longrightarrow \DD(\Gr_G)^{\mathfrak{L}N^-,\psi} \] in $\FactCat$ by acting on the unit.

\begin{theorem}

\label{csformula}

The composite functor \[ \Rep(\check{G}) \xrightarrow{\Sat^{\nv}_G} \DD(\sH_G) \longrightarrow \DD(\Gr_G)^{\mathfrak{L}N^-,\psi} \] is an isomorphism in $\FactCat$.

\end{theorem}

\begin{proof}

This is Theorem 6.36.1 of \cite{R2}.

\end{proof}

In particular, the equivalence in the theorem induces a right action of $\DD(\sH_G)$ on \[ \Rep(\check{G}) \cong \DD(\Gr_G)^{\mathfrak{L}N^-,\psi} \] in $\FactCat$. By Proposition \ref{ulagenprop}, this action corresponds to a morphism
\begin{equation}
\label{sphtobimod}
\DD(\sH_G) \longrightarrow \underline{\End}_{\FactCat_{\laxuntl}}(\Rep(\check{G})) \cong \Rep(\check{G} \times \check{G})
\end{equation}
in $\AssocAlg(\FactCat_{\laxuntl})$.

We will need the following observation.

\begin{lemma}

For any finite set $I$, the functor \[ \DD(\sH_G)_{X^I_{\dR}} \xrightarrow{(\ref{sphtobimod})} \Rep(\check{G} \times \check{G})_{X^I_{\dR}} \] is t-exact.

\label{sphtobimodexact}

\end{lemma}

\begin{proof}

First, we claim that the composite functor \[ \Rep(\check{G})_{X^I_{\dR}} \xrightarrow{\Sat^{\nv}_G} \DD(\sH_G)_{X^I_{\dR}} \longrightarrow \Rep(\check{G} \times \check{G})_{X^I_{\dR}} \] is t-exact. By the construction of (\ref{sphtobimod}), this composite agrees with the induction functor $\coind_{\check{G}}^{\check{G} \times \check{G}}$, i.e., the right adjoint to restriction along the diagonal homomorphism $\check{G} \to \check{G} \times \check{G}$. The t-exactness of $\coind_{\check{G}}^{\check{G} \times \check{G}}$ be proved in the same way as (or deduced from) that of $\coind_1^{\check{G}}$, which is Proposition 6.24.2 of \cite{R2}.

This immediately implies that the functor in the proposition is right t-exact, since $\DD(\sH_G)_{X^I_{\dR}}^{\leq 0}$ is generated under colimits by \[ \DD(\sH_G)_{X^I_{\dR}}^{\heartsuit} \cong \Rep(\check{G})_{X^I_{\dR}}^{\heartsuit}. \] Right completeness of the t-structure implies that $\DD(\sH_G)_{X^I_{\dR}}^{\geq 0}$ is generated by $\DD(\sH_G)_{X^I_{\dR}}^{\heartsuit}$ under filtered colimits, shifts, and extensions, whence the functor in question is left t-exact.

\end{proof}

\subsection{Derived Satake transform} Note that the functor (\ref{sphtobimod}) is not strictly unital with respect to the factorization structure. Namely, it preserves monoidal units, but the monoidal and factorization units in $\Rep(\check{G} \times \check{G})$ do not coincide.

Applying the functor (\ref{factmodunit}), we obtain a morphism
\begin{equation}
\label{presatake}
\DD(\sH_G) \longrightarrow \sO_{\check{G}}\modfact(\Rep(\check{G} \times \check{G}))
\end{equation}
in $\AssocAlg(\FactCat^{\laxfact})$

Observe that the monoidal functor \[ \DD(\sH_G)_{X^I_{\dR}} \xrightarrow{(\ref{presatake})} \sO_{\check{G}}\modfact(\Rep(\check{G} \times \check{G}))_{X^I_{\dR}} \] is t-exact for any finite set $I$, since the forgetful functor \[ \sO_{\check{G}}\modfact(\Rep(\check{G} \times \check{G}))_{X^I_{\dR}} \longrightarrow \Rep(\check{G} \times \check{G})_{X^I_{\dR}} \] is t-exact and conservative, and its composition with (\ref{presatake}) is t-exact by Proposition \ref{sphtobimodexact}.

By Propositions \ref{monfactcatren} and \ref{sphren}, together with Corollary \ref{factmodcoh}, we can renormalize (\ref{presatake}) to obtain a morphism \[ \Sat_G : \Sph_G \longrightarrow \Sph_{\check{G}}^{\spec} \] in $\AssocAlg(\FactCat)$. Now we can state our first main theorem precisely.

\begin{theorem}

\label{mainthm}

The functor $\Sat_G$ is an equivalence.

\end{theorem}

This theorem will be proved in \S \ref{s:mainthmproof}.

\subsection{The semi-infinite Whittaker category} Recall from \cite{R3} \S 2.34 that the categories \[ ((\DD(\mathfrak{L}G)_{\mathfrak{L}N_P\mathfrak{L}^+M})^{\mathfrak{L}N^-,\psi})_{X^I_{\dR}} \] assemble into an object of $\FactCat$, denoted there by $\Whit^{\frac{\infty}{2}}$. Following \emph{loc. cit.} in the case $P = B$, we will now construct a morphism
\begin{equation}
\label{semiinfres}
(\DD(\mathfrak{L}G)_{\mathfrak{L}N_P\mathfrak{L}^+M})^{\mathfrak{L}N^-,\psi} \longrightarrow \Rep(\check{M})
\end{equation}
in $\FactCat_{\laxuntl}$.

First, we restrict the Whittaker condition to the subgroup $N^-_M := N^- \cap M$, i.e., apply the forgetful functor \[ (\DD(\mathfrak{L}G)_{\mathfrak{L}N_P\mathfrak{L}^+M})^{\mathfrak{L}N^-,\psi} \longrightarrow (\DD(\mathfrak{L}G)_{\mathfrak{L}N_P\mathfrak{L}^+M})^{\mathfrak{L}N_M^-,\psi}. \] At the next stage, restrict along the inclusion $\mathfrak{L}P \to \mathfrak{L}G$ to obtain \[ (\DD(\mathfrak{L}G)_{\mathfrak{L}N_P\mathfrak{L}^+M})^{\mathfrak{L}N_M^-,\psi} \longrightarrow (\DD(\mathfrak{L}P)_{\mathfrak{L}N_P\mathfrak{L}^+M})^{\mathfrak{L}N_M^-,\psi} \cong \DD(\Gr_M)^{\mathfrak{L}N_M^-,\psi}. \]

Recall that for any finite set $I$ there is a degree map \[ \deg_M : \Gr_{M,X^I} \longrightarrow \pi_1(M), \] compatible with factorization, where $\pi_1(M)$ is viewed as discrete abelian group. We can identify \[ \pi_1(M) = \Lambda/(\bZ \cdot R_M), \] the lattice of coweights modulo the coroot lattice of $M$. Note that the homomorphism
\begin{align*}
\ell_P : \pi_1(M) &\longrightarrow \bZ \\
\lambda &\mapsto \langle 2\check{\rho}_G - 2\check{\rho}_M,\lambda \rangle
\end{align*}
is well-defined, and hence we obtain an automorphism
\begin{align*}
\DD(\Gr_M) &\longrightarrow \DD(\Gr_M) \\
\sM &\mapsto \sM[-\ell_P(\deg_M)]
\end{align*}
in $\FactCat$ given by a cohomological shift which depends on the connected component.

Since $\mathfrak{L}N^-_M$ is connected, this automorpism of $\DD(\Gr_M)$ preserves the Whittaker condition. Compose this automorphism with the previously constructed morphism \[ (\DD(\mathfrak{L}G)_{\mathfrak{L}N_P\mathfrak{L}^+M})^{\mathfrak{L}N^-,\psi} \longrightarrow \DD(\Gr_M)^{\mathfrak{L}N_M^-,\psi} \] to obtain another such morphism. Finally, applying Theorem \ref{csformula} to $M$, we have \[ \DD(\Gr_M)^{\mathfrak{L}N_M^-,\psi} \tilde{\longrightarrow} \Rep(\check{M}). \] This completes the construction of the morphism (\ref{semiinfres}).

\begin{theorem}

\label{semiinfresvac}

The morphism (\ref{semiinfres}) sends the unit object to the factorization algebra $\Upsilon(\check{\mathfrak{n}}_{\check{P}})$ in $\Rep(\check{M})$.

\end{theorem}

\begin{proof}

In the case $P = G$, the morphism (\ref{semiinfres}) is inverse to the equivalence of Theorem \ref{csformula}, and $\Upsilon(\check{\mathfrak{n}}_{\check{P}})$ is the unit object in $\Rep(\check{M}) = \Rep(\check{G})$, so there is nothing to show.

The case $P = B$ is Theorem 4.4.1 of \cite{R3}.

The general case is proved in \cite{FH}.

\end{proof}

In particular, applying the construction (\ref{factmodfunct}) to (\ref{semiinfres}), we obtain a morphism
\begin{equation}
\label{semiinfresenh}
(\DD(\mathfrak{L}G)_{\mathfrak{L}N_P\mathfrak{L}^+M})^{\mathfrak{L}N^-,\psi} \longrightarrow \Upsilon(\check{\mathfrak{n}}_{\check{P}})\modfact(\Rep(\check{M}))
\end{equation}
in $\FactCat$.

\subsection{The accessible Whittaker category} The morphism \[ (\DD(\mathfrak{L}G)_{\mathfrak{L}N_P\mathfrak{L}^+M})^{\mathfrak{L}N^-,\psi} \xrightarrow{\oblv_{\mathfrak{L}N^-,\psi}} \DD(\mathfrak{L}G)_{\mathfrak{L}N_P\mathfrak{L}^+M} \xrightarrow{\Av^{\mathfrak{L}^+G}_*} \DD(\mathfrak{L}^+G \backslash \mathfrak{L}G)_{\mathfrak{L}N_P\mathfrak{L}^+M} \] in $\FactCat_{\laxuntl}$ admits a left adjoint (see \cite{R3} Proposition 5.4.1), which we denote by $\Av^{\mathfrak{L}N^-,\psi}_!$. Moreover, by the construction of unital structures in \emph{loc. cit.}, the morphism $\Av^{\mathfrak{L}N^-,\psi}_!$ is strictly unital.

Define the factorization subcategory \[ (\DD(\mathfrak{L}G)_{\mathfrak{L}N_P\mathfrak{L}^+M})^{\mathfrak{L}N^-,\psi,\acc} \subset (\DD(\mathfrak{L}G)_{\mathfrak{L}N_P\mathfrak{L}^+M})^{\mathfrak{L}N^-,\psi} \] to be generated under colimits by the essential image of \[ \Av^{\mathfrak{L}N^-,\psi}_! : \DD(\mathfrak{L}^+G \backslash \mathfrak{L}G)_{\mathfrak{L}N_P\mathfrak{L}^+M} \longrightarrow (\DD(\mathfrak{L}G)_{\mathfrak{L}N_P\mathfrak{L}^+M})^{\mathfrak{L}N^-,\psi}. \] As explained in \emph{loc. cit.}, this subcategory admits a unique unital factorization structure which makes the inclusion a morphism in $\FactCat$.

\begin{theorem}

\label{semiinfresequiv}

The morphism (\ref{semiinfresenh}) restricts to an equivalence \[ (\DD(\mathfrak{L}G)_{\mathfrak{L}N_P\mathfrak{L}^+M})^{\mathfrak{L}N^-,\psi,\acc} \tilde{\longrightarrow} \Upsilon(\check{\mathfrak{n}}_{\check{P}})\mod_0^{\fact}(\Rep(\check{M})). \]

\end{theorem}

\begin{proof}

In the case $P = G$, the accessibility condition is vacuous, and this is inverse to the equivalence of Theorem \ref{csformula}.

In the case $P = B$, Theorem 5.7.1 of \cite{R3} says that this functor is fully faithful. To prove essential surjectivity, we claim that for arbitrary $P$, the image of the functor generates \[ \Upsilon(\check{\mathfrak{n}}_{\check{P}})\mod_0^{\fact}(\Rep(\check{M})) \] under colimits. The functor in question is $\Rep(\check{M})$-linear, where $\Rep(\check{M})$ acts on \[ (\DD(\mathfrak{L}G)_{\mathfrak{L}N_P\mathfrak{L}^+M})^{\mathfrak{L}N^-,\psi,\acc} \] via $\Sat_M^{\nv}$. We have a commutative triangle \[
\begin{tikzcd}
& \Rep(\check{M}) \arrow{dl} \arrow{dr} & \\
(\DD(\mathfrak{L}G)_{\mathfrak{L}N_P\mathfrak{L}^+M})^{\mathfrak{L}N^-,\psi,\acc} \arrow{rr} & & \Upsilon(\check{\mathfrak{n}}_{\check{P}})\mod_0^{\fact}(\Rep(\check{M}))
\end{tikzcd} \]
in $\FactCat$, where both diagonal morphisms are given by acting on the unit for the factorization structure. Similarly, the triangle \[
\begin{tikzcd}
& \Rep(\check{M}) \arrow{dl}[swap]{\triv_{\check{N}_{\check{P}}}} \arrow{dr} & \\
\Rep(\check{P}) \arrow{rr}{(\ref{chiralind2})} & & \Upsilon(\check{\mathfrak{n}}_{\check{P}})\mod_0^{\fact}(\Rep(\check{M}))
\end{tikzcd} \]
commutes because both circuits are $\Rep(\check{M})$-linear morphisms in $\FactCat$. Since the images of $\triv_{\check{N}_{\check{P}}}$ and (\ref{chiralind2}) generate their targets under colimits, the claim follows.

For the general case, by the argument above it suffices to show fully faithfulness. This can be deduced from Theorem \ref{semiinfresvac} by adapting the argument of \cite{R3} to the parabolic case. One also needs a parabolic version of the Arkhipov-Bezrukavnikov equivalence \cite{AB}, which can be deduced from the case $P = B$, as explained in \S 6 of \cite{bezrukavnikov-losev} (cf. also \cite{ACR}).

\end{proof}

\subsection{The t-structure on the semi-infinite spherical category} We now prepare to define the spherical Hecke category $\Sph_{G,P}$ attached to $P$. Appealing again to \cite{R3} \S 2, we recall that the categories \[ (\DD(\mathfrak{L}^+G \backslash \mathfrak{L}G)_{\mathfrak{L}N_P\mathfrak{L}^+M})_{X^I_{\dR}} \] assemble into an object of $\FactCat$. We will denote the unit object by $\Delta^0$.

Consider the correspondence of corr-unital factorization spaces \[
\begin{tikzcd}
& \mathfrak{L}^+P \backslash \mathfrak{L}P \arrow{dl}[swap]{\mathfrak{p}} \arrow{dr}{\mathfrak{q}} & \\
\mathfrak{L}^+G \backslash \mathfrak{L}G & & \mathfrak{L}^+M \backslash \mathfrak{L}M.
\end{tikzcd} \]
The composite morphism \[ \DD(\sH_M) \xrightarrow{\mathfrak{q}^!} \DD(\mathfrak{L}^+P \backslash \mathfrak{L}P/\mathfrak{L}^+M) \longrightarrow \DD(\mathfrak{L}^+P \backslash \mathfrak{L}P)_{\mathfrak{L}N_P\mathfrak{L}^+M} \] is an isomorphism in $\FactCat$. We will abuse notation and denote it simply by $\mathfrak{q}^!$, and its inverse by $\mathfrak{q}_!$.

The morphism \[ \mathfrak{p}^! : \DD(\mathfrak{L}^+G \backslash \mathfrak{L}G)_{\mathfrak{L}N_P\mathfrak{L}^+M} \longrightarrow \DD(\mathfrak{L}^+P \backslash \mathfrak{L}P)_{\mathfrak{L}N_P\mathfrak{L}^+M} \] admits a left adjoint $\mathfrak{p}_!$. Explicitly, the composite \[ \DD(\sH_M) \xrightarrow{\mathfrak{q}^!} \DD(\mathfrak{L}^+P \backslash \mathfrak{L}P)_{\mathfrak{L}N_P\mathfrak{L}^+M} \xrightarrow{\mathfrak{p}_!} \DD(\mathfrak{L}^+G \backslash \mathfrak{L}G)_{\mathfrak{L}N_P\mathfrak{L}^+M} \] is given by acting on the unit object.

Now consider the adjunction
\begin{equation}
\label{semiinfgen}
\mathfrak{p}_! \mathfrak{q}^![-\ell_P(\deg_M)] : \DD(\sH_M) \rightleftarrows \DD(\mathfrak{L}^+G \backslash \mathfrak{L}G)_{\mathfrak{L}N_P \mathfrak{L}^+M} : \mathfrak{q}_! \mathfrak{p}^![\ell_P(\deg_M)],
\end{equation}
which takes place in $\FactCat_{\laxuntl}$, with the left adjoint being strictly unital. We also remark that the right adjoint is conservative, which follows from the fact that $\mathfrak{p}$ is an ind-locally closed stratification.

We equip $(\DD(\mathfrak{L}^+G \backslash \mathfrak{L}G)_{\mathfrak{L}N_P \mathfrak{L}^+M})_{X^I_{\dR}}$ with the t-structure characterized by the requirement that \[ \mathfrak{q}_! \mathfrak{p}^![\ell_P(\deg_M)] : (\DD(\mathfrak{L}^+G \backslash \mathfrak{L}G)_{\mathfrak{L}N_P \mathfrak{L}^+M})_{X^I_{\dR}} \longrightarrow \DD(\sH_{M})_{X^I_{\dR}} \] be left t-exact. Equivalently, the connective objects are generated by the image of $\DD(\sH_M)_{X^I_{\dR}}^{\leq 0}$ under $\mathfrak{p}_! \mathfrak{q}^![-\ell_P(\deg_M)]$.

We remark that this does \emph{not} agree with the t-structure on this category defined (in the case $P = B$) by Gaitsgory in \cite{G2}. Namely, unlike the t-structure introduced in \emph{loc. cit.}, our t-structure is local on the curve $X$ (since the t-structure on \[ \DD(\sH_M)_{X^I_{\dR}} = \DD(\sH_{M,X^I}) \] has this feature). However, the two t-structures do agree when restricted to objects supported at a single point $x \in X(k)$.

\begin{lemma}

\label{semiinftstrlem}

For any finite set $I$, the functor \[ \mathfrak{p}_! \mathfrak{q}^![-\ell_P(\deg_M)] : \DD(\sH_M)_{X^I_{\dR}} \longrightarrow (\DD(\mathfrak{L}^+G \backslash \mathfrak{L}G)_{\mathfrak{L}N_P \mathfrak{L}^+M})_{X^I_{\dR}} \] is t-exact.

\end{lemma}

\begin{proof}

This functor is right t-exact by definition, so it suffices to show left t-exactness. Since its right adjoint is conservative, it is enough to prove that the resulting monad on $\DD(\sH_M)_{X^I_{\dR}}$ is left t-exact. This monad is $\DD(\sH_M)_{X^I_{\dR}}$-linear, so its value on an object $\sM$ in $\DD(\sH_M)_{X^I_{\dR}}$ is \[ (\mathfrak{q}_! \mathfrak{p}^!\Delta^0_{X^I})[\ell_P(\deg_M)] \star \sM. \]

Using $\DD(X^I)$-linearity, we can reduce to the case that $I$ has one element. So we must show that if $\sM$ belongs to $\DD(\sH_M)_{X_{\dR}}^{\geq 0}$, then \[ (\mathfrak{q}_! \mathfrak{p}^!\Delta^0_X)[\ell_P(\deg_M)] \star \sM \] is coconnective. Since $\DD(\sH_M)_{X_{\dR}}$ is right complete, we can even assume that $\sM$ belongs to \[ \DD(\sH_M)_{X_{\dR}}^{\heartsuit} \cong \Rep(\check{M})_{X_{\dR}}^{\heartsuit} \cong (\Rep(\check{M}) \otimes \DD(X))^{\heartsuit}. \]

Choose an \'{e}tale coordinate on $X$, so that $\Delta^0_X \cong \Delta^0_x \boxtimes \omega_X$ for a fixed point $x \in X(k)$. Thus it suffices to show that for any $V \in \Rep(\check{M})^{\heartsuit}$, the object \[ (\mathfrak{q}_! \mathfrak{p}^!\Delta^0_x)[\ell_P(\deg_M)] \star \Sat_{M,x}^{\nv}(V) \] belongs to $\DD(\sH_{M,x})^{\geq 0}$. Theorem 1.5.5 in \cite{G1} says that \[ (\mathfrak{q}_! \mathfrak{p}^!\Delta^0_x)[\ell_P(\deg_M)] \] is coconnective when $P = B$, and the general case can be proved similarly. Now we are done by t-exactness of the convolution product on $\DD(\sH_{M,x})$.

\end{proof}

\subsection{The t-structure on the semi-infinite Whittaker category} Consider the morphism
\begin{equation}
\label{whitgen}
\Rep(\check{M}) \longrightarrow (\DD(\mathfrak{L}N^-,\psi \backslash \mathfrak{L}G)_{\mathfrak{L}N_P\mathfrak{L}^+M})_{\acc}
\end{equation}
in $\FactCat$ given by acting on the unit of the factorization structure via $\Sat_M^{\nv}$. For any finite set $I$, we equip the category \[ (\DD(\mathfrak{L}G)_{\mathfrak{L}N_P\mathfrak{L}^+M})^{\mathfrak{L}N^-,\psi,\acc}_{X^I_{\dR}} \] with the t-structure whose connective objects are generated under colimits by the image of $\Rep(\check{M})^{\leq 0}$ under (\ref{whitgen}).

\begin{proposition}

The equivalence of Theorem \ref{semiinfresequiv} is t-exact over $X^I_{\dR}$ for every finite set $I$.

\end{proposition}

\begin{proof}

By definition, the connective objects in \[ \Upsilon(\check{\mathfrak{n}}_{\check{P}})\mod_0^{\fact}(\Rep(\check{M}))_{X^I_{\dR}} \] are generated by the image of $\Rep(\check{P})_{X^I_{\dR}}^{\leq 0}$ under the left adjoint in (\ref{chiralind2}). But the latter is generated by the image of $\Rep(\check{M})_{X^I_{\dR}}^{\leq 0}$ under \[ \triv_{\check{N}} : \Rep(\check{M})_{X^I_{\dR}} \longrightarrow  \Rep(\check{P})_{X^I_{\dR}}, \] so it suffices to show that the triangle \[
\begin{tikzcd}
& \Rep(\check{M}) \arrow{dl}[swap]{(\ref{whitgen})} \arrow{dr}{\ind^{* \to \ch} \circ \triv_{\check{N}}} & \\
(\DD(\mathfrak{L}N^-,\psi \backslash \mathfrak{L}G)_{\mathfrak{L}N_P\mathfrak{L}^+M})_{\acc} \arrow{rr}{\sim} & & \Upsilon(\check{\mathfrak{n}}_{\check{P}})\mod_0^{\fact}(\Rep(\check{M}))
\end{tikzcd} \]
in $\FactCat$ commutes. This is immediate from the fact that all three functors are $\Rep(\check{M})$-linear and unital.

\end{proof}

\begin{proposition}

\label{whitgenexact}

The morphism (\ref{whitgen}) is t-exact over $X^I_{\dR}$ for any finite set $I$ and preserves ULA objects. The resulting right adjoint \[ (\DD(\mathfrak{L}G)_{\mathfrak{L}N_P\mathfrak{L}^+M})^{\mathfrak{L}N^-,\psi,\acc} \longrightarrow \Rep(\check{M}) \] in $\FactCat_{\laxuntl}$ is conservative over $X^I_{\dR}$ for any finite set $I$.

\end{proposition}

\begin{proof}

These are consequences of Theorem \ref{semiinfresequiv}. Alternatively, one can prove the t-exactness more directly along the same lines as Lemma \ref{semiinftstrlem}, and the existence of a factorizable and conservative right adjoint follows from Theorem 5.10.1 in \cite{R3}.

\end{proof}

\subsection{Definition of the spherical semi-infinite category} We will need a couple of technical lemmas.

\begin{lemma}

\label{whitavtexact}

For any finite set $I$, the functor \[ (\DD(\mathfrak{L}^+G \backslash \mathfrak{L}G)_{\mathfrak{L}N_P \mathfrak{L}^+M})_{X^I_{\dR}} \xrightarrow{\Av^{\mathfrak{L}N^-,\psi}_!} (\DD(\mathfrak{L}G)_{\mathfrak{L}N_P\mathfrak{L}^+M})^{\mathfrak{L}N^-,\psi,\acc}_{X^I_{\dR}} \] is t-exact. Moreover, this functor is conservative on eventually coconnective objects.

\end{lemma}

\begin{proof}

To see that the functor is right t-exact, consider the commutative triangle \[
\begin{tikzcd}
& \Rep(\check{M}) \arrow{dl} \arrow{dr} & \\
\DD(\mathfrak{L}^+G \backslash \mathfrak{L}G)_{\mathfrak{L}N_P \mathfrak{L}^+M} \arrow{rr} & & (\DD(\mathfrak{L}G)_{\mathfrak{L}N_P\mathfrak{L}^+M})^{\mathfrak{L}N^-,\psi,\acc}
\end{tikzcd} \]
in $\FactCat$, where both diagonal morphisms are given by acting on the unit objects. Over $X^I_{\dR}$, the diagonal functors are t-exact and generate their targets under colimits, which implies that the horizontal functor is right t-exact.

For left t-exactness, the Cousin filtration and factorization immediately reduce the claim to the case where $I$ has one element. After \'{e}tale localization on $X$, we have a $\DD(X)$-linear equivalence \[ (\DD(\mathfrak{L}^+G \backslash \mathfrak{L}G)_{\mathfrak{L}N_P \mathfrak{L}^+M})_{X_{\dR}} \cong (\DD(\mathfrak{L}^+G \backslash \mathfrak{L}G)_{\mathfrak{L}N_P \mathfrak{L}^+M})_x \otimes \DD(X) \] for a fixed $x \in X(k)$, and similarly for \[ (\DD(\mathfrak{L}G)_{\mathfrak{L}N_P\mathfrak{L}^+M})^{\mathfrak{L}N^-,\psi,\acc}_{X_{\dR}}. \] Thus, by Lemma \ref{tensorexact} we can even reduce to proving the claim at a single point $x$. Since the t-structure on $(\DD(\mathfrak{L}^+G \backslash \mathfrak{L}G)_{\mathfrak{L}N_P \mathfrak{L}^+M})_x$ is right complete, it is enough to prove that the functor is left t-exact when restricted to the heart. By (parabolic variants of) Proposition 1.5.7 and Theorem 4.2.3 in \cite{G1}, there is a functor
\begin{equation}
\label{reppfunct}
\Rep(\check{P}) \longrightarrow (\DD(\mathfrak{L}^+G \backslash \mathfrak{L}G)_{\mathfrak{L}N_P \mathfrak{L}^+M})_x
\end{equation}
which induces an equivalence \[ \Rep(\check{P})^{\heartsuit} \tilde{\longrightarrow} (\DD(\mathfrak{L}^+G \backslash \mathfrak{L}G)_{\mathfrak{L}N_P \mathfrak{L}^+M})_x^{\heartsuit}. \] Since $\Rep(\check{P})^{\heartsuit}$ is generated under extensions and direct sums by the image of \[ \triv_{\check{N}_{\check{P}}} : \Rep(\check{M})^{\heartsuit} \longrightarrow \Rep(\check{P})^{\heartsuit}, \] the claim follows from Proposition \ref{whitgenexact}.

Finally, we prove conservativity on eventually coconnective objects. Again, using factorization we can reduce to the case that $I$ has a single element. The functor $\Av^{\mathfrak{L}N^-,\psi}_!$ is left adjoint to $\Av^{\mathfrak{L}^+G}_*$, and by the t-exactness statements proved above, the resulting monad on \[ (\DD(\mathfrak{L}^+G \backslash \mathfrak{L}G)_{\mathfrak{L}N_P \mathfrak{L}^+M})_{X_{\dR}} \] is left t-exact. Thus it suffices to show that this monad is conservative on eventually coconnective objects. Observe that the monad is given by the action of the object \[ \Av^{\mathfrak{L}^+G}_*\Av^{\mathfrak{L}N^-,\psi}_!\delta_{1,X} \] in $\DD(\sH_G)_{X_{\dR}}$. After \'{e}tale localization on $X$ we have $(\mathfrak{L}G)_X \cong (\mathfrak{L}G)_x \times X$ for a fixed $x \in X(k)$, from which it follows that the object \[ \Av^{\mathfrak{L}^+G}_*\Av^{\mathfrak{L}N^-,\psi}_!\delta_{1,X} \] has the form $\sM \boxtimes \omega_X$ with respect to the decomposition \[ \DD(\sH_G)_{X_{\dR}} = \DD(\sH_{G,X}) \cong \DD(\sH_{G,x}) \otimes \DD(X). \] The na\"{i}ve Satake equivalence implies that for any $i \in \bZ$, we have $H^i\sM \cong \Sat_{G,x}^{\nv}V$ for some $V$ in $\Rep(\check{G})^{\heartsuit}$. Since the monad in question is left t-exact, we have $H^i\sM = 0$ for $i < 0$, and Theorem \ref{csformula} implies that $H^0\sM = \delta_{1,x}$. Thus it suffices show that the action of the symmetric monoidal category $\Rep(\check{G})$ on $(D(\mathfrak{L}^+G \backslash \mathfrak{L}G)_{\mathfrak{L}N_P \mathfrak{L}^+M})_x$ is left t-exact. This follows from the fact that the functor (\ref{reppfunct}) is $\Rep(\check{G})$-linear and induces an equivalence on the hearts.

\end{proof}

\begin{lemma} 

For any finite set $I$, the t-structure on $(\DD(\mathfrak{L}^+G \backslash \mathfrak{L}G)_{\mathfrak{L}N_P \mathfrak{L}^+M})_{X^I_{\dR}}$ satisfies the conditions of Proposition \ref{factcatren}.

\end{lemma}

\begin{proof}

Most of the conditions can be easily deduced from the corresponding properties of the t-structure on $\DD(\sH_M)_{X^I_{\dR}}$, with the exception of coherence.

For coherence, we consider the adjunction of $\DD(X^I)$-module categories \[ \Av^{\mathfrak{L}N^-,\psi}_! : (\DD(\mathfrak{L}^+G \backslash \mathfrak{L}G)_{\mathfrak{L}N_P \mathfrak{L}^+M})_{X^I_{\dR}} \rightleftarrows (\DD(\mathfrak{L}G)_{\mathfrak{L}N_P\mathfrak{L}^+M})^{\mathfrak{L}N^-,\psi,\acc}_{X^I_{\dR}} : \Av^{\mathfrak{L}^+G}_*. \] Using Theorem \ref{semiinfresequiv}, we can replace the right side by its spectral counterpart: \[ (\DD(\mathfrak{L}^+G \backslash \mathfrak{L}G)_{\mathfrak{L}N_P \mathfrak{L}^+M})_{X^I_{\dR}} \rightleftarrows \Upsilon(\check{\mathfrak{n}}_{\check{P}})\mod_0^{\fact}(\Rep(\check{M}))_{X^I_{\dR}}. \] Proposition \ref{whitavtexact} says that the left adjoint is t-exact, and also conservative on eventually coconnective objects. By Lemma \ref{aclem}, an object in $(\DD(\mathfrak{L}^+G \backslash \mathfrak{L}G)_{\mathfrak{L}N_P \mathfrak{L}^+M})_{X^I_{\dR}}$ is almost compact if and only if it is so after applying the left adjoint. Now coherence of the left side follows from coherence of the right side, which was proved in Proposition \ref{upscoh}.

\end{proof}

Thus we can define \[ \Sph_{G,P} := (\DD(\mathfrak{L}^+G \backslash \mathfrak{L}G)_{\mathfrak{L}N_P \mathfrak{L}^+M})^{\ren}, \] which is \emph{a priori} an object of $\FactCat^{\laxfact}$. Note that the adjunction (\ref{semiinfgen}) renormalizes to a monadic adjunction \[ \Sph_M \rightleftarrows \Sph_{G,P}, \] and in particular $\Sph_{G,P}$ belongs to $\FactCat$ because $\Sph_M$ does.

\begin{proposition}

The object $\Sph_{G,P}$ admits a unique structure of $(\Sph_G,\Sph_M)$-bimodule in $\FactCat$ compatible with the morphism \[ \Sph_{G,P} \longrightarrow \DD(\mathfrak{L}^+G \backslash \mathfrak{L}G)_{\mathfrak{L}N_P \mathfrak{L}^+M}. \]

\end{proposition}

\begin{proof}

By Proposition \ref{monfactcatren}, it suffices to show that the $(\DD(\sH_G),\DD(\sH_M))$-bimodule structure on $\DD(\mathfrak{L}^+G \backslash \mathfrak{L}G)_{\mathfrak{L}N_P \mathfrak{L}^+M}$ takes place in the pseudo-tensor category $\FactCat^{\otimes}_{\aULA}$. For $\DD(\sH_M)$, this follows from the existence of the right adjoint in (\ref{semiinfgen}), which is $\DD(\sH_M)$-linear as well as conservative and left t-exact over $X^I_{\dR}$ for every finite set $I$. As for $\DD(\sH_G)$, by right completeness of the t-structure and the na\"{i}ve Satake equivalence, it is enough to show that any object of $\Rep(\check{G})_{X^I_{\dR}}^{\heartsuit}$ acts by a left t-exact functor. Using factorization, this immediately reduces to the corresponding claim at a point $x \in X(k)$, which was established in the proof of Lemma \ref{whitavtexact}.

\end{proof}

\subsection{Derived Satake transform for a parabolic} We begin by constructing a morphism
\begin{equation}
\label{psfunctor}
\DD(\mathfrak{L}^+ G \backslash \mathfrak{L}G)_{\mathfrak{L}N_P\mathfrak{L}^+M} \longrightarrow \Rep(\check{G} \times \check{M})
\end{equation}
in $\FactCat_{\laxuntl}$. By Proposition \ref{ulagenprop}, this is the same datum as a morphism \[ \Rep(\check{G}) \otimes \DD(\mathfrak{L}^+ G \backslash \mathfrak{L}G)_{\mathfrak{L}N_P\mathfrak{L}^+M} \longrightarrow \Rep(\check{M}). \]

By Theorem \ref{csformula}, we can identify \[ \Rep(\check{G}) \tilde{\longrightarrow} \DD(\Gr_G)^{\mathfrak{L}N^-,\psi}, \] so we will construct a morphism \[ \DD(\Gr_G)^{\mathfrak{L}N^-,\psi} \otimes \DD(\mathfrak{L}^+ G \backslash \mathfrak{L}G)_{\mathfrak{L}N_P\mathfrak{L}^+M} \longrightarrow \Rep(\check{M}). \] First, convolve in the middle to obtain a morphism \[ \DD(\Gr_G)^{\mathfrak{L}N^-,\psi} \otimes \DD(\mathfrak{L}^+ G \backslash \mathfrak{L}G)_{\mathfrak{L}N_P\mathfrak{L}^+M} \longrightarrow (\DD(\mathfrak{L}G)_{\mathfrak{L}N_P\mathfrak{L}^+M})^{\mathfrak{L}N^-,\psi}. \] Then compose with (\ref{semiinfres}) to obtain the desired morphism (\ref{psfunctor}).

\begin{theorem}

\label{psfunctorunit}

The image of the unit under the morphism (\ref{psfunctor}) is canonically isomorphic to $\Upsilon(\check{\mathfrak{n}}_{\check{P}},\sO_{\check{G}})$ as a factorization algebra.

\end{theorem}

\begin{proof}

We have already seen this in the case $P = G$, where the functor (\ref{psfunctor}) is monoidal and the $\Upsilon(\check{\mathfrak{n}}_{\check{P}},\sO_{\check{G}})$ is the monoidal unit $\sO_{\check{G}}$ in $\Rep(\check{G} \times \check{G})$.

For the general case, a parabolic variant of Theorem 7.9.1 in \cite{R2} says that the triangle \[
\begin{tikzcd}
& \Rep(\check{G}) \arrow{dl} \arrow{dr}{\Upsilon(\check{\mathfrak{n}}_{\check{P}},-)} & \\
(\DD(\mathfrak{L}G)_{\mathfrak{L}N_P\mathfrak{L}^+M})^{\mathfrak{L}N^-,\psi} \arrow{rr}{(\ref{semiinfres})} & & \Rep(\check{M})
\end{tikzcd} \]
in $\FactCat_{\laxuntl}$ commutes. Here the left morphism is the composite \[ \Rep(\check{G}) \longrightarrow \DD(\mathfrak{L}^+G \backslash \mathfrak{L}G)_{\mathfrak{L}N_P\mathfrak{L}^+M} \xrightarrow{\Av_!^{\mathfrak{L}N^-,\psi}} (\DD(\mathfrak{L}G)_{\mathfrak{L}N_P\mathfrak{L}^+M})^{\mathfrak{L}N^-,\psi}, \] where the first functor is given by acting on the unit via $\Sat^{\nv}_G$. Since $\Upsilon(\check{\mathfrak{n}},-)$ corresponds to $\Upsilon(\check{\mathfrak{n}}_{\check{P}},\sO_{\check{G}})$ under the equivalence \[ \uHom_{\FactCat_{\laxuntl}}(\Rep(\check{G}),\Rep(\check{M})) \cong \Rep(\check{G} \times \check{M}), \] the claim follows.

\end{proof}

Applying (\ref{factmodfunct}), we therefore obtain a morphism
\begin{equation}
\label{psfunc1}
\DD(\mathfrak{L}^+G \backslash \mathfrak{L}G)_{\mathfrak{L}N_P \mathfrak{L}^+M} \longrightarrow \Upsilon(\check{\mathfrak{n}}_{\check{P}},\sO_{\check{G}})\modfact(\Rep(\check{G} \times \check{M})).
\end{equation}
in $\FactCat^{\laxfact}$. By construction, it is a morphism of $(\DD(\sH_G),\DD(\sH_M))$-bimodules, where the bimodule structure on the right side comes from the action of \[ (\sO_{\check{G}}\modfact(\Rep(\check{G} \times \check{G})),\sO_{\check{M}}\modfact(\Rep(\check{M} \times \check{M}))) \] via (\ref{presatake}).

\begin{proposition}

The image of the morphism (\ref{psfunc1}) is contained in \[ \Upsilon(\check{\mathfrak{n}}_{\check{P}},\sO_{\check{G}})\mod_0^{\fact}(\Rep(\check{G} \times \check{M})). \] For any finite set $I$, the resulting functor \[ (\DD(\mathfrak{L}^+G \backslash \mathfrak{L}G)_{\mathfrak{L}N_P \mathfrak{L}^+M})_{X^I_{\dR}} \longrightarrow \Upsilon(\check{\mathfrak{n}}_{\check{P}},\sO_{\check{G}})\mod_0^{\fact}(\Rep(\check{G} \times \check{M}))_{X^I_{\dR}} \] is t-exact.  

\end{proposition}

\begin{proof}

The first claim follows from the fact that the image of the left adjoint in (\ref{semiinfgen}) generates the target under colimits, and hence the image of (\ref{psfunc1}) is contained in the $\DD(\sH_M)$-submodule category generated by the unit.

Right t-exactness of the functor in question follows from the commutative triangle \[
\begin{tikzcd}
& \Rep(\check{M}) \arrow{dl} \arrow{dr} & \\
\DD(\mathfrak{L}^+G \backslash \mathfrak{L}G)_{\mathfrak{L}N_P \mathfrak{L}^+M} \arrow{rr} & &  \Upsilon(\check{\mathfrak{n}}_{\check{P}},\sO_{\check{G}})\mod_0^{\fact}(\Rep(\check{G} \times \check{M})),
\end{tikzcd} \]
in $\FactCat$, where the diagonal morphisms are uniquely characterized by $\Rep(\check{M})$-linearity. Namely, over each $X^I_{\dR}$, both diagonal functors are t-exact and generate their targets under colimits.

Using factorization, left t-exactness reduces to the corresponding claim at a point $x \in X$. By right completeness of the t-structure, it suffices to show that (\ref{psfunc1}) sends any object in \[ (\DD(\mathfrak{L}^+G \backslash \mathfrak{L}G)_{\mathfrak{L}N_P \mathfrak{L}^+M})_x^{\heartsuit} \] to a coconnective object. As in the proof of Proposition \ref{whitavtexact}, we have an equivalence \[ (\DD(\mathfrak{L}^+G \backslash \mathfrak{L}G)_{\mathfrak{L}N_P \mathfrak{L}^+M})_x^{\heartsuit} \cong \Rep(\check{P})^{\heartsuit}, \] and $\Rep(\check{P})^{\heartsuit}$ is generated under extensions and filtered colimits by the image of \[ \triv_{\check{N}_{\check{P}}} : \Rep(\check{M})^{\heartsuit} \longrightarrow \Rep(\check{P})^{\heartsuit}. \] In view of the commutative triangle above, the claim follows from Lemma \ref{psequivheart}.

\end{proof}

Thus (\ref{psfunc1}) renormalizes to a morphism \[ \Sat_{G,P} : \Sph_{G,P} \longrightarrow \Sph_{\check{G},\check{P}}^{\spec} \] in $\FactCat$. By construction, it is a morphism of $(\Sph_G,\Sph_M)$-bimodules, where the bimodule structure on the target comes from $\Sat_G$ and $\Sat_M$.

The following is our second main theorem.

\begin{theorem}

\label{mainthmps}

The morphism $\Sat_{G,P}$ is an isomorphism.

\end{theorem}

The proof of this theorem is contained in \S \ref{s:psequiv}. Note that Theorem \ref{mainthm} is the case $G = P$.

\subsection{Reduction to a point} \label{pointred} Before proceeding, we show that in order to prove Theorem \ref{mainthmps} (and hence its special case Theorem \ref{mainthm}), it is enough to show that \[ \Sat_{G,P,x} : \Sph_{G,P,x} \longrightarrow \Sph_{\check{G},\check{P},x}^{\spec} \] is an equivalence for any $x \in X(k)$.

Consider the commutative triangle \[
\begin{tikzcd}
& \Rep(\check{M}) \arrow{dl} \arrow{dr} & \\
\Sph_{G,P} \arrow{rr}{\Sat_{G,P}} & & \Sph_{\check{G},\check{P}}^{\spec}
\end{tikzcd} \]
in $\FactCat$, where the diagonal morphisms are uniquely characterized by $\Rep(\check{M})$-linearity. For any finite set $I$, the diagonal functors preserve ULA objects over $X^I$ and generate their targets, which implies that the functor $\Sat_{G,P}$ preserves ULA objects. Thus, by Proposition B.8.1 in \cite{R2}, it suffices to show that this functor is an equivalence over each stratum in $X^I_{\dR}$. Using factorization, we can therefore reduce to the case that $I$ is a singleton, i.e., we must show that $\Sat_{G,P,X_{\dR}}$ is an equivalence.

Choosing an \'{e}tale coordinate on $X$, we can moreover assume that we have $\DD(X)$-linear equivalences \[ \Sph_{G,P,X_{\dR}} \cong \Sph_{G,P,x} \otimes \DD(X) \] and \[ \Sph_{\check{G},\check{P},X_{\dR}}^{\spec} \cong \Sph_{\check{G},\check{P},x}^{\spec} \otimes \DD(X) \] for a fixed $x \in X(k)$, and that the $\DD(X)$-linear functor $\Sat_{G,P,X_{\dR}}$ is induced from $\Sat_{G,P,x}$. Thus it suffices to show that \[ \Sat_{G,P,x} : \Sph_{G,P,x} \longrightarrow \Sph_{\check{G},\check{P},x}^{\spec} \] is an equivalence.

\section{The case of a torus}\label{s:torus}

The goal of this section is to formulate and prove Theorem \ref{lgcft}, a factorizable version of local geometric class field theory in the de Rham setting, and then deduce Theorem \ref{mainthm} in the case that $G=T$ is a torus.

\subsection{Gauge forms} Let $H$ denote an arbitrary algebraic group. The space $\mathfrak{h} \otimes \Omega^1_D$ over $\Ran$ is defined as follows: a point $S \to \mathfrak{h} \otimes \Omega^1_D$ consists of $\underline{x} : S \to \Ran$ together with a section of the vector bundle $(\mathfrak{h} \otimes \Omega^1_X)|_{D_{\underline{x}}}$.

Fix a point $x \in X(k)$. We define $(\mathfrak{h} \otimes \Omega^1_D)^{x \ram}_{\underline{x}}$ similarly to $\mathfrak{h} \otimes \Omega^1_D$, but the section is allowed to have a pole along $\{x\} \times S$.

The spaces $\mathfrak{h} \otimes \Omega^1_D$ and $(\mathfrak{h} \otimes \Omega^1_D)^{x \ram}$ admit factorization structures which are co-unital and corr-unital, respectively. The inclusion \[ \mathfrak{h} \otimes \Omega^1_D \longrightarrow (\mathfrak{h} \otimes \Omega^1_D)^{x \ram} \] is a morphism of corr-unital factorization spaces.

Let $(\mathfrak{L}^+H)^{x \ram}$ be the group classifying maps $D_{\underline{x}} \to H$ which are allowed to be meromorphic along $\{ x \} \times S$. It also forms a corr-unital factorization space, and receives a homomorphism \[ \mathfrak{L}^+H \longrightarrow (\mathfrak{L}^+H)^{x \ram} \] compatible with this structure.

The group $(\mathfrak{L}^+H)^{x \ram}$ acts on $(\mathfrak{h} \otimes \Omega^1_D)^{x \ram}$, compatibly with the factorization structure, via the gauge action (cf. \S 1.12 of \cite{R4}). The subgroup $\mathfrak{L}^+H$ preserves the subspace $\mathfrak{h} \otimes \Omega^1_D$. Define \[ \LS_H(D) := (\mathfrak{h} \otimes \Omega^1_D)/\mathfrak{L}^+H \] and \[ \LS_H(D)^{x \ram} := (\mathfrak{h} \otimes \Omega^1_D)^{x \ram}/(\mathfrak{L}^+H)^{x \ram}, \] a co-unital and corr-unital factorization space respectively. We have a canonical morphism of corr-unital factorization spaces \[ \iota : \LS_H(D) \longrightarrow \LS_H(D)^{x \ram}. \]

The co-unital factorization structure on $\LS_H(D)$ makes $\QCoh(\LS_H(D))$ into an object of $\FactCat^{\laxfact}$. We recall from \cite{R3}, Lemma 9.8.1 that we have a canonical equivalence \[ \QCoh(\LS_H(D)) \tilde{\longrightarrow} \Rep(H) \] in $\FactCat^{\laxfact}$, and in particular $\QCoh(\LS_H(D))$ actually belongs to $\FactCat$.

\subsection{Unital structure} We would like to upgrade $\QCoh(\LS_H(D)^{x \ram})$ to a unital factorization category. Given an injection of finite sets $J \to I$, the corr-unital structure on $\LS_H(D)^{x \ram}$ determines a correspondence \[
\begin{tikzcd}
& \LS_H(D)^{x \ram}_{J \to I} \arrow{dl}[swap]{\alpha} \arrow{dr}{\beta} & \\
\LS_H(D)^{x \ram}_{X^J_{\dR}} & & \LS_H(D)^{x \ram}_{X^I_{\dR}}.
\end{tikzcd} \]
The morphism $\beta$ is ind-schematic, but generally not schematic unless $H$ is unipotent. In particular, the usual direct image of quasicoherent sheaves along this map is poorly behaved.

\begin{proposition}

\label{lsunrres}

If $H$ is reductive, then for any injection of finite sets $J \to I$, the functor \[ \beta^* : \QCoh(\LS_H(D)^{x \ram}_{X^I_{\dR}}) \longrightarrow \QCoh(\LS_H(D)^{x \ram}_{J \to I}) \] admits a left adjoint.

\end{proposition}

\begin{proof}

Writing $\LS_H(D)^{x\ram,\log} := (\mathfrak{h} \otimes \Omega^1_D)^{x\ram}/\mathfrak{L}^+H$, we can factorize $\beta$ as \[ \LS_H(D)^{x \ram}_{J \to I} \xrightarrow{\beta_1} \LS_H(D)^{x\ram,\log}_{X^I_{\dR}} \xrightarrow{\beta_2} \LS_H(D)^{x \ram}_{X^I_{\dR}}. \] To see that $\beta_1^*$ admits a left adjoint, it suffices to show the same for inverse image along the inclusion \[ (\mathfrak{h} \otimes \Omega^1_D)^{x\ram}_{J \to I} \longrightarrow (\mathfrak{h} \otimes \Omega^1_D)^{x\ram}_{X^I_{\dR}}. \] For any $n \geq 0$, let \[ (\mathfrak{h} \otimes \Omega^1_D)^{x\ram,\leq n}_{J \to I} \subset (\mathfrak{h} \otimes \Omega^1_D)^{x\ram}_{J \to I} \] be the closed subscheme where the $1$-form is allowed to have poles of order at most $n$ at $x$. Then we have \[ (\mathfrak{h} \otimes \Omega^1_D)^{x\ram}_{J \to I} = \underset{n}{\colim} \ (\mathfrak{h} \otimes \Omega^1_D)^{x\ram,\leq n}_{J \to I}, \] and the transition maps are regular closed embeddings. Moreover, the map \[ (\mathfrak{h} \otimes \Omega^1_D)^{x\ram,\leq n}_{J \to I} \longrightarrow (\mathfrak{h} \otimes \Omega^1_D)^{x\ram,\leq n}_{X^I_{\dR}} \] is a regular closed embedding for any $n \geq 0$, from which the claim follows.

The existence of the left adjoint $(\beta_2^*)^{\LL}$ follows from the fact that $\LS_H(D)^{x \ram}_{X^I_{\dR}}$ is the quotient of $\LS_H(D)^{x\ram,\log}_{X^I_{\dR}}$ by a Hecke groupoid, which is ind-proper because $H$ is reductive.

\end{proof}

As a consequence of the proposition, we can endow the object $\QCoh(\LS_H(D)^{x \ram})$ with a unital structure and hence view it as an object of $\FactCat^{\laxfact}$. Explicitly, the structure morphism attached to an injection $J \to I$ is given by \[ (\beta^*)^{\LL}\alpha^* : \QCoh(\LS_H(D)^{x \ram}_{X^J_{\dR}}) \longrightarrow \QCoh(\LS_H(D)^{x \ram}_{X^I_{\dR}}). \] As explained in \S 9.9 of \cite{R3}, the $1$-affineness of $\LS_H(D)_{X^I_{\dR}}$ implies that $\QCoh(\LS_H(D)^{x \ram})$ factorizes strictly, i.e., belongs to $\FactCat$.

It follows that \[ (\iota^*)^{\LL} : \QCoh(\LS_H(D)) \longrightarrow \QCoh(\LS_H(D)^{x \ram}) \] is a morphism in $\FactCat$. Moreover, the pointwise tensor product on $\QCoh(\LS_H(D)^{x \ram})$ upgrades it to an object of $\ComAlg(\FactCat_{\laxuntl})$. 

\begin{rem}

Note that the morphism \[ \Vect \longrightarrow \QCoh(\LS_H(D)^{x \ram}) \] corresponding to the monoidal unit $\sO_{\LS_H(D)^{x \ram}}$ belongs to $\FactCat_{\laxuntl}$ but not $\FactCat$, i.e., it is only lax unital with respect to factorization. Namely, the map \[ (\iota^*)^{\LL}\sO_{\LS_H(D)} \longrightarrow \sO_{\LS_H(D)^{x \ram}} \] from the factorization unit to the monoidal unit is not an isomorphism.

\end{rem}

\subsection{Local geometric class field theory in the de Rham setting} Define the space $\Gr_T^{\infty \cdot x}$ as follows: a point $S \to \Gr_T^{\infty \cdot x}$ consists of $\underline{x} : S \to \Ran$, a $T$-bundle $\sP_T$ on $X \times S$, a trivialization of $\sP_T$ over $(X \times S) \setminus \Gamma_{\underline{x}}$, and a trivialization of $\sP_T$ over $D_{\{x\} \times S}$. More concisely, a point of $\Gr_T^{\infty \cdot x}$ consists of a point of $\Gr_T$ together with a full level structure at $x$.

In particular, the fiber of $\Gr_T^{\infty \cdot x}$ at $x$ identifies with $(\mathfrak{L}T)_x$, while its restriction to $X \setminus \{ x \}$ identifies with $\Gr_T$. The space $\Gr_T^{\infty \cdot x}$ admits a natural corr-unital factorization structure, which is compatible with the projection \[ p : \Gr_T^{\infty \cdot x} \longrightarrow \Gr_T. \]

The commutative group structure on $\Gr_T^{\infty \cdot x}$ endows $\DD(\Gr_T^{\infty \cdot x})$ with a symmetric monoidal structure under convolution, which makes it an object of $\ComAlg(\FactCat_{\laxuntl})$.

\begin{rem}

The monoidal unit morphism \[ \Vect \longrightarrow \DD(\Gr_T^{\infty} \cdot x) \] takes place in $\FactCat_{\laxuntl}$ but not $\FactCat$, which follows from the fact that the morphism \[ \delta_1 \longrightarrow \omega_{\mathfrak{L}^+T} \] in $\DD(\mathfrak{L}T)_x$ is not an isomorphism.

\end{rem}

\begin{theorem}

\label{lgcft}

There is a canonical equivalence \[ \bL_T : \DD(\Gr_T^{\infty \cdot x}) \tilde{\longrightarrow} \QCoh(\LS_{\check{T}}(D)^{x\ram}) \] in $\ComAlg(\FactCat_{\laxuntl})$ fitting into a commutative square \[
\begin{tikzcd}
\DD(\Gr_T^{\infty \cdot x}) \arrow{r}{\bL_T} \arrow{d}{p_{\dR,*}} & \QCoh(\LS_{\check{T}}(D)^{x\ram}) \arrow{d}{\iota^*} \\
\DD(\Gr_T) \arrow{r}{\sim} & \QCoh(\LS_{\check{T}}(D)).
\end{tikzcd} \]

\end{theorem}

For any $n \geq 0$, let $\Gr_T^{n \cdot x}$ be the quotient of $\Gr_T^{\infty \cdot x}$ classifying a point of $\Gr_T$ together with a structure of level $n$ at $x$ (i.e., we only ask for a trivialization of $\sP_T$ over the $n^{\text{th}}$ order formal neighborhood $D_{\{ x \} \times S}^{(n)}$). To prove the theorem, we will construct a system of isomorphisms \[ \bL_T^{(n)} : \DD(\Gr_T^{n \cdot x}) \tilde{\longrightarrow} \QCoh(\LS_{\check{T}}(D)^{x\ram,\leq n}) \] in $\ComAlg(\FactCat_{\laxuntl})$, and obtain $\bL_T$ by passing to the limit.

\subsection{Duality for quasicoherent sheaves on local systems} Fix a finite set $I$ and $n \geq 0$. Write \[ \pi : \LS_{\check{T}}(D)^{x\ram,\leq n}_{X^I_{\dR}} \longrightarrow X^I_{\dR} \] for the projection. We will now define a $\DD(X^I)$-linear functor \[ \pi_{*,\ren} : \QCoh(\LS_{\check{T}}(D)^{x\ram,\leq n}_{X^I_{\dR}}) \longrightarrow \DD(X^I.) \]

Letting \[ \rho : \LS_{\check{T}}(D)^{x\ram,\log,\leq n}_{X^I_{\dR}} \longrightarrow X^I_{\dR} \] denote the projection, the usual direct image functor \[ \rho_* : \QCoh(\LS_{\check{T}}(D)^{x\ram,\log,\leq n}_{X^I_{\dR}}) \longrightarrow \DD(X^I) \] is continuous and $\DD(X^I)$-linear by Corollary 3.12.2 of \cite{R4} (the proof there is given at $x$ rather than over $X^I_{\dR}$, but adapts to the latter setting). The map \[ \LS_{\check{T}}(D)^{x\ram,\log,\leq n}_{X^I_{\dR}} \longrightarrow \LS_{\check{T}}(D)^{x\ram,\leq n}_{X^I_{\dR}} \] realizes the latter as the quotient of the former by the formal group \[ \sG := ((\mathfrak{L}^+\check{T})^{x\ram,\leq n}/\mathfrak{L}^+T)_{X^I_{\dR}}. \] The $\DD(X^I)$-module category $\IndCoh(\sG)$ admits a natural symmetric monoidal structure with respect to convolution, which is in fact rigid because $\sG$ has finite-dimensional tangent space at the identity. Since $\rho$ is equivariant for the action of $\sG$, the functor $\rho_*$ therefore factors through \[ \QCoh(\LS_{\check{T}}(D)^{x\ram,\log,\leq n}_{X^I_{\dR}}) \longrightarrow \QCoh(\LS_{\check{T}}(D)^{x\ram,\log,\leq n}_{X^I_{\dR}})_{\IndCoh(\sG)} \cong \QCoh(\LS_{\check{T}}(D)^{x\ram,\leq n}_{X^I_{\dR}}), \] which yields the desired functor $\pi_{*,\ren}$.

\begin{lemma}

\label{lsselfdual}

The $\DD(X^I)$-linear pairing \[ \QCoh(\LS_{\check{T}}(D)^{x\ram,\leq n}_{X^I_{\dR}}) \underset{\DD(X^I)}{\otimes} \QCoh(\LS_{\check{T}}(D)^{x\ram,\leq n}_{X^I_{\dR}}) \xrightarrow{\Delta^*} \QCoh(\LS_{\check{T}}(D)^{x\ram,\leq n}_{X^I_{\dR}}) \xrightarrow{\pi_{*,\ren}} \DD(X^I) \] is perfect. This self-duality makes the following square commute: \[
\begin{tikzcd}
\QCoh(\LS_{\check{T}}(D)^{x\ram,\leq n}_{X^I_{\dR}}) \arrow{r}{\sim} \arrow{d}{\pi_{*,\ren}} & \QCoh(\LS_{\check{T}}(D)^{x\ram,\leq n}_{X^I_{\dR}})^{\vee} \arrow{d}{(\pi^*)^{\vee}} \\
\DD(X^I) \arrow{r}{\sim} & \DD(X^I)^{\vee}.
\end{tikzcd} \]
\end{lemma}

\begin{proof}

Since $\rho_*$ is continuous and \[ \Delta : \LS_{\check{T}}(D)^{x\ram,\log,\leq n}_{X^I_{\dR}} \longrightarrow \LS_{\check{T}}(D)^{x\ram,\log,\leq n}_{X^I_{\dR}} \underset{X^I_{\dR}}{\times} \LS_{\check{T}}(D)^{x\ram,\log,\leq n}_{X^I_{\dR}} \] is affine, it follows that \[ \QCoh(\LS_{\check{T}}(D)^{x\ram,\log,\leq n}_{X^I_{\dR}}) \underset{\DD(X^I)}{\otimes} \QCoh(\LS_{\check{T}}(D)^{x\ram,\log,\leq n}_{X^I_{\dR}}) \xrightarrow{\Delta^*} \QCoh(\LS_{\check{T}}(D)^{x\ram,\log,\leq n}_{X^I_{\dR}}) \xrightarrow{\rho_*} \DD(X^I) \] is a perfect pairing. Now the fact that the pairing in lemma is perfect follows from the rigidity of $\IndCoh(\sG)$.

As for the commutativity of the square, note that $\rho_*$ and $\rho^*$ are dual by the construction of the self-duality on \[ \QCoh(\LS_{\check{T}}(D)^{x\ram,\log,\leq n}_{X^I_{\dR}}). \] The claim then follows formally from the construction of $\pi_{*,\ren}$.

\end{proof}

\subsection{Proof of Theorem \ref{lgcft}} For any $n \geq 0$ and any finite set $I$, the Contou-Carr\`ere pairing gives rise to a bimultiplicative line bundle
\begin{equation}
\label{ccpair1}
(\Gr_{T,X^I}^{n \cdot x})_{\dR} \underset{X^I_{\dR}}{\times} \LS_{\check{T}}(D)^{x\ram,\leq n}_{X^I_{\dR}} \longrightarrow \BB \bG_m
\end{equation}
(cf. \S 6.3 of \cite{HR}). Since $\Gr_{T,X^I}^{n \cdot x}$ is an ind-scheme locally of finite type, Verdier duality defines a perfect pairing \[ \DD(\Gr_{T,X^I}^{n \cdot x}) \underset{\DD(X^I)}{\otimes} \DD(\Gr_{T,X^I}^{n \cdot x}) \longrightarrow \DD(X^I). \] Thus the line bundle (\ref{ccpair1}) determines an object of \[ \QCoh((\Gr_{T,X^I}^{n \cdot x})_{\dR} \underset{X^I_{\dR}}{\times} \LS_{\check{T}}(D)^{x\ram,\leq n}_{X^I_{\dR}}) \tilde{\longrightarrow} \Fun_{\DD(X^I)}(\DD(\Gr_{T,X^I}^{n \cdot x}),\QCoh(\LS_{\check{T}}(D)^{x\ram,\leq n}_{X^I_{\dR}})), \] which is the desired $\DD(X^I)$-linear functor \[ \bL_T^{(n)} : \DD(\Gr_{T,X^I}^{n \cdot x}) \longrightarrow \QCoh(\LS_{\check{T}}(D)^{x\ram,\leq n}_{X^I_{\dR}}). \] The bimultiplicativity of (\ref{ccpair1}) implies that $\bL_T^{(n)}$ is symmetric monoidal, and the compatibility of (\ref{ccpair1}) with factorization implies that $\bL_T^{(n)}$ defines a morphism in $\ComAlg(\Fact_{\laxuntl})$ as $I$ varies.

On the other hand, by Lemma \ref{lsselfdual}, the line bundle (\ref{ccpair1}) determines a $\DD(X^I)$-linear functor \[ \QCoh(\LS_{\check{T}}(D)^{x\ram,\leq n}_{X^I_{\dR}}) \longrightarrow \DD(\Gr_{T,X^I}^{n \cdot x}). \] Define $^{\prime}\bL_T^{(n)}$ to be the composition of this functor with $\inv^!$, where \[ \inv : \Gr_{T,X^I}^{n \cdot x} \longrightarrow \Gr_{T,X^I}^{n \cdot x} \] is inversion of the group law. We claim that $^{\prime}\bL_T^{(n)}$ is inverse to $\bL_T^{(n)}$.

Denote by $\sL$ the line bundle (\ref{ccpair1}), viewed as an object of \[ \QCoh((\Gr_{T,X^I}^{n \cdot x})_{\dR} \underset{X^I_{\dR}}{\times} \LS_{\check{T}}(D)^{x\ram,\leq n}_{X^I_{\dR}}). \] Write \[ \upsilon : \Gr_{T,X^I}^{n \cdot x} \longrightarrow X^I \] for the projection.

\begin{lemma}

\label{lgcftlem}

There exist canonical isomorphisms \[ (\upsilon_{\dR,*} \otimes \id)(\sL) \cong \sO_{\LS_{\check{T}}(D)^{x\ram,\leq n}_{X^I_{\dR}}} \quad \text{and} \quad (\id \otimes \pi_{*,\ren})(\sL) \cong \delta_{1,X^I_{\dR}}. \]

\end{lemma}

Before proving the lemma, let us explain how it implies that $\bL_T^{(n)}$ and $^{\prime}\bL_T^{(n)}$ are mutually inverse. Using the self-duality of $\DD(\Gr_{T,X^I}^{n \cdot x})$, we see that the composition \[ ^{\prime}\bL_T^{(n)} \circ \bL_T^{(n)} : \DD(\Gr_{T,X^I}^{n \cdot x}) \longrightarrow \DD(\Gr_{T,X^I}^{n \cdot x}) \] is given by a kernel $\sK$ in \[ \DD(\Gr_{T,X^I}^{n \cdot x} \underset{X^I_{\dR}}{\times} \Gr_{T,X^I}^{n \cdot x}). \] The bimultiplicativity of (\ref{ccpair1}) implies that \[ \sK \cong (\id \otimes \inv^!)\mult^!(\id \otimes \pi_{*,\ren})(\sL), \] where \[ \mult : \Gr_{T,X^I}^{n \cdot x} \underset{X^I}{\times} \Gr_{T,X^I}^{n \cdot x} \longrightarrow \Gr_{T,X^I}^{n \cdot x} \] denotes the multiplication map. The lemma therefore implies that \[ \sK \cong \Delta_{\dR,*}\omega_{\Gr_{T,X^I}^{n \cdot x}} \] is isomorphic to the kernel corresponding to the identity functor, and hence $^{\prime}\bL_T^{(n)} \circ \bL_T^{(n)}$ is isomorphic to the identity functor. A symmetric argumentshows that $\bL_T^{(n)} \circ ^{\prime}\bL_T^{(n)}$ is isomorphic to the identity.

\begin{proof}[Proof of Lemma \ref{lgcftlem}]

We begin by observing that $\bL_T^{(n)}$ and $^{\prime}\bL_T^{(n)}$ are mutually inverse over $X \setminus \{x \}$, where they agree with the standard equivalence \[ \DD(\Gr_T) \tilde{\longrightarrow} \Rep(\check{T}) = \QCoh(\LS_{\check{T}}(D)) \] and its inverse respectively.

It follows that $(\id \otimes \pi_{*,\ren})(\sL)$ identifies with $\delta_1$ away from $x$. The factorization structure on $(\id \otimes \pi_{*,\ren})(\sL)$ therefore upgrades it to an object of \[ \delta_1\modfact(\DD(\Gr_T^{n \cdot x}))_x. \] Since $\delta_1$ is the unit object for the factorization structure on $\DD(\Gr_T)$, restriction along the inclusion\[ \Gr_{T,x}^{n \cdot x} \longrightarrow \Gr_{T,X^I}^{n \cdot x} \] determines an equivalence \[ \delta_1\modfact(\DD(\Gr_T^{n \cdot x}))_x \cong \DD(\Gr_{T,x}^{n \cdot x}). \] To establish the isomorphism \[ (\id \otimes \pi_{*,\ren})(\sL) \cong \delta_{1,X^I_{\dR}}, \] it therefore it suffices to prove that \[ (\id \otimes \pi_{*,\ren})(\sL)|^!_{\Gr_{T,x}^{n \cdot x}} \cong \delta_1. \]

Similarly, using the fact that $\sO_{\LS_{\check{T}}(D)}$ is the unit for the factorization structure on $\QCoh(\LS_{\check{T}}(D))$, a symmetric argument reduces the construction of the other isomorphism to the claim that \[  (\upsilon_{\dR,*} \otimes \id)(\sL)|_{\LS_{\check{T}}(\mathring{D}_x)^{\leq n}} \cong \sO_{\LS_{\check{T}}(\mathring{D}_x)^{\leq n}}. \]

Thus we have reduced to checking that $\bL_T^{(n)}$ and $^{\prime}\bL_T^{(n)}$ are mutually inverse over $x$. This is well-known (cf. \cite{HR} \S 6.3): a choice of coordinate at $x$ allows one to decompose $\Gr_{T,x}^{n \cdot x}$ and $\LS_{\check{T}}(\mathring{D}_x)^{\leq n}$ into direct factors which can be analyzed explicitly.

\end{proof}

\subsection{Proof of Theorem \ref{mainthm} for $G = T$} As explained in \S \ref{pointred}, it suffices to prove that $\Sat_{T,x}$ is an equivalence. Let us reinterpret the monoidal functor \[ \DD(\sH_{T,x}) \xrightarrow{(\ref{presatake})} \sO_{\check{T}}\modfact(\Rep(\check{T} \times \check{T}))_x. \] Being the fiber of the factorization category $\DD(\Gr_T^{\infty \cdot x})$ at $x$, the category $\DD(\mathfrak{L}T)_x$ admits commuting structures of $\DD(\mathfrak{L}T)_x$-module and factorization $\DD(\Gr_T)$-module at $x$. In particular, it defines a functor \[ \DD(\mathfrak{L}T)_x\mod \longrightarrow \DD(\Gr_T)\modfact_x. \]

Similarly, the category $\QCoh(\LS_{\check{T}}(\mathring{D}_x))$ admits commuting structures of $\QCoh(\LS_{\check{T}}(\mathring{D}_x))$-module and factorization $\Rep(\check{T})$-module at $x$, hence it defines a functor \[ \QCoh(\LS_{\check{T}}(\mathring{D}_x))\mod \longrightarrow \Rep(\check{T})\modfact_x. \]

Taking the fiber of $\bL_T$ at $x$, we obtain an equivalence \[ \mathbbm{L}_{T,x} : \DD(\mathfrak{L}T)_x \tilde{\longrightarrow} \QCoh(\LS_{\check{T}}(\mathring{D}_x)) \] of symmetric monoidal categories, which moreover respects the factorization module structures for $\DD(\Gr_T) \cong \Rep(\check{T})$. It therefore gives rise to a commutative square \[
\begin{tikzcd}
\DD(\mathfrak{L}T)_x\mod \arrow{r} \arrow{d}[sloped]{\sim} & \DD(\Gr_T)\modfact_x \arrow{d}[sloped]{\sim} \\
\QCoh(\LS_{\check{T}}(\mathring{D}_x))\mod \arrow{r} & \Rep(\check{T})\modfact_x
\end{tikzcd} \]
in $\DGCat\mod$.

We have a monoidal equivalence \[ \DD(\sH_{T,x}) \tilde{\longrightarrow} \uEnd_{\DD(\mathfrak{L}T)_x\mod}(\DD(\Gr_{T,x})), \] where the right side means relative inner endomorphisms with respect to the action of $\DGCat$ on $\DD(\mathfrak{L}T)_x\mod$. On the other hand, there is a monoidal equivalence \[ \sO_{\check{T}}\modfact(\Rep(\check{T} \times \check{T}))_x \tilde{\longrightarrow} \uEnd_{\Rep(\check{T})\modfact_x}(\Rep(\check{T})) \] by \cite{R3} \S 9.22.

Tracing through the constructions, we see that the previously constructed monoidal functor \[ \DD(\sH_{T,x}) \longrightarrow \sO_{\check{T}}\modfact(\Rep(\check{T} \times \check{T}))_x \] corresponds under the above identifications with the monoidal functor \[ \uEnd_{\DD(\mathfrak{L}T)_x\mod}(\DD(\Gr_{T,x})) \longrightarrow \uEnd_{\Rep(\check{T})\modfact_x}(\Rep(\check{T})) \] induced by the upper-right circuit of the above commutative square. Using the commutativity of the square (and hence Theorem \ref{lgcft}), it follows that this agrees with the monoidal functor induced by the lower-left circuit of the square. But this is an equivalence by Theorem 9.13.1 in \cite{R3} (in fact, the theorem is proved after a series of reductions by establishing precisely this equivalence).

This completes the proof of Theorem \ref{mainthm} for $G = T$.

\section{Factorization modules at a point}\label{s:factmodpt}

In this section, we will give an explicit description of the category of factorization modules \[ \Upsilon(\mathfrak{n}_P,\sO_{G})\modfact(\Rep(G) \otimes \Rep(M))_x \] at a fixed point $x \in X(k)$.

\subsection{Lie-$!$ coalgebras} A \emph{Lie-}$!$ \emph{coalgebra} in $\Rep(H)$ on $X$ is a Lie coalgebra in the symmetric monoidal category $\Rep(H) \otimes \DD(X)$. One can give an alternative definition as follows. The category $\Rep(H)_{\Ran}$ carries a non-unital symmetric monoidal structure denoted by $\otimes^*$ (cf. \cite{R1} \S 7.17). Note that the restriction functor \[ \Delta^! : \Rep(H)_{\Ran} \longrightarrow \Rep(H)_{X_{\dR}} \cong \Rep(H) \otimes \DD(X) \] carries $\otimes^*$ to $\otimes^!$, and in particular lifts to a functor \[ (\Delta^!)^{\enh} : \LieCoalg^{\otimes^*}(\Rep(H)_{\Ran}) \longrightarrow \LieCoalg^{\otimes^!}(\Rep(H) \otimes \DD(X)). \]

\begin{proposition}

The functor $(\Delta^!)^{\enh}$ restricts to an equivalence on the full subcategory of $\LieCoalg^{\otimes^*}(\Rep(H)_{\Ran})$ consisting of objects whose underlying D-module on $\Ran$ is supported on the main diagonal.

\end{proposition}

\begin{proof}

This follows from the observation that the left adjoint \[ \Delta_{\dR,*} : \Rep(H) \otimes \DD(X) \longrightarrow \Rep(H)_{\Ran} \] of $\Delta^!$ is fully faithful and left-lax symmetric monoidal.

\end{proof}

There is another non-unital symmetric monoidal structure $\oplus^*$ on $\Rep(H)_{\Ran}$, which is defined in the same way as $\otimes^*$ but with $\boxtimes$ replaced by $\boxplus$. Note that unlike $\otimes^*$, the operation $\oplus^*$ does not preserve colimits in each variable. The restriction functor \[ \Delta^! : \Rep(H)_{\Ran} \longrightarrow \Rep(H)_{X_{\dR}} \cong \Rep(H) \otimes \DD(X) \] carries $\oplus^*$ to $\oplus$, and hence lifts to a functor \[ (\Delta^!)^{\enh,\oplus} : \ComAlg^{\oplus^*}(\Rep(H)_{\Ran}) \longrightarrow \ComAlg^{\oplus}(\Rep(H) \otimes \DD(X)) \cong \Rep(H) \otimes \DD(X). \]

\begin{proposition}

\label{factplus}

The functor $(\Delta^!)^{\enh,\oplus}$ admits a fully faithful left adjoint \[ \Fact^{\oplus} : \Rep(H) \otimes \DD(X) \longrightarrow \ComAlg^{\oplus^*}(\Rep(H)_{\Ran}). \] The essential image consists of those objects $\sM$ such that the structure map $\sM \oplus^* \sM \to \sM$ induces an equivalence \[ (\sM \boxplus \sM)|_{(\Ran \times \Ran)_{\disj}} \tilde{\longrightarrow} (\add^!\sM)|_{(\Ran \times \Ran)_{\disj}}. \]

\end{proposition}

\begin{proof}

This is Theorem 3.3.3 in \cite{nick-thesis}.

\end{proof}

The category $\Rep(H)_{\Ran}$ is also equipped with the ``pointwise" tensor product $\otimes^!$. Note that for any $\sM$ in $\Rep(H)_{\Ran}$, the functor $(-) \otimes^! \sM$ is lax symmetric monoidal with respect to $\oplus^*$. In particular $\otimes^!$ naturally lifts to a symmetric monoidal structure on $\ComAlg^{\oplus^*}(\Rep(H)_{\Ran})$. The functor $(\Delta^!)^{\enh,\oplus}$ is symmetric monoidal with respect to $\otimes^!$, and hence $\Fact^{\oplus}$ is left-lax symmetric monoidal. It therefore lifts to a functor \[ \LieCoalg^{\otimes^!}(\Rep(H) \otimes \DD(X)) \longrightarrow \LieCoalg^{\otimes^!}(\ComAlg^{\oplus^*}(\Rep(H)_{\Ran})). \]

\subsection{Lie-$!$ comodules} If $L$ is a Lie-$!$ coalgebra in $\Rep(H)$, we have seen that the object $\Delta_{\dR,*}L$ in $\Rep(H)_{\Ran}$ has the structure of Lie coalgebra with respect to $\otimes^*$. Given a map of prestacks $Z \to \Ran$, we obtain a category $\Rep(H)_{\Ran_Z}$ over the marked Ran space $\Ran_Z$. This category is acted on by $\Rep(H)_{\Ran}$ with its convolution product $\otimes^*$.

We can therefore define the category of \emph{Lie-}$!$ \emph{comodules} for $L$ over $Z$ by \[ L\comod^!(\Rep(H))_Z := \Delta_{\dR,*}L\comod^{\otimes^*}(\Rep(H)_{\Ran_Z}) \underset{\Rep(H)_{\Ran_Z}}{\times} \Rep(H)_Z. \] That is, an object of $\Delta_{\dR,*}L\comod^{\otimes^*}(\Rep(H)_{\Ran_Z})$ is a Lie-$!$ comodule if the underlying object of $\Rep(H)_{\Ran_Z}$ is supported on the main diagonal $Z \to \Ran_Z$.

This construction can be promoted to an object $L\comod^!(\Rep(H))$ in $\FactCat^{\laxfact}$, similarly to the case of Lie-$*$ modules.

On the other hand, we have the object $\Fact^{\oplus}(L)$ in $\LieCoalg^{\otimes^!}(\ComAlg^{\oplus^*}(\Rep(H)_{\Ran}))$. For any $Z \to \Ran$, we obtain $\Fact^{\oplus}(L)_Z$, an object of $\LieCoalg^{\otimes^!}(\Rep(H)_Z)$. Thus we can consider the category \[ \Fact^{\oplus}(L)_Z\comod(\Rep(H)_Z) \] over $Z$. These assemble these into an object $\Fact^{\oplus}(L)\comod(\Rep(H))$ of $\FactCat^{\laxfact}$.

\begin{proposition}

\label{liecomod}

There is a canonical equivalence \[ L\comod^!(\Rep(H)) \tilde{\longrightarrow} \Fact^{\oplus}(L)\comod(\Rep(H)) \] in $\FactCat^{\laxfact}$, which commutes with the forgetful functors to $\Rep(H)$.

\end{proposition}

\begin{proof}

For any $Z \to \Ran$, we construct an equivalence \[ L\comod^!(\Rep(H))_Z \tilde{\longrightarrow} \Fact^{\oplus}(L)_Z\comod(\Rep(H)_Z) \] which lifts the identity functor on $\Rep(H)_Z$. The compatibility with factorization structures will be manifest.

Since $\Fact^{\oplus}$ is oplax symmetric monoidal with respect to $\otimes^!$, so is
\begin{align*}
\Rep(H)_{X_{\dR}} &\longrightarrow \Rep(H)_Z \\
\sM &\mapsto \Fact^{\oplus}(\sM)_Z.
\end{align*}
This defines an oplax action of $\Rep(H)_{X_{\dR}}$ on $\Rep(H)_Z$. On the other hand, we have the oplax symmetric monoidal functor \[ \Delta_{\dR,*} : (\Rep(H)_{X_{\dR}},\otimes^!) \longrightarrow (\Rep(H)_{\Ran},\otimes^*). \] The action of $(\Rep(H)_{\Ran},\otimes^*)$ on $\Rep(H)_{\Ran_Z}$ therefore restricts to an oplax action of $\Rep(H)_{X_{\dR}}$. Thus it suffices to show that \[ \Delta_{Z,*} : \Rep(H)_Z \longrightarrow \Rep(H)_{\Ran_Z} \] is a morphism of oplax $\Rep(H)_{X_{\dR}}$-module categories with respect to the above actions. This follows from the observation that the $\sM \mapsto \Fact^{\oplus}(\sM)_Z$ is isomorphic to the composition of 
\begin{align*}
\Rep(H)_{X_{\dR}} &\longrightarrow \Rep(H)_{\Ran_Z} \\
\sM &\mapsto \Delta_{\dR,*}(\sM) \otimes^* \Delta_{Z,*}\sO_Z
\end{align*}
with \[ \Delta_Z^! : \Rep(H)_{\Ran_Z} \longrightarrow \Rep(H)_Z. \]
\end{proof}

\subsection{Koszul duality for Lie coalgebras} Let $\bO$ be a non-unital symmetric monoidal DG category. The functor \[ \triv_{\ComAlg} : \bO \longrightarrow \ComAlg^{\nonuntl}(\bO) \] admits a left adjoint $\coPrim_{\ComAlg}$, and similarly \[ \triv_{\LieCoalg} : \bO \longrightarrow \LieCoalg(\bO) \] admits a right adjoint $\Prim_{\LieCoalg}$. We define \[ \coChev := [1] \circ \coPrim_{\ComAlg} \quad \text{and} \quad \Chev := [-1] \circ \Prim_{\LieCoalg}. \] The Koszul duality between the commutative operad and the Lie co-operad means that these lift to adjoint functors \[ \coChev^{\enh} : \ComAlg^{\nonuntl}(\bO) \rightleftarrows \LieCoalg(\bO) : \Chev^{\enh}. \]

\begin{proposition}

\label{grpliecoalg}

The adjoint functors \[ \BB_{\LieCoalg(\bO)^{\op}} : \Grp(\LieCoalg(\bO)^{\op}) \rightleftarrows \LieCoalg(\bO)^{\op} : \Omega_{\LieCoalg(\bO)^{\op}} \] are equivalences.

\end{proposition}

\begin{proof}

This is proved in much the same way as Proposition 1.6.4 in Chapter 6 of \cite{gr-ii}, the key difference being that the commutativity of the square \[
\begin{tikzcd}
\Grp(\LieCoalg(\bO)^{\op}) \arrow{r}[yshift=0.2em]{\BB_{\LieCoalg(\bO)^{\op}}} \arrow{d}{\oblv} & \LieCoalg(\bO)^{\op} \arrow{d}{\oblv} \\
\bO^{\op} \arrow{r}{[1]_{\bO^{\op}}} & \bO^{\op}.
\end{tikzcd} \]
is less obvious here. To prove this commutativity, we first rewrite \[ \LieCoalg(\bO)^{\op} \cong \LieAlg(\bO^{\op}). \] We remark that $\LieAlg(\bO^{\op})$ does not quite fit into the framework of \emph{loc. cit.}, since the tensor product on $\bO^{\op}$ does not generally preserve colimits, but nonetheless we can consider Lie algebras in this category. Given $L$ in $\Grp(\LieAlg(\bO^{\op}))$, we note that the image of the simplicial object $(\BB_{\LieAlg(\bO^{\op})} L)^{\bullet}$ under the forgetful functor \[ \Grp(\LieAlg(\bO^{\op})) \longrightarrow \Grp(\bO^{\op}) \cong \bO^{\op} \] (the equivalence here is due to the stability of $\bO^{\op}$) is simply the \v{C}ech nerve of $L \to 0$. The tensor product on $\bO^{\op}$ certainly preserves this particular sifted colimit in each variable, since the geometric realization of the \v{C}ech nerve in $\bO^{\op}$ computes the shift $L[1]$. The claim now follows.

\end{proof}

\subsection{Commutative Hopf algebras} The category of non-unital commutative Hopf algebras in $\bO$ is defined by \[ \ComHopfAlg^{\nonuntl}(\bO) := \Grp(\ComAlg^{\nonuntl}(\bO)^{\op})^{\op}. \] By Proposition (\ref{grpliecoalg}) and its proof, the functor \[ (\BB_{\LieCoalg(\bO)^{\op}})^{\op} \circ \Grp(\coChev^{\enh,\op})^{\op} : \ComHopfAlg^{\nonuntl}(\bO) \longrightarrow \LieCoalg(\bO) \] fits into a commutative square \[
\begin{tikzcd}
\ComHopfAlg^{\nonuntl}(\bO) \arrow{r} \arrow{d}{\oblv} & \LieCoalg(\bO) \arrow{d}{\oblv} \\
\ComAlg^{\nonuntl}(\bO) \arrow{r}{\coPrim_{\ComAlg}} & \bO,
\end{tikzcd} \]
and hence we denote it by \[ \coPrim_{\ComHopfAlg}^{\enh} := (\BB_{\LieCoalg(\bO)^{\op}})^{\op} \circ \Grp(\coChev^{\enh,\op})^{\op}. \] Note that $\coPrim_{\ComHopfAlg}^{\enh}$ admits the right adjoint \[ \UU^{\co} : \Grp(\Chev^{\enh,\op})^{\op} \circ (\Omega_{\LieCoalg(\bO)^{\op}})^{\op} : \LieCoalg(\bO) \longrightarrow \ComHopfAlg^{\nonuntl}(\bO). \] This choice of notation is justified by the following observation.

\begin{proposition}

The composition \[ \LieCoalg(\bO) \xrightarrow{\UU^{\co}} \ComHopfAlg^{\nonuntl}(\bO) \xrightarrow{\oblv} \AssocCoalg^{\nonuntl}(\bO) \] is right adjoint to the functor \[ \cores^{\AssocCoalg \to \LieCoalg} : \AssocCoalg^{\nonuntl}(\bO) \longrightarrow \LieCoalg(\bO) \] of corestriction along the morphism of co-operads $\emph{Assoc}^* \to \emph{Lie}^*$.

\end{proposition}

\begin{proof}

The proof is analogous to that of Theorem 6.1.2 in Chapter 6 of \cite{gr-ii}.

\end{proof}

\begin{corollary}

\label{hopfcomod}

For any $A$ in $\ComHopfAlg^{\nonuntl}(\bO)$, there is a canonical morphism \[ \cores^{\AssocCoalg \to \LieCoalg}(A) \longrightarrow \coPrim_{\ComHopfAlg}^{\enh}(A) \] in $\LieCoalg(\bO)$, with the same underlying morphism in $\bO$ as the unit map \[ A \longrightarrow \triv_{\ComAlg}\coPrim_{\ComAlg}(A). \] In particular, the forgetful functor \[ A\comod(\bO) \longrightarrow \bO \] naturally lifts to a functor \[ A\comod(\bO) \longrightarrow \coPrim_{\ComHopfAlg}^{\enh}(A)\comod(\bO). \]

\end{corollary}

\begin{proof}

The morphism in question can be obtained as the composition
\begin{align*}
\cores^{\AssocCoalg \to \LieCoalg}(A) &\longrightarrow \cores^{\AssocCoalg \to \LieCoalg}(\UU^{\co}(\coPrim_{\ComHopfAlg}^{\enh}(A))) \\
&\longrightarrow \coPrim_{\ComHopfAlg}^{\enh}(A).
\end{align*}

\end{proof}

\subsection{Deformation theory} Let $\pi : Z \to Y$ be a morphism of prestacks equipped with a section $\sigma : Y \to Z$ which is affine.

\begin{proposition}

The cotangent complex $T^*(Y/Z)$ relative to $\sigma$ has a natural structure of Lie coalgebra in $\QCoh(Y)$, and $\sigma^*$ naturally lifts to a functor \[ (s^*)^{\enh} : \QCoh(Z) \longrightarrow T^*(Y/Z)\comod(\QCoh(Y)). \]

\label{colie}

\end{proposition}

\begin{proof}

Let $\sG := Y \times_Z Y$, a group in prestacks affine over $Y$. Without loss of generality, we can replace $Z$ by the classifying prestack $\BB_Y \sG$. Writing $e : Y \to \sG$ for the unit section, let $A \in \ComHopfAlg^{\nonuntl}(\QCoh(Y)^{\leq 0})$ be the direct image of the ideal sheaf $\sI_e$ under the projection $\sG \to Y$. Then we have a canonical equivalence \[ A\comod(\QCoh(Y)) \tilde{\longrightarrow} \QCoh(\BB_Y \sG) \] which fits into a commutative triangle \[
\begin{tikzcd}
A\comod(\QCoh(Y)) \arrow{r} \arrow{dr}{\oblv} & \QCoh(\BB_Y \sG) \arrow{d}{s^*} \\
& \QCoh(Y).
\end{tikzcd} \]

It follows from the definitions that we can identify \[ \coPrim_{\ComAlg}(A) \cong e^*T^*(\sG/Y) \cong T^*(Y/B_Y\sG). \] Using the functor $\coPrim_{\ComHopfAlg}^{\enh}$ constructed in the previous section, we obtain a Lie coalgebra structure on $T^*(Y/B_Y\sG)$. The second claim follows from Corollary \ref{hopfcomod}.

\end{proof}

\subsection{Deformations of factorization spaces} Now suppose that $\pi : Z \to Y$ is a morphism of corr-unital commutative factorization spaces, again equipped with an affine section $\sigma : Y \to Z$. We assume that for any injection of finite sets $J \to I$, the structure map \[ Y_{I,J} \longrightarrow Y_{X^I_{\dR}} \] is a regular closed embedding, and likewise for $Z$. This allows us to equip $\QCoh(Y)$ and $\QCoh(Z)$ with unital structures, and hence upgrade them to objects of $\FactCat^{\laxfact}$. Namely, if \[
\begin{tikzcd}
& Y_{I,J} \arrow{dl}[swap]{\alpha} \arrow{dr}{\beta} & \\
Y_{X^J_{\dR}} & & Y_{X^I_{\dR}}
\end{tikzcd} \]
is the correspondence defining the corr-unital structure on $Y$, then the corresponding unital structure on $\QCoh(Y)$ is given by \[ (\beta^*)^{\LL}\alpha^* : \QCoh(Y)_{X^J_{\dR}} \longrightarrow \QCoh(Y)_{X^I_{\dR}}. \]

\begin{lemma}

\label{cotanfact}

There is a canonical isomorphism of Lie coalgebras \[ \Fact^{\oplus}(T^*(Y_{X_{\dR}}/Z_{X_{\dR}})) \tilde{\longrightarrow} T^*(Y_{\Ran}/Z_{\Ran}) \] in $\QCoh(Y_{\Ran})$.

\end{lemma}

\begin{proof}

The commutative factorization structures on $Y$ and $Z$ induce a structure of commutative $\oplus^*$-algebra on $T^*(Y_{\Ran}/Z_{\Ran})$, which moreover lies in the essential image of $\Fact^{\oplus}$ by Proposition \ref{factplus}. It therefore suffices to produce an isomorphism of Lie-$!$ coalgebras \[ T^*(Y_{X_{\dR}}/Z_{X_{\dR}}) \tilde{\longrightarrow} (\Delta^!)^{\enh}T^*(Y_{\Ran}/Z_{\Ran}), \] which is immediate from the definitions.

\end{proof}

\begin{proposition}

The functor $\sigma^*$ lifts to a morphism \[ \QCoh(Z) \longrightarrow T^*(Y_{X_{\dR}}/Z_{X_{\dR}})\comod^!(\QCoh(Y)) \] in $\FactCat^{\laxfact}$.

\end{proposition}

\begin{proof}

This follows from Propositions \ref{colie} and \ref{liecomod} in view of Lemma \ref{cotanfact}.

\end{proof}

\subsection{Deformations of local systems} Let us specialize to the situation \[ \pi : \LS_P(D)^{x \ram} \underset{\LS_M(D)^{x \ram}}{\times} \LS_M(D) \rightleftarrows \LS_M(D) : \sigma. \] We begin with a geometric observation.

\begin{proposition}

\label{affsect}

For any injection of finite sets $J \to I$, the corr-unital structure map \[ \LS_P(D)^{x \ram}_{J \to I} \underset{\LS_M(D)^{x \ram}_{J \to I}}{\times} \LS_M(D)_{X^I_{\dR}} \longrightarrow \LS_P(D)^{x \ram}_{X^I_{\dR}} \underset{\LS_M(D)^{x \ram}_{X^I_{\dR}}}{\times} \LS_M(D)_{X^I_{\dR}} \] is a regular closed embedding. Moreover, the section \[ \sigma: \LS_M(D)_{X^I_{\dR}} \longrightarrow \LS_P(D)^{x \ram}_{X^I_{\dR}} \underset{\LS_M(D)^{x \ram}_{X^I_{\dR}}}{\times} \LS_M(D)_{X^I_{\dR}} \] is affine.

\end{proposition}

\begin{proof}

Since $N_P$ is unipotent, the canonical map \[ ((\mathfrak{p} \otimes \Omega_D^1)^{x \ram,\leq 1}_{J \to I} \underset{(\mathfrak{m} \otimes \Omega_D^1)^{x \ram,\leq 1}_{J \to I}}{\times} (\mathfrak{m} \otimes \Omega_D^1)_{X^I_{\dR}})/(\mathfrak{L}^+P)_{X^I_{\dR}} \longrightarrow \LS_P(D)^{x \ram}_{J \to I} \underset{\LS_M(D)^{x \ram}_{J \to I}}{\times} \LS_M(D)_{X^I_{\dR}} \] is an isomorphism. Both assertions follow readily from this presentation.

\end{proof}

The proposition allows us to endow \[ \QCoh(\LS_P(D)^{x \ram} \underset{\LS_M(D)^{x \ram}}{\times} \LS_M(D)) \] with a unital structure, and hence view it as an object of $\FactCat^{\laxfact}$. As shown in \S 9.9 of \cite{R3}, the $1$-affineness of $\LS_P(D)_{X^I_{\dR}}$ implies that the factorization structure is strict, i.e., this object belongs to $\FactCat$.

The object $j_!j^!(\mathfrak{n}_P^* \otimes \omega_X)$ inherits a structure of Lie-$!$ coalgebra in $\Rep(M)$ from $\mathfrak{n}_P^* \otimes \omega_X$. It is Verdier dual to the Lie-$*$ algebra $j_*j^*(\mathfrak{n}_P \otimes k_X)$ which appears in \S 8 of \cite{R3}.

\begin{proposition}

The relative cotangent complex of \[ \sigma : \LS_M(D)_{X_{\dR}} \longrightarrow \LS_P(D)^{x \ram}_{X_{\dR}} \underset{\LS_M(D)^{x \ram}_{X_{\dR}}}{\times} \LS_M(D)_{X_{\dR}} \] is canonically isomorphic to $j_!j^!(\mathfrak{n}_P^* \otimes \omega_X)$ as a Lie-$!$ coalgebra in $\Rep(M)$.

\end{proposition}

\begin{proof}

Let $U := X \setminus \{ x \}$. Since \[ \LS_H(D)^{x \ram}_{X_{\dR}} \underset{X_{\dR}}{\times} U_{\dR} = \LS_H(D)_{U_{\dR}} = \bB H \times U_{\dR}, \] we have \[ j^!T^*(\sigma) \cong T^*(\LS_M(D)_{U_{\dR}}/\LS_P(D)_{U_{\dR}}) \cong \mathfrak{n}_P^* \otimes \omega_U \cong j^!(\mathfrak{n}_P^* \otimes \omega_X), \] from which we obtain a morphism of Lie-$!$ coalgebras \[ j_!j^!(\mathfrak{n}_P^* \otimes \omega_X) \longrightarrow T^*(\sigma). \] It is an isomorphism over $U$ by construction, so it suffices to prove that \[ i^!j_!j^!(\mathfrak{n}_P^* \otimes \omega_X) \longrightarrow i^!T^*(\sigma) \] is an isomorphism, where $i : \{ x \} \to X$ is the inclusion.

Observe that \[ i^!T^*(\sigma) \cong T^*(\LS_M(D_x)/\LS_P(\mathring{D}_x) \underset{\LS_M(\mathring{D}_x)}{\times} \LS_M(D_x)) \cong \mathfrak{n}_P^* \otimes \CC^{\dR}(\mathring{D}_x), \] where \[ \CC^{\dR}(\mathring{D}_x) = \cofib(\sO_{\mathring{D}_x} \xrightarrow{d} \Omega^1_{\mathring{D}_x}) \] is the complex of de Rham chains on the punctured disk. By construction, the morphism in question is given by tensoring the canonical isomorphism \[ i^!j_!\omega_U \tilde{\longrightarrow} \CC^{\dR}(\mathring{D}_x) \] with $\mathfrak{n}_P^*$.

\end{proof}

\begin{corollary}

The functor $\sigma^*$ lifts to a morphism \[ \QCoh(\LS_P(D)^{x \ram} \underset{\LS_M(D)^{x \ram}}{\times} \LS_M(D)) \longrightarrow j_*j^*(\mathfrak{n}_P \otimes k_X)\mod^*(\Rep(M)) \] in $\FactCat$.

\end{corollary}

\begin{proof}

Observe that \[ \QCoh(\LS_H(D)_{X_{\dR}}) = \QCoh(\bB H \times X_{\dR}) = \Rep(H) \otimes \DD(X) \] is canonically self-dual via the self-duality on $\Rep(H)$ and $\DD(X)$. The resulting equivalence \[ \Rep(H)_{X_{\dR}}^{\cc,\op} \tilde{\longrightarrow} \Rep(H)_{X_{\dR}}^{\cc} \] lifts to a contravariant equivalence $L \mapsto \bD(L)$ between Lie-$!$ coalgebras in $\Rep(H)$ whose underlying object of $\Rep(H)_{X_{\dR}}$ is compact and Lie-$*$ algebras in $\Rep(H)$ satisfying the same condition. If $L$ is such a Lie-$*$ algebra, then there is a tautological isomorphism \[ L\mod^*(\Rep(H)) \tilde{\longrightarrow} \bD(L)\comod^!(\Rep(H)) \] in $\FactCat^{\laxfact}$.

The fact that $j_*j^*(\mathfrak{n}_P \otimes k_X)\mod^*(\Rep(M))$ belongs to $\FactCat$ (as opposed to merely $\FactCat^{\laxfact}$) follows from the existence of the monadic adjunction \[ \ind : \mathfrak{n}_P\mod(\Rep(M)) \rightleftarrows j_*j^*(\mathfrak{n}_P \otimes k_X)\mod^*(\Rep(M)) : \oblv \] in $\FactCat^{\laxfact}_{\laxuntl}$ and the strictness of the factorization structure on $\mathfrak{n}_P\mod(\Rep(M))$.

\end{proof}

\subsection{Factorization $\Rep(P)$-module structure} Consider the canonical morphism of factorization spaces \[ \iota : \LS_P(D) \longrightarrow \LS_P(D)^{x \ram} \underset{\LS_M(D)^{x \ram}}{\times} \LS_M(D). \] Proposition \ref{affsect} implies that $\iota^*$ admits a left adjoint $(\iota^*)^{\LL}$.

\begin{proposition}

\label{liestarmodprop1}

The functor $(\iota^*)^{\LL}$ fits into a commutative square \[
\begin{tikzcd}
\Rep(P) \arrow{d}{(\iota^*)^{\LL}} \arrow{r}{\res^{N_P}_{\mathfrak{n}_P}} & \mathfrak{n}_P\mod(\Rep(M)) \arrow{d} \\
\QCoh(\LS_P(D)^{x \ram} \underset{\LS_M(D)^{x \ram}}{\times} \LS_M(D)) \arrow{r}{(s^*)^{\enh}} & j_*j^*(\mathfrak{n}_P \otimes k_X)\mod^*(\Rep(M))
\end{tikzcd} \]
in $\FactCat$, where the right vertical functor is given by induction along the morphism of Lie-$*$ algebras $\mathfrak{n}_P \otimes k_X \to j_*j^*(\mathfrak{n}_P \otimes k_X)$.

\end{proposition}

\begin{proof}

The square in question tautologically commutes if the vertical functors are replaced with their right adjoints. Thus we must show that the resulting natural transformation \[ \ind_{\mathfrak{n}_P \otimes k_X}^{j_*j^*(\mathfrak{n}_P \otimes k_X)} \res_{\mathfrak{n}_P}^{N_P} \longrightarrow (s^*)^{\enh}(\iota^*)^{\LL} \] is an isomorphism.

For any finite set $I$, both composite functors are $\Rep(M)_{X^I_{\dR}}$-linear, and since $\Rep(P)_{X^I_{\dR}}$ is generated as a $\Rep(M)_{X^I_{\dR}}$-module by $\triv_P(\omega_{X^I})$, it suffices to check that this natural transformation is an isomorphism after evaluating on this object. Applying the conservative functor \[ \oblv : j_*j^*(\mathfrak{n}_P \otimes k_X)\mod^*(\Rep(M))_{X^I_{\dR}} \longrightarrow \Rep(M)_{X^I_{\dR}}, \] we have reduced to proving that \[ \ind_{\mathfrak{n}_P \otimes k_X}^{j_*j^*(\mathfrak{n}_P \otimes k_X)} \triv_{\mathfrak{n}_P}(\omega_{X^I}) \longrightarrow s^*(\iota^*)^{\LL} \triv_P(\omega_{X^I}) \] is an isomorphism in $\Rep(M)_{X^I_{\dR}}$. Using factorization, we we can assume that $I$ is a singleton. Over $U = X \setminus \{ x \}$, then this morphism is simply the identity map on $\triv_M(\omega_U)$, so it suffices to check that the induced map on the fiber at $x$ is an isomorphism.

In particular, we have reduced to proving the commutativity of the square in question at the point $x$. Then we are in the realm of finite-dimensional deformation theory, and it can be identified with the commutative square \[
\begin{tikzcd}
\QCoh(\LS_P(D_x)) \arrow{d}{(\iota^*)^{\LL}} \arrow{r} & \QCoh(\LS_P(D_x)^{\swedge}_{\LS_M(D_x)}) \arrow{d}{(\widehat{\iota}^*)^{\LL}} \\
\QCoh(\LS_P(\mathring{D}_x) \underset{\LS_M(\mathring{D}_x)}{\times} \LS_M(D_x)) \arrow{r} & \QCoh((\LS_P(\mathring{D}_x) \underset{\LS_M(\mathring{D}_x)}{\times} \LS_M(D_x))^{\swedge}_{\LS_M(D_x)}),
\end{tikzcd} \]
where the horizontal morphisms are given by inverse image and \[ \widehat{\iota} : \LS_P(D_x)^{\swedge}_{\LS_M(D_x)} \longrightarrow (\LS_P(\mathring{D}_x) \underset{\LS_M(\mathring{D}_x)}{\times} \LS_M(D_x))^{\swedge}_{\LS_M(D_x)} \] is obtained from $\iota$ by formal completion along $\LS_M(D_x)$. All of the stacks appearing here are formally smooth, and for such stacks we have the equivalence \[ \Upsilon_Y : \QCoh(Y) \tilde{\longrightarrow} \IndCoh(Y), \] which intertwines inverse image for $\QCoh$ with $!$-inverse image for $\IndCoh$. Since the maps $\iota$ and $\widehat{\iota}$ are proper, the above square can be identified with \[
\begin{tikzcd}
\IndCoh(\LS_P(D_x)) \arrow{d}{\iota_*} \arrow{r} & \IndCoh(\LS_P(D_x)^{\swedge}_{\LS_M(D_x)}) \arrow{d}{\widehat{\iota}_*} \\
\IndCoh(\LS_P(\mathring{D}_x) \underset{\LS_M(\mathring{D}_x)}{\times} \LS_M(D_x)) \arrow{r} & \IndCoh((\LS_P(\mathring{D}_x) \underset{\LS_M(\mathring{D}_x)}{\times} \LS_M(D_x))^{\swedge}_{\LS_M(D_x)}),
\end{tikzcd} \]
which commutes by base change for the standard functors on $\IndCoh$.

\end{proof}

In particular, it follows that $(s^*)^{\enh}$ induces a morphism
\begin{equation}
\label{liestarmodpt}
\QCoh(\LS_P(\mathring{D}_x) \underset{\LS_M(\mathring{D}_x)}{\times} \LS_M(D_x)) \longrightarrow j_*j^*(\mathfrak{n}_P \otimes k_X)\mod^*(\Rep(M))_x
\end{equation} 
of factorization $\Rep(P)$-module categories at $x$.

\begin{proposition}

\label{liestarmodprop2}

The morphism (\ref{liestarmodpt}) restricts to an equivalence on the full subcategory \[ \QCoh((\LS_P(\mathring{D}_x) \underset{\LS_M(\mathring{D}_x)}{\times} \LS_M(D_x))^{\swedge}_{\LS_M(D_x)}), \] embedded via the left adjoint of the restriction functor \[ \QCoh(\LS_P(\mathring{D}_x) \underset{\LS_M(\mathring{D}_x)}{\times} \LS_M(D_x)) \longrightarrow \QCoh((\LS_P(\mathring{D}_x) \underset{\LS_M(\mathring{D}_x)}{\times} \LS_M(D_x))^{\swedge}_{\LS_M(D_x)}). \]

\end{proposition}

\begin{proof}

Identify (\ref{liestarmodpt}) with the composition
\begin{align*}
\QCoh(\LS_P(\mathring{D}_x) \underset{\LS_M(\mathring{D}_x)}{\times} \LS_M(D_x)) &\stackrel{\Upsilon}{\tilde{\longrightarrow}} \IndCoh(\LS_P(\mathring{D}_x) \underset{\LS_M(\mathring{D}_x)}{\times} \LS_M(D_x)) \\
&\longrightarrow \IndCoh((\LS_P(\mathring{D}_x) \underset{\LS_M(\mathring{D}_x)}{\times} \LS_M(D_x))^{\swedge}_{\LS_M(D_x)})
\end{align*}
as in the proof of Proposition \ref{liestarmodprop1}.

\end{proof}

\subsection{Factorization modules for $\Upsilon(\mathfrak{n}_P)$} The lax factorization category $\QCoh(\LS_H(D)^{x \ram})$ has fiber $\QCoh(\LS_H(\mathring{D}_x))$ at $x$, and therefore induces on $\QCoh(\LS_H(\mathring{D}_x))$ the structure of factorization $\Rep(H)$-module at $x$. The action of $\QCoh(\LS_{\Gamma}(\mathring{D}_x))$ on itself by tensor product of quasicoherent sheaves commutes with the factorization $\Rep(H)$-module structure, yielding a functor
\begin{align*}
\Phi_H : \QCoh(\LS_H(\mathring{D}_x))\mod &\longrightarrow \Rep(H)\modfact_x
\end{align*}
which commutes with the forgetful functors to $\DGCat$. On the level of objects, this means that there is a natural factorization $\Rep(H)$-module structure on any $\sC$ in $\QCoh(\LS_H(\mathring{D}_x))\mod$ coming from the factorization module structure on the second factor of \[ \sC \otimes_{\QCoh(\LS_H(\mathring{D}_x))} \QCoh(\LS_H(\mathring{D}_x)) = \sC. \]

The restriction functor \[ \QCoh(\LS_H(\mathring{D}_x)) \longrightarrow \QCoh(\LS_H(\mathring{D}_x)^{\swedge}_{\LS_H(D_x)}) \] admits a fully faithful $\QCoh(\LS_H(\mathring{D}_x))$-linear left adjoint. It follows that the restriction of scalars functor \[ \QCoh(\LS_H(\mathring{D}_x)^{\swedge}_{\LS_H(D_x)})\mod \longrightarrow \QCoh(\LS_H(\mathring{D}_x))\mod \] is fully faithful.

\begin{theorem}

The restriction of $\Phi_H$ to the full subcategory $\QCoh(\LS_H(\mathring{D}_x)^{\swedge}_{\LS_H(D_x)})\mod$ is fully faithful.

\label{factmodff}

\end{theorem}

\begin{proof}

This is Theorem 9.13.1 of \cite{R3}.

\end{proof}

\begin{proposition}

\label{specwhitmonad}

There is a canonical isomorphism \[ \Phi_P(\QCoh((\LS_P(\mathring{D}_x) \underset{\LS_M(\mathring{D}_x)}{\times} \LS_M(D_x))^{\swedge}_{\LS_M(D_x)})) \tilde{\longrightarrow} \Upsilon(\mathfrak{n}_P)\modfact(\Rep(M))_x \] in $\Rep(P)\modfact_x$, where $\QCoh(\LS_P(\mathring{D}_x)^{\swedge}_{\LS_P(D_x)})$ acts on \[ \QCoh((\LS_P(\mathring{D}_x) \underset{\LS_M(\mathring{D}_x)}{\times} \LS_M(D_x))^{\swedge}_{\LS_M(D_x)}) \] via restriction along the composition
\begin{align*}
(\LS_P(\mathring{D}_x) \underset{\LS_M(\mathring{D}_x)}{\times} \LS_M(D_x))^{\swedge}_{\LS_M(D_x)} &\longrightarrow (\LS_P(\mathring{D}_x) \underset{\LS_M(\mathring{D}_x)}{\times} \LS_M(D_x))^{\swedge}_{\LS_P(D_x)} \\
&\longrightarrow \LS_P(\mathring{D}_x)^{\swedge}_{\LS_P(D_x)},
\end{align*}
and the factorization $\Rep(P)$-module structure on $\Upsilon(\mathfrak{n}_P)\modfact(\Rep(M))_x$ arises from the morphism \[ \Rep(P) \longrightarrow \mathfrak{n}_P\mod(\Rep(M)) \xrightarrow{\ind^{* \to \fact}} \Upsilon(\mathfrak{n}_P)\modfact(\Rep(M))_x \] in $\FactCat$.

\end{proposition}

\begin{proof}

Lemma 8.16.1 of \cite{R3} says that \[ \ind^{* \to \ch} : j_*j^*(\mathfrak{n}_P \otimes k_X)\mod^*(\Rep(M))_x \tilde{\longrightarrow} \UU^{\ch}(j_*j^*(\mathfrak{n}_P \otimes k_X))\mod^{\ch}(\Rep(M))_x \] is an equivalence. Identifying \[ j^*\UU^{\ch}(j_*j^*(\mathfrak{n}_P \otimes k_X)) \tilde{\longrightarrow} j^*\UU^{\ch}(\mathfrak{n}_P \otimes k_X) \] as chiral algebras on $X \setminus \{ x \}$, we obtain an equivalence \[ j_*j^*(\mathfrak{n}_P \otimes k_X)\mod^*(\Rep(M))_x \tilde{\longrightarrow} \Upsilon(\mathfrak{n}_P)\mod^{\fact}(\Rep(M))_x. \] By construction, this equivalence respects the factorization $\Rep(P)$-module structures.

We can apply $\Phi_P$ to the $\QCoh(\LS_P(\mathring{D}_x))$-linear functor \[ (\iota_x^*)^{\LL} : \Rep(P) = \QCoh(\LS_P(D_x)) \longrightarrow \QCoh((\LS_P(\mathring{D}_x) \underset{\LS_M(\mathring{D}_x)}{\times} \LS_M(D_x))^{\swedge}_{\LS_M(D_x)}) \] to obtain a  morphism in $\Rep(P)\mod^{\fact}_x$. Inspecting the constructions, we see that taking the fiber at $x$ of the morphism $(\iota^*)^{\LL}$ in $\FactCat^{\laxfact}$ from Proposition \ref{liestarmodprop1} yields the same morphism. Now apply Proposition \ref{liestarmodprop2}.

\end{proof}

\subsection{Factorization $\Upsilon(\mathfrak{n},\sO_{G})$-modules at a point} Finally, we are ready to prove the main theorem of this section.

\begin{theorem}

\label{factmodpoint}

There is a canonical equivalence \[ \Upsilon(\mathfrak{n}_P,\sO_G)\modfact(\Rep(G \times M))_x \tilde{\longrightarrow} \QCoh((\LS_P(\mathring{D}_x) \underset{\LS_{G \times M}(\mathring{D}_x)}{\times} \LS_{G \times M}(D_x))^{\swedge}_{\LS_M(D_x)}). \] 

\end{theorem}

\begin{proof}

As explained in \cite{R3} \S\S 9.22-24, for any factorization $\Rep(G)$-module category $\sC$ at $x$, we can identify \[ \sO_G\modfact(\Rep(G) \otimes \sC)_x \tilde{\longrightarrow} \Fun_{\Rep(G)\modfact_x}(\Rep(G),\sC). \] Here we view the diagonal bimodule $\sO_G$ as a factorization algebra in $\Rep(G \times G)$, and $\Rep(G) \otimes \sC$ as a factorization $\Rep(G \times G)$-module at $x$. Moreover, if $\sC$ lies in the essential image of \[ \QCoh(\LS_G(\mathring{D}_x)^{\swedge}_{\LS_G(D_x)})\mod \] under $\Phi_G$, it follows that
\begin{align*}
\Rep(G) \underset{\QCoh(\LS_G(\mathring{D}_x)^{\swedge}_{\LS_G(D_x)})}{\otimes} \sC &\tilde{\longrightarrow} \Fun_{\QCoh(\LS_G(\mathring{D}_x)^{\swedge}_{\LS_G(D_x)})}(\Rep(G),\sC) \\
&\tilde{\longrightarrow} \Fun_{\Rep(G)\modfact_x}(\Rep(G),\sC) \\
&\tilde{\longrightarrow} \sO_G\modfact(\Rep(G) \otimes \sC)_x,
\end{align*}
where we used Theorem \ref{factmodff} for the second equivalence. The first equivalence uses the observation that $\Rep(G)$ is dualizable and self-dual as an object of $\QCoh(\LS_G(\mathring{D}_x)^{\swedge}_{\LS_G(D_x)})\mod$, cf. \emph{loc. cit.} \S 9.18.

We now apply the above in the case \[ \sC = \QCoh((\LS_P(\mathring{D}_x) \underset{\LS_M(\mathring{D}_x)}{\times} \LS_M(D_x))^{\swedge}_{\LS_M(D_x)}), \] viewed as a $\QCoh(\LS_G(\mathring{D}_x)^{\swedge}_{\LS_G(D_x)})$-module via inverse image along the projection \[ \LS_P(\mathring{D}_x)^{\swedge}_{\LS_P(D_x)} \longrightarrow \LS_G(\mathring{D}_x)^{\swedge}_{\LS_G(D_x)}. \] Note that by Lemma \ref{factmodtrans}, we have an equivalence \[ \sO_G\modfact(\Rep(G) \otimes \Upsilon(\mathfrak{n}_P)\modfact(\Rep(M)))_x \tilde{\longrightarrow} \Upsilon(\mathfrak{n}_P,\sO_G)\modfact(\Rep(G \times M))_x. \] Applying Proposition \ref{specwhitmonad} and the equivalence in the previous paragraph identifies the right hand side with \[ \Rep(G) \underset{\QCoh(\LS_G(\mathring{D}_x)^{\swedge}_{\LS_G(D_x)})}{\otimes} \QCoh((\LS_P(\mathring{D}_x) \underset{\LS_M(\mathring{D}_x)}{\times} \LS_M(D_x))^{\swedge}_{\LS_M(D_x)}). \] The morphism \[ \LS_G(D_x) \longrightarrow \LS_G(\mathring{D}_x)^{\swedge}_{\LS_G(D_x)}, \] being isomorphic to $\{ 0 \}/G \to \widehat{\mathfrak{g}}_0/G$, is affine, which implies that the above tensor product identifies with \[ \QCoh(\LS_G(D_x) \underset{\LS_G(\mathring{D}_x)^{\swedge}_{\LS_G(D_x)}}{\times} (\LS_P(\mathring{D}_x) \underset{\LS_M(\mathring{D}_x)}{\times} \LS_M(D_x))^{\swedge}_{\LS_M(D_x)}). \] Finally, using the identification \[ \LS_H(\mathring{D}_x)^{\swedge}_{\LS_H(D_x)} \tilde{\longrightarrow} \widehat{\mathfrak{h}}_0/H, \] it is easy to see that the canonical map from the above fiber product to \[ \LS_G(D_x) \underset{\LS_G(\mathring{D}_x)}{\times} \LS_P(\mathring{D}_x) \underset{\LS_M(\mathring{D}_x)}{\times} \LS_M(D_x) = \LS_P(\mathring{D}_x) \underset{\LS_{G \times M}(\mathring{D}_x)}{\times} \LS_{G \times M}(D_x) \] factors through an isomorphism with the formal completion of the latter along $\LS_M(D_x)$.

\end{proof}

\begin{corollary}

\label{specpsmonad}

The monad on $\Rep(P)$ induced by the adjunction \[ \Rep(P) \rightleftarrows \Sph^{\spec}_{G,P,x}, \] where the left adjoint is (\ref{psgenren}), is isomorphic to the monad given by the associative algebra $\Sym((\mathfrak{g}/\mathfrak{n}_P)[-2])$.

\end{corollary}

\begin{proof}

Observe that the equivalence of Theorem \ref{factmodpoint} restricts to an equivalence \[ \Upsilon(\mathfrak{n}_P,\sO_G)\mod_0^{\fact}(\Rep(G \times M))_x \tilde{\longrightarrow} \QCoh((\LS_P(\mathring{D}_x) \underset{\LS_{G \times M}(\mathring{D}_x)}{\times} \LS_{G \times M}(D_x))^{\swedge}_{\LS_P(D_x)}). \] We claim that this equivalence is t-exact with respect to the usual t-structure on quasicoherent sheaves: recall that \[ (\LS_P(\mathring{D}_x) \underset{\LS_{G \times M}(\mathring{D}_x)}{\times} \LS_{G \times M}(D_x))^{\swedge}_{\LS_P(D_x)} \cong (\mathfrak{n}_P \underset{\mathfrak{g}}{\times} \{ 0 \})/P, \] and in particular this is an algebraic stack of finite type. The t-structure on \[ \QCoh((\mathfrak{n}_P \underset{\mathfrak{g}}{\times} \{ 0 \})/P) \] is uniquely characterized by the fact that direct image along \[ \{ 0 \}/P \longrightarrow (\mathfrak{n}_P \underset{\mathfrak{g}}{\times} \{ 0 \})/P \] is t-exact, whence the claim.

\begin{comment}

By Theorem 8.13.1 of \cite{R3}, we have \[ \sO_G\modfact(\Rep(G \times P))_x \cong \QCoh(\LS_G(D_x) \underset{\LS_G(\mathring{D}_x)}{\times} \LS_P(D_x)), \] and under the equivalence of the theorem, the left adjoint in (\ref{psconngen}) corresponds to direct image along the morphism \[ \LS_P(D_x) \underset{\LS_G(\mathring{D}_x)}{\times} \LS_G(D_x) \longrightarrow ((\LS_P(\mathring{D}_x) \underset{\LS_{G \times M}(\mathring{D}_x)}{\times} \LS_{G \times M}(D_x))^{\swedge}_{\LS_P(D_x)}. \] Choosing a formal coordinate at $x$, this identifies with the morphism \[ (\{ 0 \} \underset{\mathfrak{g}}{\times} \{ 0 \})/P \longrightarrow (\mathfrak{n}_P \underset{\mathfrak{g}}{\times} \{ 0 \})/P. \] In particular, we see that this map is a closed embedding, and moreover is an eventually coconnective nil-isomorphism. It follows that direct image of quasicoherent sheaves along this map is t-exact, and admits a continuous and conservative right adjoint. The t-exactness follows.

\end{comment}

It follows that the above equivalence preserves coherence, and hence renormalizes to \[ \Sph_{G,P,x}^{\spec} \tilde{\longrightarrow} \IndCoh((\LS_P(\mathring{D}_x) \underset{\LS_{G \times M}(\mathring{D}_x)}{\times} \LS_{G \times M}(D_x))^{\swedge}_{\LS_P(D_x)}). \] The functor (\ref{psgenren}) corresponds to direct image of ind-coherent sheaves along \[ \iota : \LS_P(D_x) \longrightarrow (\LS_P(\mathring{D}_x) \underset{\LS_{G \times M}(\mathring{D}_x)}{\times} \LS_{G \times M}(D_x))^{\swedge}_{\LS_P(D_x)}, \] where we identify \[ \IndCoh(\LS_P(D_x)) \cong \QCoh(\LS_P(D_x)) \cong \Rep(P) \] using the t-exact functor $\Psi$. After a choice of formal coordinate at $x$, this becomes the map \[ \{ 0 \}/P \longrightarrow (\mathfrak{n}_P \underset{\mathfrak{g}}{\times} \{ 0 \})/P. \] If we identify \[ \Gamma(\mathfrak{n}_P \underset{\mathfrak{g}}{\times} \{ 0 \},\sO) = \Sym((\mathfrak{g}/\mathfrak{n}_P)^*[1]), \] then we have \[ \iota^!\iota_*V = \Hom_{\Sym((\mathfrak{g}/\mathfrak{n}_P)^*[1])}(k,V) = \Sym((\mathfrak{g}/\mathfrak{n}_P)[-2]) \otimes V \] for any $V$ in $\Rep(P)$, i.e., the monad $\iota^!\iota_*$ is given by the algebra $\Sym((\mathfrak{g}/\mathfrak{n}_P)[-2])$ Koszul dual to $\Sym((\mathfrak{g}/\mathfrak{n}_P)^*[1])$, as desired.

\end{proof}

\section{The equivalence for $P = G$} \label{s:mainthmproof}

In this section, we complete the proof of the factorizable derived Satake equivalence, as formulated in Theorem \ref{mainthm}.

\subsection{The derived Satake transform as a morphism of monads} By the construction of the morphism (\ref{presatake}), we have a commutative triangle \[
\begin{tikzcd}
& \Rep(\check{G}) \arrow{dl}[above,xshift=-0.3em]{\Sat_G^{\nv}} \arrow{dr}{\oblv^{\comfact}} & \\
\DD(\sH_G) \arrow{rr}{(\ref{presatake})} & & \sO_{\check{G}}\modfact(\Rep(\check{G} \times \check{G}))
\end{tikzcd} \]
in $\FactCat^{\laxfact}$. After renormalization, we obtain a commutative triangle\[
\begin{tikzcd}
\label{monadmaptri}
& \Rep(\check{G}) \arrow{dl} \arrow{dr} & \\
\Sph_G \arrow{rr}{\Sat_G} & & \Sph_{\check{G}}^{\spec}
\end{tikzcd} \]
in $\FactCat$. By Lemma \ref{naivesatula} and Theorem \ref{vacacthm} respectively, the diagonal functors preserve ULA objects over each $X^I$, hence admit factorizable right adjoints. Moreover, their right adjoints are conservative and hence monadic. The above triangle therefore corresponds to a morphism
\begin{equation}
\label{monadmap}
\Psi_G \longrightarrow \Psi_{\check{G}}^{\spec}
\end{equation}
between the corresponding monads on $\Rep(\check{G})$ in $\FactCat_{\laxuntl}$, which we will prove is an isomorphism. As explained in \S \ref{pointred}, Theorem \ref{mainthm} will follow if we prove that the morphism of monads \[ \Psi_{G,x} \longrightarrow \Psi_{\check{G},x}^{\spec} \] is an isomorphism over each $x \in X(k)$.

\subsection{Explicit description of the monads at a point} We will need the following lemma.

\begin{lemma}

\label{innerhomlem}

Let $H$ be an affine group scheme and suppose that $\Rep(H)$ acts on a DG category $\sC$. For any objects $c,d$ in $\sC$, we have a canonical isomorphism \[ \uHom_{\Rep(H)}(c,d) \tilde{\longrightarrow} \Hom_{\sC}(c,\sO_H \star d) \] in $\Rep(H)$, where the $H$-action on the right side arises by functoriality from the $H$-action on $\sO_H$ by right translations.

\end{lemma}

\begin{comment}

\begin{proof}

The left side is the value of the composite functor \[ \Rep(H) \xrightarrow{(-) \star d} \sC \xrightarrow{\uHom_{\sC}(c,-)} \Rep(H) \] on the trivial representation $k$. Clearly the first functor identifies with the composition \[ \Rep(H) = \Rep(H) \otimes \Vect \xrightarrow{\id_{\Rep(H)} \otimes d} \Rep(H) \otimes \sC \xrightarrow{\act} \sC. \] The functor $\uHom_{\sC}(c,-)$ is canonically isomorphic to the composition \[ \sC \xrightarrow{\act^R} \Rep(H) \otimes \sC \xrightarrow{\id_{\Rep(H)} \otimes \Hom_{\sC}(c,-)} \Rep(H) \otimes \Vect = \Rep(H), \] which can be seen by comparing left adjoints.

On the other hand, the square \[
\begin{tikzcd}
\Rep(H) \otimes \sC \arrow{r}{\act} \arrow{d}[swap]{\Ind_H^{H \times H}  \otimes \id_{\sC}} & \sC \arrow{d}{\act^R} \\
\Rep(H) \otimes \Rep(H) \otimes \sC \arrow{r}[yshift=0.25em]{\id_{\Rep(H)} \otimes \act} & \Rep(H) \otimes \sC 
\end{tikzcd} \]
commutes up to natural isomorphism. Evaluating the composite functor
\begin{align*}
\Rep(H) &\xrightarrow{\id_{\Rep(H)} \otimes d} \Rep(H) \otimes \sC  \\
&\xrightarrow{\Ind_H^{H \times H}  \otimes \id_{\sC}} \Rep(H) \otimes \Rep(H) \otimes \sC \\
&\xrightarrow{\id_{\Rep(H)} \otimes \act} \Rep(H) \otimes \sC \\
&\xrightarrow{\id_{\Rep(H)} \otimes \Hom_{\sC}(c,-)} \Rep(H)
\end{align*}
on the trivial representation yields the representation $\Hom_{\sC}(c,\sO_H \star d)$, so combining the above identifications proves the lemma.

\end{proof}

\end{comment}

Now we deduce an explicit description of the monad $\Psi_{G,x}$ from a well-known equivariant cohomology calculation originally due to Ginzburg.

\begin{proposition}

\label{geommonad}

The monad \[ \Psi_{G,x}: \Rep(\check{G}) \longrightarrow \Sph_{G,x} \longrightarrow \Rep(\check{G}) \] is isomorphic to the monad given by the associative algebra $\Sym(\check{\mathfrak{g}}[-2])$.

\end{proposition}

\begin{proof}

The functor \[ \Rep(\check{G}) \longrightarrow \Sph_{G,x} \] obtained by renormalizing $\Sat_{G,x}^{\nv}$ is monoidal, and in particular can be viewed as a $\Rep(\check{G})$-linear functor. It follows that $\Psi_{G,x}$ is given by the associative algebra $\uEnd_{\Rep(\check{G})}(\delta_{\mathfrak{L}^+G})$. Lemma \ref{innerhomlem} yields a $\check{G}$-equivariant isomorphism \[ \uEnd_{\Rep(\check{G})}(\delta_{\mathfrak{L}^+G}) \tilde{\longrightarrow} \Hom_{\Sph_{G,x}}(\delta_{\mathfrak{L}^+G},\sO_{\check{G}} \star \delta_{\mathfrak{L}^+G}). \] Note that this is an isomorphism of associative algebras in $\Rep(\check{G})$, where the algebra structure on the right side is induced by the algebra structure on $\sO_{\check{G}}$ and composition of morphisms in $\Sph_{G,x}$.

Then apply Theorem 7.6.1 from \cite{ABG}, which supplies an isomorphism of associative algebras \[ \Hom_{\Sph_{G,x}}(\delta_1,\sO_{\check{G}} \star \delta_{\mathfrak{L}^+G}) \tilde{\longrightarrow} \Sym(\check{\mathfrak{g}}[-2]). \]

\end{proof}

Next, we show that the monad $\Psi_{\check{G},x}^{\spec}$ is abstractly isomorphic to the same associative algebra as $\Psi_G$.

\begin{proposition}

\label{specmonad}

The monad \[ \Psi_{\check{G},x}^{\spec} : \Rep(\check{G}) \longrightarrow \Sph_{\check{G},x}^{\spec} \longrightarrow \Rep(\check{G}) \] is isomorphic to the monad given by the associative algebra $\Sym(\check{\mathfrak{g}}[-2])$.

\end{proposition}

\begin{proof}

This is Corollary \ref{specpsmonad}, applied in the case $\check{P} = \check{G}$.

\end{proof}

\subsection{Another reduction step} We claim that if $\Sat_{G,x}$ induces an isomorphism \[ \End_{\Sph_{G,x}}(\delta_{\mathfrak{L}^+G}) \tilde{\longrightarrow} \End_{\Sph_{\check{G},x}^{\spec}}(\Vac_{\sO_{\check{G}}}) \] on endomorphisms of the unit objects, then it is an equivalence.

We have previously shown that if the morphism of monads (\ref{monadmap}) is an isomorphism over $x$, then $\Sat_{G,x}$ is an equivalence. Propositions \ref{geommonad} and \ref{specmonad} show that both of these monads are abstractly isomorphic to the associative algebra $\Sym(\check{\mathfrak{g}}[-2])$ in $\Rep(\check{G})$, so we must prove that the endomorphism \[ \varphi_G : \Sym(\check{\mathfrak{g}}[-2]) \longrightarrow \Sym(\check{\mathfrak{g}}[-2]) \] corresponding to (\ref{monadmap}) is an isomorphism.

It suffices to show that $\varphi_G$ is an isomorphism at the level of cohomology. We view the cohomology of the associative DG algebra $\Sym(\check{\mathfrak{g}}[-2])$ as a classical graded commutative algebra, so we need only prove that $\varphi_G$ restricts to an isomorphism on the generators $\check{\mathfrak{g}}$. On the other hand, the subspace \[ H^2(\Sym(\check{\mathfrak{g}}[-2])^{\check{G}}) \subset \check{\mathfrak{g}} \] generates the adjoint representation, so it is enough to show that $\varphi_G$ restricts to an isomorphism on the algebra of invariants $\Sym(\check{\mathfrak{g}}[-2])^{\check{G}}$.

From the proof of Proposition \ref{geommonad} we obtain a commutative square \[
\begin{tikzcd}
\Sym(\check{\mathfrak{g}}[-2])^{\check{G}} \arrow{r} \arrow{d}[sloped]{\sim} & \Sym(\check{\mathfrak{g}}[-2]) \arrow{d}[sloped]{\sim} \\
\End_{\Sph_{G,x}}(\delta_{\mathfrak{L}^+G}) \arrow{r} & \uEnd_{\Rep(\check{G})}(\delta_{\mathfrak{L}^+G}),
\end{tikzcd} \]
where the left vertical isomorphism is the composition \[ \Sym(\check{\mathfrak{g}}[-2])^{\check{G}} \tilde{\longrightarrow} \Sym(\mathfrak{g}[-2])^G \tilde{\longrightarrow} \CC_{\dR}(\bB G) \tilde{\longrightarrow} \End_{\Sph_{G,x}}(\delta_{\mathfrak{L}^+G}). \] Similarly, Proposition \ref{specmonad} yields a commutative square  \[
\begin{tikzcd}
\Sym(\check{\mathfrak{g}}[-2])^{\check{G}} \arrow{r} \arrow{d}[sloped]{\sim} & \Sym(\check{\mathfrak{g}}[-2]) \arrow{d}[sloped]{\sim} \\
\End_{\Sph_{\check{G},x}^{\spec}}(\Vac_{\sO_{\check{G}}}) \arrow{r} & \uEnd_{\Rep(\check{G})}(\Vac_{\sO_{\check{G}}}).
\end{tikzcd} \]
The claim now follows.

\subsection{Conclusion of the proof} Finally, to prove that $\Sat_{G,x}$ induces an isomorphism on endomorphisms of the unit objects, we reduce to the case $G = T$ using $\Sat_{G,B}$ as an intermediary. Consider the commutative diagram in $\FactCat$ \[
\begin{tikzcd}
\Sph_G \arrow{r}{\Sat_G} \arrow{d} & \Sph_{\check{G}}^{\spec} \arrow{d} \\
\Sph_{G,B} \arrow{r}{\Sat_{G,B}} & \Sph_{\check{G},\check{B}}^{\spec} \\
\Sph_T \arrow{u} \arrow{r}{\Sat_T} \arrow{u} & \Sph_{\check{T}}^{\spec} \arrow{u},
\end{tikzcd} \]
where $\Sph_G \to \Sph_{G,B}$ is the unique morphism of $\Sph_G$-modules in $\FactCat$ (uniqueness is a consequence of unitality), and similarly for the other vertical functors.

Taking the fiber at $x$ and passing to endomorphisms of the unit objects, we obtain a commutative diagram of associative algebras \[
\begin{tikzcd}
\End_{\Sph_{G,x}}(\delta_{\mathfrak{L}^+G}) \arrow{r} \arrow{d} & \End_{\Sph_{\check{G},x}^{\spec}}(\Vac_{\sO_{\check{G}}}) \arrow{d} \\
\End_{\Sph_{G,B,x}}(\Delta^0) \arrow{r} & \End_{\Sph_{\check{G},\check{B},x}^{\spec}}(\Vac_{\Upsilon(\check{\mathfrak{n}},\sO_{\check{G}})}) \\
\End_{\Sph_{T,x}}(\delta_{\mathfrak{L}^+T}) \arrow{u} \arrow{r} & \End_{\Sph_{\check{T},x}^{\spec}}(\Vac_{\sO_{\check{T}}}) \arrow{u}.
\end{tikzcd} \]

Let us describe the algebras in the middle row: first, we have \[ \Sym(\check{\mathfrak{t}}[-2]) \tilde{\longrightarrow} \Sym(\mathfrak{t}^*[-2]) \tilde{\longrightarrow} \CC_{\dR}(\bB T) \tilde{\longrightarrow} \End_{\Sph_{G,B,x}}(\Delta^0). \] By Corollary \ref{specpsmonad}, we have an equivalence \[ \Sph_{\check{G},\check{B},x}^{\spec} \tilde{\longrightarrow} \Sym((\check{\mathfrak{g}}/\check{\mathfrak{n}})[-2])\mod(\Rep(\check{B})) \] which sends \[ \Vac_{\Upsilon(\check{\mathfrak{n}},\sO_{\check{G}})} \mapsto \Sym((\check{\mathfrak{g}}/\check{\mathfrak{n}})[-2]). \] In particular, we have isomorphisms
\begin{align*}
\End_{\Sph_{\check{G},\check{B},x}^{\spec}}(\Vac_{\Upsilon(\check{\mathfrak{n}},\sO_{\check{G}})}) &\tilde{\longrightarrow} \End_{\Sym((\check{\mathfrak{g}}/\check{\mathfrak{n}})[-2])\mod(\Rep(\check{B}))}(\Sym((\check{\mathfrak{g}}/\check{\mathfrak{n}})[-2])) \\
&\tilde{\longrightarrow} \Sym((\check{\mathfrak{g}}/\check{\mathfrak{n}})[-2])^{\check{B}} = \Sym(\check{\mathfrak{t}}[-2]).
\end{align*}

To summarize, under these identifications the above commutative diagram of associative algebras becomes \[
\begin{tikzcd}
\Sym(\check{\mathfrak{g}}[-2])^{\check{G}} \arrow{r}{\varphi_G} \arrow{d} & \Sym(\check{\mathfrak{g}}[-2])^{\check{G}} \arrow{d} \\
\Sym(\check{\mathfrak{t}}[-2]) \arrow{r} & \Sym(\check{\mathfrak{t}}[-2]) \\
\Sym(\check{\mathfrak{t}}[-2]) \arrow{u}{\id} \arrow{r}{\varphi_T} & \Sym(\check{\mathfrak{t}}[-2]) \arrow{u}{\id}.
\end{tikzcd} \]
Here the downward vertical arrows are given by the Chevalley homomorphism \[ \Sym(\check{\mathfrak{g}}[-2])^{\check{G}} \tilde{\longrightarrow} \Sym(\check{\mathfrak{t}}[-2])^W \longrightarrow \Sym(\check{\mathfrak{t}}[-2]). \] We have previously proved that Theorem \ref{mainthm} holds in the case $G = T$, which implies that $\varphi_T$ is an isomorphism. It follows that \[ \varphi_G : \Sym(\check{\mathfrak{g}}[-2])^{\check{G}} \longrightarrow \Sym(\check{\mathfrak{g}}[-2])^{\check{G}} \] is injective at the level of cohomology, hence an isomorphism, since $\Sym(\check{\mathfrak{g}}[-2])^{\check{G}}$ has finite-dimensional cohomologies.

\section{The equivalence for a proper parabolic} \label{s:psequiv}

In this section, we deduce Theorem \ref{mainthmps} from Theorem \ref{mainthm} and the main theorem of \cite{R3}.

\subsection{Recovering spherical objects from Whittaker invariants} Consider the pairing in $\FactCat$ defined as the composite \[ \Rep(\check{G}) \otimes \DD(\mathfrak{L}^+G \backslash \mathfrak{L}G) \tilde{\longrightarrow} \DD(\Gr_G)^{\mathfrak{L}N^-,\psi} \otimes \DD(\mathfrak{L}^+G \backslash \mathfrak{L}G) \longrightarrow \DD(\mathfrak{L}G)^{\mathfrak{L}N^-,\psi}, \] where the first functor is the equivalence of Theorem \ref{csformula} and the second functor is given by convolution. By Proposition \ref{ulagenprop}, this corresponds to a morphism
\begin{equation}
\label{whitconv}
\DD(\mathfrak{L}^+G \backslash \mathfrak{L}G) \longrightarrow \Rep(\check{G}) \otimes \DD(\mathfrak{L}G)^{\mathfrak{L}N^-,\psi}
\end{equation}
in $\FactCat_{\laxuntl}$. By construction, this morphism is equivariant for the right action of $\mathfrak{L}G$.

\begin{proposition}

The morphism (\ref{whitconv}) factors through the inclusion\[ \Rep(\check{G}) \otimes \DD(\mathfrak{L}G)^{\mathfrak{L}N^-,\psi,\acc} \longrightarrow \Rep(\check{G}) \otimes \DD(\mathfrak{L}G)^{\mathfrak{L}N^-,\psi}, \] and moreover sends the unit $\delta_{\mathfrak{L}^+G}$ to the image of the regular bimodule $\sO_{\check{G}}$ under the functor
\begin{align*}
\Rep(\check{G}) \otimes \Rep(\check{G}) &\tilde{\longrightarrow} \Rep(\check{G}) \otimes \DD(\Gr_G)^{\mathfrak{L}N^-,\psi} \\
&\xrightarrow{\oblv_{\mathfrak{L}^+G}} \Rep(\check{G}) \otimes \DD(\mathfrak{L}G)^{\mathfrak{L}N^-,\psi,\acc}.
\end{align*}

\end{proposition}

\begin{proof}

For first assertion, note that (\ref{whitconv}) can also be obtained from the composite \[ \Rep(\check{G}) \otimes \DD(\mathfrak{L}^+G \backslash \mathfrak{L}G) \longrightarrow \DD(\mathfrak{L}^+G \backslash \mathfrak{L}G) \xrightarrow{\Av_!^{\mathfrak{L}N^-,\psi}} \DD(\mathfrak{L}G)^{\mathfrak{L}N^-,\psi}, \] where the first morphism is the action of $\Rep(\check{G})$ coming from $\Sat_G^{\nv}$.

For the second claim, apply the functor of $\mathfrak{L}^+G$-invariants to (\ref{whitconv}) to obtain a morphism \[ \DD(\sH_G) \longrightarrow \Rep(\check{G}) \otimes \DD(\Gr_G)^{\mathfrak{L}N^-,\psi} \tilde{\longrightarrow} \Rep(\check{G}) \otimes \Rep(\check{G}). \] By construction, this agrees with the morphism of associative algebras (\ref{sphtobimod}) and in particular preserves the monoidal unit. The claim follows.

\end{proof}

Observe that for a fixed $x \in X(k)$, the category \[ \Rep(\check{G}) \otimes \DD(\mathfrak{L}G)_x^{\mathfrak{L}N^-,\psi,\acc} \] has a canonical structure of factorization module at $x$ for the factorization category \[ \Rep(\check{G}) \otimes \DD(\Gr_G)^{\mathfrak{L}N^-,\psi} \cong \Rep(\check{G}) \otimes \Rep(\check{G}), \] compatible with the functor \[ \oblv_{\mathfrak{L}^+G} : \Rep(\check{G}) \otimes \DD(\Gr_G)_x^{\mathfrak{L}N^-,\psi} \longrightarrow \Rep(\check{G}) \otimes \DD(\mathfrak{L}G)_x^{\mathfrak{L}N^-,\psi,\acc}. \] The previous proposition implies that (\ref{whitconv}) induces a morphism \[ \DD(\mathfrak{L}^+G \backslash \mathfrak{L}G)_x \longrightarrow \sO_{\check{G}}\modfact(\Rep(\check{G}) \otimes \DD(\mathfrak{L}G)^{\mathfrak{L}N^-,\psi,\acc})_x, \] equivariant for the right action of $(\mathfrak{L}G)_x$.

If $\sC$ is a left $\DD(\mathfrak{L}G)_x$-module, we obtain a functor \[ \sC^{\mathfrak{L}^+G} \tilde{\longrightarrow} \DD(\mathfrak{L}^+G \backslash \mathfrak{L}G)_x \underset{\DD(\mathfrak{L}G)_x}{\otimes} \sC \longrightarrow \sO_{\check{G}}\modfact(\Rep(\check{G}) \otimes\DD(\mathfrak{L}G)^{\mathfrak{L}N^-,\psi,\acc})_x \underset{\DD(\mathfrak{L}G)_x}{\otimes} \sC. \] Note that since the $\DD(\mathfrak{L}G)_x$-module and factorization $\Rep(\check{G})$-module structures on the category \[ \DD(\mathfrak{L}G)_x^{\mathfrak{L}N^-,\psi,\acc} \] commute, the category \[ \sC^{\mathfrak{L}N^-,\psi,\acc} \cong \DD(\mathfrak{L}G)_x^{\mathfrak{L}N^-,\psi,\acc} \underset{\DD(\mathfrak{L}G)_x}{\otimes} \sC \] automatically inherits a factorization $\Rep(\check{G})$-module structure, and hence we can identify \[ \sO_{\check{G}}\modfact(\Rep(\check{G}) \otimes\DD(\mathfrak{L}G)^{\mathfrak{L}N^-,\psi,\acc})_x \underset{\DD(\mathfrak{L}G)_x}{\otimes} \sC \tilde{\longrightarrow} \sO_{\check{G}}\modfact(\Rep(\check{G}) \otimes \sC^{\mathfrak{L}N^-,\psi,\acc})_x. \] We will show that the resulting functor
\begin{equation}
\label{sphfromwhit}
\sC^{\mathfrak{L}^+G} \longrightarrow \sO_{\check{G}}\modfact(\Rep(\check{G}) \otimes \sC^{\mathfrak{L}N^-,\psi,\acc})_x.
\end{equation}
is ``nearly'' an equivalence.

\subsection{Tempered geometric Satake equivalence} Given a $\DD(\mathfrak{L}G)_x$-module category $\sC$, we define the \emph{anti-tempered} subcategory of $\sC^{\mathfrak{L}^+G}$ by \[ \sC^{\mathfrak{L}^+G,\antitemp} := \ker(\sC^{\mathfrak{L}^+G} \xrightarrow{\Av^{\mathfrak{L}N^-,\psi}_!} \sC^{\mathfrak{L}N^-,\psi}). \] The \emph{tempered} subcategory $\sC^{\mathfrak{L}^+G,\temp}$ is defined to be the right complement of $\sC^{\mathfrak{L}^+G,\antitemp}$. The inclusion \[ \sC^{\mathfrak{L}^+G,\temp} \longrightarrow \sC^{\mathfrak{L}^+G} \] admits a right adjoint.

\begin{proposition}

\label{tempsat}

The composite functor \[ \DD(\Gr_G)_x^{\mathfrak{L}^+G,\temp} \longrightarrow \DD(\Gr_G)_x^{\mathfrak{L}^+G} = \DD(\sH_G)_x \xrightarrow{(\ref{presatake})} \sO_{\check{G}}\modfact(\Rep(\check{G}) \otimes \Rep(\check{G}))_x \] is an equivalence.

\end{proposition}

\begin{proof}

Recall that (\ref{presatake}) is t-exact, and that by Theorem \ref{mainthm} it becomes an equivalence after renormalization. By Corollary 2.3.7 of \cite{dario-ramanujan}, the subcategory \[ \DD(\Gr_G)_x^{\mathfrak{L}^+G,\antitemp} \subset \DD(\sH_G)_x \] consists precisely of the infinitely connective objects, i.e., those with vanishing cohomology objects in all degrees. It follows that the functor \[ \DD(\sH_G)_x \longrightarrow \DD(\Gr_G)_x^{\mathfrak{L}^+G,\temp} \] right adjoint to the inclusion realizes the latter as the left completion of the t-structure on the former. On the other hand, the proof of Proposition \ref{specpsmonad} shows that we have a t-exact equivalence \[ \sO_{\check{G}}\modfact(\Rep(\check{G}) \otimes \Rep(\check{G}))_x \cong \QCoh((\{ 0 \} \underset{\check{\mathfrak{g}}}{\times} \{ 0 \})/\check{G}), \] and in particular the t-structure on the source is left complete. Thus \[ \Sph_{\check{G},x}^{\spec} = \sO_{\check{G}}\modfact(\Rep(\check{G}) \otimes \Rep(\check{G}))_x^{\ren} \longrightarrow \sO_{\check{G}}\modfact(\Rep(\check{G}) \otimes \Rep(\check{G}))_x \] realizes the target as the left completion of the source. It follows that the functor in the proposition is obtained by applying left completion to the equivalence of Theorem \ref{mainthm}, hence is an equivalence.

\end{proof}

\subsection{A formula for tempered $\fL^+G$-invariants} We will need a couple of lemmas.

\begin{lemma}

\label{whitaccgen}

The convolution functor \[ \DD(\Gr_G)^{\mathfrak{L}N^-,\psi}_x \underset{\DD(\sH_G)_x}{\otimes} \DD(\mathfrak{L}^+G \backslash \mathfrak{L}G)_x \longrightarrow \DD(\mathfrak{L}G)_x^{\mathfrak{L}N^-,\psi} \] is an equivalence onto \[ \DD(\mathfrak{L}G)_x^{\mathfrak{L}N^-,\psi,\acc}. \]

\end{lemma}

\begin{proof}

This functor is fully faithful because $\sH_{G,x}$ is ind-proper (see Theorem 3.1.5 of \cite{CD}), so it suffices to show that the image generates the target under colimits. Tensor the functor \[ \Av^{\mathfrak{L}N^-,\psi}_! : \DD(\sH_G)_x \longrightarrow \DD(\Gr_G)^{\mathfrak{L}N^-,\psi}_x \] with $\DD(\mathfrak{L}^+G \backslash \mathfrak{L}G)_x$ over $\DD(\sH_G)_x$ to obtain a functor \[ \DD(\mathfrak{L}^+G \backslash \mathfrak{L}G)_x \longrightarrow \DD(\Gr_G)^{\mathfrak{L}N^-,\psi}_x \underset{\DD(\sH_G)_x}{\otimes} \DD(\mathfrak{L}^+G \backslash \mathfrak{L}G)_x. \] Its composition with the functor in question identifies with \[ \Av^{\mathfrak{L}N^-,\psi}_! : \DD(\mathfrak{L}^+G \backslash \mathfrak{L}G)_x \longrightarrow \DD(\mathfrak{L}G)_x^{\mathfrak{L}N^-,\psi}, \] whose image generates the accessible subcategory by definition.

\end{proof}

\begin{lemma}

The functor \[ \DD(\Gr_G)_x^{\mathfrak{L}^+G,\temp} \underset{\DD(\sH_G)_x}{\otimes} \DD(\mathfrak{L}^+G \backslash\mathfrak{L}G)_x \longrightarrow \DD(\Gr_G)_x^{\mathfrak{L}^+G} \underset{\DD(\sH_G)_x}{\otimes} \DD(\mathfrak{L}^+G \backslash\mathfrak{L}G)_x = \DD(\mathfrak{L}G)_x^{\mathfrak{L}^+G} \] is an equivalence onto $\DD(\mathfrak{L}G)_x^{\mathfrak{L}^+G,\temp}$.

\end{lemma}

\begin{proof}

Recall that the inclusion \[ \DD(\Gr_G)_x^{\mathfrak{L}^+G,\temp} \longrightarrow \DD(\Gr_G)_x^{\mathfrak{L}^+G} \] admits a right adjoint, which is automatically $\DD(\sH_G)_x$-linear because $\DD(\sH_G)_x$ is compactly generated by left and right dualizable objects. It follows that the functor in question is fully faithful.

To see that the image of this functor is $\DD(\mathfrak{L}G)_x^{\mathfrak{L}^+G,\temp}$, it suffices to show that that the kernel of the right adjoint \[ \DD(\mathfrak{L}G)_x^{\mathfrak{L}^+G} \longrightarrow \DD(\Gr_G)_x^{\mathfrak{L}^+G,\temp} \underset{\DD(\sH_G)_x}{\otimes} \DD(\mathfrak{L}^+G \backslash\mathfrak{L}G)_x \] is equal to $\DD(\mathfrak{L}G)_x^{\mathfrak{L}^+G,\antitemp}$. This follows from the commutative square \[
\begin{tikzcd}
\DD(\Gr_G)_x^{\mathfrak{L}^+G} \underset{\DD(\sH_G)_x}{\otimes} \DD(\mathfrak{L}^+G \backslash\mathfrak{L}G)_x \arrow{r}{\sim} \arrow{d}{\Av_!^{\mathfrak{L}N^-,\psi} \otimes \id} & \DD(\mathfrak{L}^+G \backslash\mathfrak{L}G)_x \arrow{d}{\Av_!^{\mathfrak{L}N^-,\psi}} \\
\DD(\Gr_G)_x^{\mathfrak{L}N^-,\psi} \underset{\DD(\sH_G)_x}{\otimes} \DD(\mathfrak{L}^+G \backslash\mathfrak{L}G)_x \arrow{r}{\sim} & \DD(\mathfrak{L}G)_x^{\mathfrak{L}N^-,\psi,\acc}.
\end{tikzcd} \]

\end{proof}

Restrict (\ref{sphfromwhit}) to the tempered subcategory to obtain a functor
\begin{equation}
\label{tempsphfromwhit}
\sC^{\mathfrak{L}^+G,\temp} \longrightarrow \sO_{\check{G}}\modfact(\Rep(\check{G}) \otimes \sC^{\mathfrak{L}N^-,\psi,\acc})_x.
\end{equation}

\begin{proposition}

\label{tempsatuniv}

The functor (\ref{tempsphfromwhit}) is an equivalence for any $\DD(\mathfrak{L}G)_x$-module category $\sC$.

\end{proposition}

\begin{proof}

First, we can immediately reduce to the universal case $\sC = \DD(\mathfrak{L}G)_x$. Namely, we have \[ \sC^{\mathfrak{L}^+G,\temp} \cong \DD(\mathfrak{L}G)_x^{\mathfrak{L}^+G,\temp} \underset{\DD(\mathfrak{L}G)_x}{\otimes} \sC, \] and the functor in question is given by tensoring
\begin{equation}
\label{sphfromwhituniv}
\DD(\mathfrak{L}G)_x^{\mathfrak{L}^+G,\temp} \longrightarrow \sO_{\check{G}}\modfact(\Rep(\check{G}) \otimes \DD(\mathfrak{L}G)^{\mathfrak{L}N^-,\psi,\acc})_x
\end{equation}
with $\sC$ over $\DD(\mathfrak{L}G)_x$.

Next, we reduce to the case $\sC = \DD(\Gr_G)_x$. First, observe that
\begin{align*}
\sO_{\check{G}}\modfact(\Rep(\check{G}) \otimes & \DD(\Gr_G)^{\mathfrak{L}N^-,\psi})_x \underset{\DD(\sH_G)_x}{\otimes} \DD(\mathfrak{L}^+G \backslash\mathfrak{L}G)_x \\
&\tilde{\longrightarrow} \sO_{\check{G}}\modfact(\Rep(\check{G}) \otimes \DD(\Gr_G)^{\mathfrak{L}N^-,\psi} \underset{\DD(\sH_G)}{\otimes} \DD(\mathfrak{L}^+G \backslash\mathfrak{L}G))_x \\
&\tilde{\longrightarrow} \sO_{\check{G}}\modfact(\Rep(\check{G}) \otimes \DD(\mathfrak{L}G)^{\mathfrak{L}N^-,\psi,\acc})_x,
\end{align*}
where the first equivalence is because the right $\DD(\sH_G)_x$-module and factorization $\Rep(\check{G})$-module structures on \[ \DD(\Gr_G)^{\mathfrak{L}N^-,\psi}_x \] commute, and the second equivalence uses Lemma \ref{whitaccgen}. On the other hand, we have \[ \DD(\Gr_G)_x^{\mathfrak{L}^+G,\temp} \underset{\DD(\sH_G)_x}{\otimes} \DD(\mathfrak{L}^+G \backslash\mathfrak{L}G)_x \tilde{\longrightarrow} \DD(\mathfrak{L}G)_x^{\mathfrak{L}^+G,\temp} \] by the previous lemma. Thus if we can establish that (\ref{tempsphfromwhit}) is an equivalence for $\sC = \DD(\Gr_G)_x$, then the case $\sC = \DD(\mathfrak{L}G)_x$ will follow by tensoring with $\DD(\mathfrak{L}^+G \backslash \mathfrak{L}G)_x$ over $\DD(\sH_G)_x$.

To see that \[ \DD(\Gr_G)_x^{\mathfrak{L}^+G,\temp} \longrightarrow \sO_{\check{G}}\modfact(\Rep(\check{G}) \otimes \DD(\Gr_G)^{\mathfrak{L}N^-,\psi,\acc})_x \] is an equivalence, observe that \[ \DD(\Gr_G)^{\mathfrak{L}N^-,\psi,\acc} = \DD(\Gr_G)^{\mathfrak{L}N^-,\psi,\acc} \tilde{\longrightarrow} \Rep(\check{G}) \] by Theorem \ref{csformula}. The resulting functor \[ \DD(\Gr_G)_x^{\mathfrak{L}^+G,\temp} \longrightarrow \sO_{\check{G}}\modfact(\Rep(\check{G}) \otimes \Rep(\check{G}))_x \] is an equivalence by Proposition \ref{tempsat}.

\end{proof}

\subsection{Proof of Theorem \ref{mainthmps}} As shown in \S \ref{pointred}, it is enough to prove that $\Sat_{G,P,x}$ is an equivalence for a fixed $x \in X(k)$.

Apply Proposition \ref{tempsatuniv} in the case $\sC = (\DD(\mathfrak{L}G)_{\mathfrak{L}N_P\mathfrak{L}^+M})_x$, which yields the equivalence \[ (\DD(\mathfrak{L}G)_{\mathfrak{L}N_P\mathfrak{L}^+M})_x^{\mathfrak{L}^+G,\temp} \tilde{\longrightarrow} \sO_{\check{G}}\modfact(\Rep(\check{G}) \otimes (\DD(\mathfrak{L}G)_{\mathfrak{L}N_P\mathfrak{L}^+M})^{\mathfrak{L}N^-,\psi,\acc})_x. \] Theorem \ref{semiinfresequiv} supplies an isomorphism \[ (\DD(\mathfrak{L}G)_{\mathfrak{L}N_P\mathfrak{L}^+M})^{\mathfrak{L}N^-,\psi,\acc} \tilde{\longrightarrow} \Upsilon(\check{\mathfrak{n}}_{\check{P}})\mod_0^{\fact}(\Rep(\check{M})) \] in $\FactCat$, and the commutative triangle in the proof of Theorem \ref{psfunctorunit} yields a commutative triangle \[
\begin{tikzcd}
& \Rep(\check{G}) \arrow{dl} \arrow{dr}{\Upsilon(\check{\mathfrak{n}}_{\check{P}},-)} & \\
(\DD(\mathfrak{L}G)_{\mathfrak{L}N_P\mathfrak{L}^+M})^{\mathfrak{L}N^-,\psi,\acc} \arrow{rr}{\sim} & & \Upsilon(\check{\mathfrak{n}}_{\check{P}})\mod_0^{\fact}(\Rep(\check{M})).
\end{tikzcd} \]
Here the left diagonal morphism is the composite \[ \Rep(\check{G}) \longrightarrow \DD(\mathfrak{L}^+G \backslash \mathfrak{L}G)_{\mathfrak{L}N_P\mathfrak{L}^+M} \xrightarrow{\Av_!^{\mathfrak{L}N^-,\psi}} (\DD(\mathfrak{L}G)_{\mathfrak{L}N_P\mathfrak{L}^+M})^{\mathfrak{L}N^-,\psi,\acc}, \] where the first morphism is given by the acting on the unit via $\Sat_G^{\nv}$. We therefore obtain an equivalence \[ (\DD(\mathfrak{L}G)_{\mathfrak{L}N_P\mathfrak{L}^+M})_x^{\mathfrak{L}^+G,\temp} \tilde{\longrightarrow} \Upsilon(\check{\mathfrak{n}}_{\check{P}},\sO_{\check{G}})\modfact(\Rep(\check{G}) \otimes \Upsilon(\check{\mathfrak{n}})\mod_0^{\fact}(\Rep(\check{M})))_x, \] and by definition the latter category is equivalent to \[ \Upsilon(\check{\mathfrak{n}}_{\check{P}},\sO_{\check{G}})\mod_0^{\fact}(\Rep(\check{G}) \otimes \Rep(\check{M}))_x. \]

Equip the category \[ (\DD(\mathfrak{L}G)_{\mathfrak{L}N_P\mathfrak{L}^+M})_x^{\mathfrak{L}^+G,\temp} \] with the t-structure characterized by the requirement that the functor
\begin{equation}
\label{tempquotps}
(\DD(\mathfrak{L}G)_{\mathfrak{L}N_P\mathfrak{L}^+M})_x^{\mathfrak{L}^+G} \longrightarrow (\DD(\mathfrak{L}G)_{\mathfrak{L}N_P\mathfrak{L}^+M})_x^{\mathfrak{L}^+G,\temp}
\end{equation}
right adjoint to the inclusion is left t-exact. In fact, this functor is t-exact: its kernel \[ (\DD(\mathfrak{L}G)_{\mathfrak{L}N_P\mathfrak{L}^+M})_x^{\mathfrak{L}^+G,\antitemp} \] equals the kernel of $\Av^{\mathfrak{L}N^-,\psi}_!$ by definition, and by Proposition \ref{whitavtexact} the latter consists of infinitely connective objects, i.e., those with vanishing cohomology objects in all degrees. In particular this kernel is stable under truncation functors, which implies that (\ref{tempquotps}) is t-exact as claimed.

Combining the above equivalences, we obtain
\begin{equation}
\label{temppsequiv}
(\DD(\mathfrak{L}G)_{\mathfrak{L}N_P\mathfrak{L}^+M})_x^{\mathfrak{L}^+G,\temp} \tilde{\longrightarrow} \Upsilon(\check{\mathfrak{n}}_{\check{P}},\sO_{\check{G}})\mod_0^{\fact}(\Rep(\check{G}) \otimes \Rep(\check{M}))_x.
\end{equation}
Tracing through the constructions, we see that this is the restriction of (\ref{psfunc1}) to the tempered subcategory. We claim that the triangle \[
\begin{tikzcd}
& (\DD(\mathfrak{L}G)_{\mathfrak{L}N_P\mathfrak{L}^+M})_x^{\mathfrak{L}^+G} \arrow{dl}[swap]{(\ref{tempquotps})} \arrow{dr}{(\ref{psfunc1})} & \\
(\DD(\mathfrak{L}G)_{\mathfrak{L}N_P\mathfrak{L}^+M})_x^{\mathfrak{L}^+G,\temp} \arrow{rr}{(\ref{temppsequiv})} & & \Upsilon(\check{\mathfrak{n}}_{\check{P}},\sO_{\check{G}})\mod_0^{\fact}(\Rep(\check{G}) \otimes \Rep(\check{M}))_x
\end{tikzcd} \]
commutes, which amounts to the assertion that (\ref{psfunc1}) vanishes on \[ (\DD(\mathfrak{L}G)_{\mathfrak{L}N_P\mathfrak{L}^+M})_x^{\mathfrak{L}^+G,\antitemp}. \] As noted above, this subcategory consists of infinite connective objects, and the t-structure on \[ \Upsilon(\check{\mathfrak{n}}_{\check{P}},\sO_{\check{G}})\mod_0^{\fact}(\Rep(\check{G}) \otimes \Rep(\check{M}))_x \cong \QCoh((\check{\mathfrak{n}}_{\check{P}} \underset{\check{\mathfrak{g}}}{\times} \{ 0 \})/\check{P}) \] is left complete, so the claim follows.

It therefore suffices to show that (\ref{tempquotps}) restricts to an equivalence on eventually coconnective objects, or equivalently on coconnective objects. The restriction of this functor to coconnective objects admits the left adjoint \[ ((\DD(\mathfrak{L}G)_{\mathfrak{L}N_P\mathfrak{L}^+M})_x^{\mathfrak{L}^+G,\temp})^{\geq 0} \longrightarrow (\DD(\mathfrak{L}G)_{\mathfrak{L}N_P\mathfrak{L}^+M})_x^{\mathfrak{L}^+G} \xrightarrow{\tau^{\geq 0}} ((\DD(\mathfrak{L}G)_{\mathfrak{L}N_P\mathfrak{L}^+M})_x^{\mathfrak{L}^+G})^{\geq 0}, \] which is fully faithful because (\ref{tempquotps}) is t-exact. So it is enough to show that (\ref{tempquotps}) is conservative on eventually coconnective objects. But \[ \Av^{\mathfrak{L}N^-,\psi}_! : (\DD(\mathfrak{L}G)_{\mathfrak{L}N_P\mathfrak{L}^+M})_x^{\mathfrak{L}^+G} \longrightarrow (\DD(\mathfrak{L}G)_{\mathfrak{L}N_P\mathfrak{L}^+M})_x^{\mathfrak{L}N^-,\psi} \] factors through (\ref{tempquotps}), and the former is conservative on eventually coconnective objects by Proposition \ref{whitavtexact}.

\bibliography{bibtex.bib}{}
\bibliographystyle{alpha}

%\begin{thebibliography}{10}
%
%\bibitem{ABG}
%
%S. Arkhipov, R. Bezrukavnikov, and V. Ginzburg, Quantum groups, the loop Grassmannian, and the Springer resolution. \emph{J. Amer. Math. Soc.} \textbf{17}, 595-678 (2004).
%
%\bibitem{BNP}
%
%D. Ben-Zvi, D. Nadler, and A. Pregel, Integral transforms for coherent sheaves. Available at \url{https://arxiv.org/pdf/1312.7164.pdf}.
%
%\bibitem{B}
%
%D. Beraldo, Loop group actions on categories and Whittaker invariants. \emph{Advances in Mathematics} \textbf{322}, 565-636 (2017).
%
%\bibitem{FG}
%
%J. Francis and D. Gaitsgory, Chiral Koszul Duality. Available at \url{https://arxiv.org/pdf/1103.5803.pdf}.
%
%\bibitem{G}
%
%D. Gaitsgory, The semi-infinite intersection cohomology sheaf II: the Ran space version. Available at \url{https://people.math.harvard.edu/~gaitsgde/GL/SemiinfRan.pdf}.
%
%\bibitem{R1}
%
%S. Raskin, Chiral categories. Available at \url{https://web.ma.utexas.edu/users/sraskin/chiralcats.pdf}.
%
%\bibitem{R2}
%
%S. Raskin, Chiral principal series categories I: Finite-dimensional calculations. Available at \url{https://web.ma.utexas.edu/users/sraskin/cpsi.pdf}.
%
%\bibitem{R3}
%
%S. Raskin, Chiral principal series categories II: The factorizable Whittaker category. Available at \url{https://web.ma.utexas.edu/users/sraskin/cpsii.pdf}.
%
%\bibitem{R4}
%
%S. Raskin, On the notion of spectral decomposition in local geometric Langlands. Available at \url{https://arxiv.org/pdf/1511.01378.pdf}.
%
%\bibitem{R5}
%
%S. Raskin, Homological methods in semi-infinite contexts. Available at \url{https://web.ma.utexas.edu/users/sraskin/topalg.pdf}.
%
%\bibitem{RY}
%
%S. Raskin and D. Yang, Affine Beilinson-Bernstein localization at the critical level. Available at \url{https://arxiv.org/pdf/2203.13885.pdf}.
%
%\end{thebibliography}

\end{document}